\documentclass{amsart}
\usepackage{ADSS}
\usepackage[hyperindex]{hyperref}
\title[Supercuspidal characters of $\protect\SL_2$]
{Supercuspidal characters of $\protect\SL_2$ over a $p$-adic field}
\author[Adler]{Jeffrey D. Adler}
\address{Department of Mathematics and Statistics \\
The American University \\
4400 Massachusetts Ave NW \\
Washington, DC  20016-8050}
\email{jadler@american.edu}
\author[DeBacker]{Stephen DeBacker}
\address{Department of Mathematics \\
University of Michigan \\
530 Church St \\
2074 East Hall \\
Ann Arbor, MI 48109-1043}
\email{smdbackr@umich.edu}
\author[Sally]{Paul J.~Sally, Jr.}
\address{Department of Mathematics\\
The University of Chicago \\
5734 S.\ University Ave \\
Chicago, IL  60637}
\email{sally@math.uchicago.edu}
\author[Spice]{Loren Spice}
\address{Department of Mathematics \\
Texas Christian University \\
TCU Box 298900 \\
2840 W.\ Bowie St \\
Fort Worth, TX 76109}
\email{lspice@tcu.edu}
\thanks{%
The first-named author was partially supported
by NSF grant DMS-0854844.
The second- and last-named authors were partially supported
by NSF grant DMS-0854897.%
}
\begin{document}

\dedicatory{Dedicated to the memory of Joseph Shalika}

\begin{abstract}
The character formulas of \cite{sally-shalika:characters}
are an early triumph in \(p\)-adic harmonic analysis, but,
to date, the calculations underlying the formulas have not
been available.  In this paper, which should be viewed as a
precursor of the forthcoming volume
\cite{adler-debacker-roche-sally-spice:sl2}, we
leverage modern technology (for example, the Moy--Prasad
theory)
to compute explicit character tables.
An interesting highlight is the computation of the
`exceptional' supercuspidal characters, i.e., those
depth-zero representations not arising by
inflation--induction from a Deligne--Lusztig representation
of finite \(\SL_2\);
this provides a concrete application for the recent work of
DeBacker and Kazhdan \cite{debacker-kazhdan:mk-zero}.
\end{abstract}
\maketitle

\setcounter{tocdepth}1
\tableofcontents

\section{Introduction}

\subsection{History}
\label{sec:history}

Supercuspidal representations of reductive $p$-adic groups
were discovered by F.~Mautner in the late 1950s.  In fact,
one of us (Sally) heard him lecture on his discovery at
Brandeis in 1959.  His construction is contained in the
following theorem, which appeared in the
\textit{American Journal of Mathematics} in 1964.
The notation is explicated in the body of the paper.
\begin{thm}[%
	\cite{mautner:spherical-2}*{Theorem 9.1}%
]
Let $G=\PGL_2(\field)$ and $K=\PGL_2(\pint)$,
where \(\field\) is a \(p\)-adic field
and \(\pint\) its ring of integers,
and write
\[
N =  \set{\begin{smallpmatrix} 1 & 0 \\ x & 1
\end{smallpmatrix}}{ x \in \pint} .
\]
Let $u$ be an
irreducible, unitary representation of $K$ whose restriction
to $N$ does not contain the trivial representation.  Let
$U = \Ind_K^G u$ be the
induced representation (compact induction).  Then $U$ is
the direct sum of a finite number of irreducible, unitary
representations $U^{(j)}$ of $G$.  In a suitable orthonormal
basis, the matrix coefficients of each $U^{(j)}$ are
functions of compact support on $G$.
\end{thm}

In 1963, I.~M.~Gel'fand and M.~I.~Graev published a paper in
\textit{Uspekhi} \cite{gelfand-graev:sl2} in which they studied
the representation theory of $p$-adic $\SL_2$.   Their
methods and presentation hewed closely to those used in the
study of the discrete series of real $\SL_2$, and
did not make use of induction from compact, open subgroups.
Their realization of the discrete series can
be directly compared to the representation constructed by
A.~Weil \cite{weil:weil}, as shown by S. Tanaka in
\cite{tanaka:sl2}.
Gel'fand and Graev presented formulas for the
discrete-series characters
\cite{gelfand-graev:sl2}*{\S\S5.3, 5.4}
and recovered the Plancherel formula
(\S6.2 \loccit),
but they made some errors concerning irreducibility
of the discrete series and unitary equivalence.

In his thesis \cite{shalika:thesis}*{\S\S 1.5--1.9}, J.~A.~Shalika
constructed the supercuspidal representations of $p$-adic
$\SL_2$ by using the Weil representation.  He then proved
their irreducibility and identified the equivalences among them
in complete detail.
In addition, he showed
(\S\S3, 4 \loccit)
that each supercuspidal
representation could be induced from a maximal compact
subgroup by restricting,
picking out an irreducible
component, and inducing back up.

The existence of a
Frobenius-type inducing formula for supercuspidal characters
was shown by T.~Shintani \cite{shintani:sq-int}*{Theorem 3},
who worked with a group of square matrices over a
\(p\)-adic field whose determinant is a unit in the ring of
integers.
In 1968, Sally and Shalika used such a formula
(see \cite{sally:remarks}*{Theorem 1.9})
to compute the irreducible characters of the supercuspidal
representations of $\SL_2$
\cite{sally-shalika:characters}.
Their formulas have a sign error in a few cases,
but their later papers
\cites{sally-shalika:plancherel,sally-shalika:orbital-integrals}
are not affected.  See Remark \ref{rem:far} for details.

There were several additional expositions related to the
discrete-series characters of rank-one groups over a
\(p\)-adic field.
\begin{itemize}
\item
A.~Silberger \cite{silberger:pgl2} computed characters for
$\PGL_2$ by a type of limit formula;
\item
H.~Jacquet and R.~P.~Langlands \cite{jacquet-langlands}
used the information from
Sally--Shalika to analyze supercuspidal characters for
$\GL_2$;
\item
H.~Shimizu \cite{shimizu:gl2} published some
character computations for $\GL_2$;
and
\item
Sally gave an
example of the inducing construction for $\SL_2$ in
\cite{sally:characters}.
\end{itemize}

\subsection{Outline}

The aim of this paper is to provide a complete guide for
computing the supercuspidal characters of $\SL_2$ over a
\(p\)-adic field.
As discussed in \S\ref{sec:history},
these characters have been
available in some form at least since the
1960s~\cites{gelfand-graev:sl2,sally-shalika:characters}.
Over the intervening half-century, significant advances have
been made in our understanding of both reductive $p$-adic
groups and their representation theory.  A goal of this
paper is to bring these modern, general tools to bear on
the problem of explicit character computation.

The formulas in \cite{sally-shalika:characters} were found
by using \cite{shalika:thesis}*{Theorem 3.1.2} to recognize
the supercuspidal representations of \(G = \SL_2(\field)\)
as induced (in the sense of \S3.1 \loccit) from
(finite-dimensional) representations of maximal compact
subgroups of \(G\), and then employing a \(p\)-adic analogue
of the Frobenius formula \cite{sally:remarks}*{Theorem~1.9}.
Broadly speaking, this paper follows the same path.
After establishing some basic notation
(\S\S\ref{sec:sqrt}--\ref{sec:dualityfour}),
we discuss how to construct all supercuspidal representations of $G$
(\S\S\ref{sec:ktypes}--\ref{sec:inducing}),
and then finish by computing the characters of these representations
(\S\S\ref{sec:MK}--\ref{sec:exceptional}).

We begin by establishing some basic facts and notation
about fields (\S\ref{sec:sqrt}) and tori (\S\ref{sec:tori}).
After examining a certain principal-value integral
that will appear in the character formulas
(\S\ref{sec:SSPhi}), we  discuss in \S\ref{sec:building}
the pioneering work of Bruhat--Tits \cites{
	bruhat-tits:reductive-groups-1,
	bruhat-tits:reductive-groups-2
} and Moy--Prasad \cites{
	moy-prasad:k-types,
	moy-prasad:jacquet
}.  Bruhat--Tits theory underlies nearly everything that we do; indeed,
without it, we would not have the language in
which to state many of our results.
In order not to require of the reader familiarity with the
general notions of Bruhat--Tits theory, we specialize them to \(\SL_2\),
where the group filtrations can be described very concretely
(see \S\ref{sec:filt}),
and related to filtrations of tori
(see \S\ref{sec:gp-torus})
and, via the Cayley map, of the Lie algebra
(see \S\ref{sec:cayley}).
The Cayley map has many of the properties of the exponential
map (see Lemma \ref{lem:cayley}),
but can converge on a larger domain.

After summarizing our choices for normalization of measures
(\S\ref{sec:measure})%
---we will usually use Waldspurger's normalization,
adapted to the structure theory of Bruhat--Tits and
Moy--Prasad---%
and discussing the Fourier transform
(\S\ref{sec:dualityfour}), 
we turn our attention to the problem of classifying all
supercuspidal representations of $G$.
We do this via the theory of types,
reviewed in \S\ref{sec:ktypes}.
An unrefined minimal $K$-type is a certain pair $(K,\xi)$
consisting of a compact, open subgroup \(K\)
and a representation \(\xi\) of it.
Every irreducible representation of a $p$-adic group
contains an unrefined minimal $K$-type,
unique up to a natural relation
(see Definition \ref{defn:k-type}).

For $\SL_2$,  the unrefined minimal $K$-type contained
in a representation
is sufficient to determine whether that representation is supercuspidal;
we list all those that can occur
in a supercuspidal representation.
The final task is to determine from a given unrefined minimal $K$-type
\emph{all} possible supercuspidal representations that contain it.
For depth-zero, unrefined minimal $K$-types, the above plan
is carried out in \S\ref{sec:depthzero};
and for positive depth representations it is carried out in
\S\ref{sec:posdepth}.

The calculations of~\cite{sally-shalika:characters} were long, involved, and, since the state
of the art in structure theory of \(p\)-adic groups then
(1968) was not nearly so advanced as it is now, somewhat
\textit{ad hoc}; and, perhaps most importantly, they have
never appeared in print.  In the final sections of this paper, we justify the
calculations of \cite{sally-shalika:characters}, making use
of modern technology whenever it simplifies matters.  Two
particularly powerful references that are available to us
are \cite{debacker-reeder:depth-zero-sc}, which handles all
but four of the depth-zero, supercuspidal representations;
and \cite{adler-spice:explicit-chars}, which handles all
positive-depth supercuspidal representations.
The remaining four supercuspidals, which we call
`exceptional' (see \S\ref{sec:exceptional}), require a bit
more care; but, even in this case, most of the necessary work
has already been done,
by Waldspurger \cite{waldspurger:nilpotent}
and DeBacker--Kazhdan \cite{debacker-kazhdan:mk-zero},
and there remains only the (non-trivial!\@) task of
specializing this work to the case of \(\SL_2\).
See \S\ref{sec:characters} for a summary of the results.

In future
work \cite{adler-debacker-roche-sally-spice:sl2}, the
present authors, together with Alan Roche, will continue
this program to present a complete picture of harmonic
analysis on \(p\)-adic \(\SL_2\).
One of our goals will be to make more accessible some of the
general tools that have been developed over the last fifteen
years by specializing them to the case of \(\SL_2\).
Thus, rather than citing major theorems (such as the
Bernstein decomposition theorem, or the main theorems of Moy--Prasad),
we will prove them
in this case wherever doing so has illustrative value.
We will construct all irreducible representations and compute
their characters.
We will treat the principal series in an old-fashioned way,
via intertwining operators, and also via the theory of types.
We will construct the unitary, smooth, and tempered duals,
and describe the discrete series.
We will compute the Fourier transforms of nilpotent orbital
integrals on the Lie algebra, descibe the local character expansions
of all irreducible representations,
and compute the Plancherel measure.

\subsection{General notation}
\label{sec:notn}

If \(S\) is a ring, then we denote by \indexmem{S\mult}
the group of units in \(S\).

Suppose that \indexmem\field is a non-discrete, non-Archimedean local field
with normalized valuation \indexmem\ord
(i.e., \(\ord(\field) = \Z \cup \sset{{+}\infty}\)).
Let \indexmem\pint denote the ring of integers in $\field$
and \indexmem\pp the prime ideal of $\pint$.
Fix an element $\indexmem\epsilon \in \pint\mult\setminus (\pint\mult)^2$
and a uniformizer $\indexmem\varpi \in \pint$.

Let $\resfld$ denote the residue field $\pint/\pp$ of $\field$.
Then the image in \(\resfld\mult\) of \(\epsilon\) is a
non-square in \(\resfld\mult\).
We write \(\indexmem p = \operatorname{char}(\resfld)\)
and \(\indexmem q = \card\resfld\), and assume
throughout that \(p \ne 2\).

\begin{defn}
If \(\AddChar\) is an (additive) character of \(\field\)
and \(b \in \field\), then write \(\AddChar_b\) for the
additive character given by \(\abmapto t{\AddChar(b t)}\).
If \(\AddChar\) is non-trivial, then the
\term[depth (additive character)]{depth}
\(\depth(\AddChar)\) of \(\AddChar\) is the smallest index
\(r \in \R\) such that \(\AddChar\) is trivial on
\(\pp^{\rdown r + 1}\).
(If \(\AddChar\) \emph{is} trivial, then we may define
\(\depth(\AddChar) = {-}\infty\).)
\end{defn}

Note that, for \(\AddChar\) an (additive) character of
\(\field\) and \(b \in \field\), we have
\begin{equation}
\label{eq:depth-Phi-b}
\depth(\AddChar_b) = \depth(\AddChar) - \ord(b).
\end{equation}

We fix, for the remainder of the paper, an additive
character \(\AddChar\) of depth \(0\).
Explicitly, \(\AddChar\) is trivial on \(\pp\),
but not on \(\pint\).
We will use boldface letters to denote algebraic groups,
boldface Fraktur letters to denote their Lie algebras,
and
the corresponding regular letters to denote their groups of
rational points.
For example, \(T = \bT(\field)\) and \(\ttt = \Lie(T)\).

Put \(\indexmem\bG = \SL_2\).
Thus, by our convention,
\(\indexmem G = \bG(\field) = \SL_2(\field)\)
is the subgroup of all determinant-one matrices
in the group \(\indexmem{\GL_2(\field)}\)
of invertible \(2\times2\) matrices,
and $\indexmem\gg = \indexmem{\sl_2(\field)}$ is
the subalgebra of trace-zero matrices in the 
the Lie algebra \(\indexmem{\gl_2(\field)}\)
of \(2\times2\) matrices over \(\field\).

When we are dealing with complicated exponents,
we will sometimes write
\indexmem{\exp_t(s)} instead of \(t^s\),
for \(t \in \R_{> 0}\) and \(s \in \C\).

As mentioned, our calculations use rather general results in
\(p\)-adic harmonic analysis,
which, in most cases, have been proven only subject to
some restrictions.
We discuss those restrictions now.

Since
\begin{itemize}
\item
\bG is split, hence tame;
\item
\bG, which is its own derived group, is simply
connected;
and
\item
the only bad prime for \(\SL_2\) (in the sense of
\cite{adler-spice:good-expansions}*{%
	Definition \xref{exp-defn:bad-prime}%
}) is \(2\),
\end{itemize}
we have by
\cite{adler-spice:good-expansions}*{%
	Remark \xref{exp-rem:when-hyps-hold}%
} that all the hypotheses of \S\xref{exp-sec:assumptions}
\loccit hold.
We shall demonstrate explicitly that
\cite{adler-spice:explicit-chars}*{%
	Hypothesis \xref{char-hyp:X*-central}%
} holds;
see Notations \ref{notn:depth-0-X}
and \ref{notn:pos-depth-X}.

The next hypothesis is only needed when we cite
\cite{debacker-reeder:depth-zero-sc}*{Lemma 12.4.3}
(our Lemma \ref{lem:MK-green}),
i.e., in the depth-zero cases of Proposition \ref{prop:MK}
and Theorems \ref{thm:near} and \ref{thm:exc-near}.
Although it is possible to remove this restriction in our
setting, we have not done so here.

\begin{hyp}[%
	\cite{debacker-reeder:depth-zero-sc}*{%
		Restriction 12.4.1(2)%
	}%
]
\label{hyp:debacker-reeder:depth-zero-sc:res-12.4.1(2)}
The characteristic of \(\field\) is \(0\),
and the residual characteristic satisfies
\(p \ge 2e + 3\), where \(e\) is the absolute ramification
degree of \(\field\) (i.e., its ramification degree over
\(\Q_p\)).
\end{hyp}

\subsection{Character formulas}
\label{sec:characters}
In this section, we summarize the character values computed in this paper.
The formulas which occur use a large amount of notation that has not been defined
yet; it is described in detail in
\S\S\ref{sec:tori}--\ref{sec:dualityfour}.

We adopt the parametrization of supercuspidal representations presented in Theorem~\ref{thm:sc}.  By Remarks \ref{rem:depth-0-leftist}
and \ref{rem:pos-depth-leftist},
we may, and do, restrict our attention to tori of the form
 \(T^\theta = T^{\theta, 1}\)
for some
\(\theta \in \sset{\epsilon, \varpi, \epsilon\varpi}\). 

As described in \cite{debacker-sally:germs},
computations of characters of \(p\)-adic groups have indicated that,
broadly speaking,
they have a `geometric part' near the identity,
where they are described in terms of functions associated to
(co)adjoint orbits
\cites{
murnaghan:chars-u3,
murnaghan:chars-sln,
murnaghan:chars-classical,
murnaghan:chars-gln,
cunningham:ell-exp,
adler-debacker:mk-theory,
cunningham-gordon:sl2-motivic,
jkim-murnaghan:charexp,
jkim-murnaghan:gamma-asymptotic
},
and an `arithmetic part' far from the identity,
where they are described by some analogue of Weyl's
classical character formula.
(Actually, \cite{adler-spice:explicit-chars}*{%
	Theorem \xref{char-thm:full-char}%
} shows that one should usually expect \emph{mixed}
arithmetic--geometric behavior, even for supercuspidal
characters; but there is a clean separation in the case of
\(\SL_2\).)
To make precise the notion of being near or far from the
identity, we use Moy--Prasad's notion of depth; see
Definitions \ref{defn:depth-element}
	and \ref{defn:depth-rep}.
Specifically, the geometric part of the character of a
representation \(\pi\) applies to those elements \(\gamma\)
such that \(\depth(\gamma) > \depth(\pi)\),
whereas the arithmetic part applies to those elements
\(\gamma\) such that \(\depth(\gamma) < \depth(\pi)\).
In the intermediate range,
where \(\depth(\gamma) = \depth(\pi)\),
the character exhibits qualitatively different behavior,
related to special functions on \(p\)-adic and finite fields
(see \cite{spice:sl2-mu-hat}*{\S\xref{sl2-sec:Bessel}}).
We call this range the `bad shell'; the terminology `shell'
comes from the fact that the depth is the analogue of the
valuation on a \(p\)-adic field, so that the set of elements
of fixed depth may be thought of as an analogue of the set
difference of two \(p\)-adic balls.

All of the available quantitative information about
supercuspidal characters comes from the evaluation of
Harish-Chandra's integral formula
\eqref{eq:hc:harmonic:thm-12}.
The integral here is taken over the full group \(G\),
which is far too large to handle directly; so the main focus
in evaluating it is on finding many sub-integrals that equal \(0\).
The remaining terms can then often be related to
calculations on a finite field, or a finite group of Lie
type.
Few details of this part of the character computation are
included in the present paper;
we refer instead to \cites{
	debacker-reeder:depth-zero-sc,
	adler-spice:explicit-chars
}, whose general results we take for granted.
The challenge is to interpret these general results as
explicitly as possible for the special case of \(\SL_2\).

\cite{adler-spice:explicit-chars}*{%
	Proposition \xref{char-prop:induction1}%
} defines a crucial ingredient in the Weyl-sum-type formula
that gives the arithmetic part of a supercuspidal character;
it is a fourth root of unity called a \textit{Gauss sum}
(see \S\xref{char-ssec:gauss} \loccit).
The explicit formula of
Proposition \xref{char-prop:gauss-sum} \loccit
describes this fourth root of unity in terms of
the Galois action on (absolute) roots.
The most technically demanding part of our
positive-depth character computations is probably
the specialization of this explicit description to our
setting;
see \S\ref{sec:roots}.
With this in place, we compute the order of a
certain coset space, which turns out to be a Weyl
discriminant (see Lemma \ref{lem:disc-as-index}),
to complete our explicit description of the arithmetic part
of an `ordinary' (positive-depth) supercuspidal character,
in the sense of Definition \ref{defn:pos-depth-param}.

From the point of view of this paper, the geometric part of
the character is a combination of Fourier transforms of
semisimple orbital integrals; this is Murnaghan's version of
Kirillov theory (see \S\ref{sec:MK}).
In the `ordinary' case
(see Definitions \ref{defn:depth-0-param} and
\ref{defn:pos-depth-param}),
there is only one orbital integral involved;
but, in the `exceptional' case, the situation is more
complicated.  See \eqref{eq:pi-as-green-no-really}.
Thus, once we have identified the coefficients occurring in
the combination (in particular, the formal degree; see
Lemma \ref{lem:deg-pi}), we recall the results of
\cite{spice:sl2-mu-hat}*{Theorem \xref{sl2-thm:uniform}}
on semisimple orbital integrals to complete the explicit
description of the geometric part of the character formula.

We summarize all this below; but, for the sake of brevity,
we take some shortcuts.
In this section, the letter $\gamma$ always stands for a
regular, semisimple element of $G$.
That is, $\gamma$ is always a noncentral semisimple element of $G$.
Second, when we write, for example, 
$$
\Theta_{\pi^\pm} (\gamma) =
\frac1 2\Bigl\{
	\frac1{\abs{D_G(\gamma)}^{1/2}} - 1
\Bigr\},
\quad\gamma \in A_{\MPlus0},
$$
we are really giving the character value on any
\(G\)-conjugate of an element of \(A_{\MPlus0}\).
In this way, we describe the characters `shell by shell'.
Further, since the central character of
\(\pi^\pm(T, \psi)\) or \(\pi(T, \psi)\)
is \(\res\psi to{Z(G)}\),
we have that
\(\Theta_{\pi^\pm}(z \gamma) = \psi(z) \Theta_\pi(\gamma) \)
for all $\gamma$ and all $z \in Z(G)$, and similarly for $\Theta_{\pi(T,\psi)}(z\gamma)$.
That is, the formula above really gives the character value
on any
\(G\)-conjugate of an element of \(Z(G) A_{\MPlus0}\).
Thus, the term `otherwise' in the character formulas below should
be understood to mean, not just (for example)
that \(\gamma \not\in A_{\MPlus0}\),
but in fact that
\(\gamma \not\in Z(G)\dotm\Int(G)A_{\MPlus0}\).

From Theorem~\ref{thm:far},  Theorem~\ref{thm:pos-bad-un}, and Theorem~\ref{thm:near},
in the unramified case (see Definition \ref{defn:lots}),
a supercuspidal representation $\pi = \pi(T^\epsilon, \psi)$ of depth $r$ 
has character
\[
\Theta_{\pi}(\gamma)
= 
\begin{cases}
  \displaystyle \frac1 2\sgn_\epsilon\bigl(\Im_\epsilon(\gamma)\bigr)
\frac{
	\psi(\gamma) + \psi(\gamma\inv)
}{
	\smabs{D_G(\gamma)}^{1/2}
}
\bigl[(-1)^{r + 1} + H(\SSAddChar, \field_\epsilon)\bigr]
& \quad{\gamma \in T^\epsilon  \setminus Z(G) T^\epsilon_{\MPlus r} }\\
 \displaystyle c_0(\pi) + H(\SSAddChar, \field_\epsilon)\frac{
	\sgn_\epsilon\bigl(\eta\inv\Im_\epsilon(\gamma)\bigr)
}{
	\abs{D_G(\gamma)}^{1/2}
}
& \quad{\gamma \in T_{\MPlus r}^{\epsilon,\eta}}\\
 \displaystyle c_0(\pi) +
\frac1{\abs{D_G(\gamma)}^{1/2}}
& \quad{ \gamma \in A_{\MPlus r}}\\
c_0(\pi) & \quad \text{otherwise, if \(\gamma \in G_{\MPlus r}\)} \\
0 & \quad \text{otherwise, if \(\gamma \notin G_{\MPlus r}\).}
\end{cases}
\]
Here $\eta \in \{1, \varpi\}$, and \(c_0(\pi) = -q^r\).

From Theorem~\ref{thm:far} (along with Lemma~\ref{lem:S-psi}),
Theorem~\ref{thm:pos-bad-ram},
and Theorem~\ref{thm:near},
for a ramified supercuspidal representation $\pi = \pi(T^\varpi, \psi)$ of depth $r$
we have 
\newcommand{\formula}[1]{#1 &}
\newcommand{\condition}[1]{#1 \\[1.5em]}
\newcommand\lastcondition[1]{#1}
\newcommand{\longformula}[1]{\multicolumn{2}{l}{#1} \\}
\newcommand{\longcondition}[1]{& #1 \\[.6em]}
\setlength{\extrarowheight}{.4em}
\[
\Theta_{\pi}(\gamma)
= 
\left\{
\begin{array}{ll}
\longformula{  \displaystyle
	\dfrac{\sgn_\varpi\bigl(\Im_\varpi(\gamma)\bigr)H(\SSAddChar, \field_\varpi)}%
 	{ \smabs{D_G(\gamma)}^{1/2}}
	\Bigl\{
		\psi(\gamma) \, \, + \, \, \psi(\gamma\inv)\bigl[
			\dfrac{\sgn_\varpi(-1) + 1}{2}
		\bigr]
	\Bigr\}
} 
\longcondition{\gamma \in T^\theta \setminus Z(G) T^\theta_{r} } 
\longformula{
 \displaystyle \dfrac{q^{-1/2}}{2\smabs{D_G(\gamma)}^{1/2}}
	\displaystyle\sum_{\substack{
		\gamma' \in (C_\varpi)_{r:\MPlus r} \\
		\gamma' \ne \gamma^{\pm1}
	}}
		\sgn_\varpi\bigl(
			\Tr_\varpi(\gamma - \gamma')
		\bigr)\psi(\gamma') 
}
\longformula{  \displaystyle  \qquad \qquad 
+ \quad \frac1 2 H(\SSAddChar, \field_\varpi)
\sgn_\varpi\bigl(\eta\inv\Im_\varpi(\gamma)\bigr)
\dfrac{
	\psi(\gamma) + \psi(\gamma\inv)
}{
	\smabs{D_G(\gamma)}^{1/2}
}
}
\longcondition{\gamma \in T^{\varpi,\eta}_r \setminus T^{\varpi,\eta}_{\MPlus r} }
\formula{  \displaystyle
\dfrac{q^{-1/2}}{2\smabs{D_G(\gamma)}^{1/2}}
\displaystyle\sum_{\gamma' \in (C_\varpi)_{r:\MPlus r}}
	\sgn_\varpi\bigl(
		\Tr_{\epsilon\varpi}(\gamma) - \Tr_\varpi(\gamma')
	\bigr)\psi(\gamma')
}
\condition{
\gamma \in T^{\epsilon \varpi,\eta}_r \setminus  T^{\epsilon \varpi,\eta}_{\MPlus r}
}

\formula{ \displaystyle
c_0(\pi) + H(\SSAddChar, \field_\varpi)\dfrac{
	\sgn_\varpi\bigl(\eta\inv\Im_\varpi(\gamma)\bigr)
}{
	\abs{D_G(\gamma)}^{1/2}
}
}
\condition{\gamma \in T^{\varpi, \eta}_{\MPlus r}}
\formula{ \displaystyle
c_0(\pi) +
\dfrac1{\abs{D_G(\gamma)}^{1/2}}
}
\condition{ \gamma \in A_{\MPlus r}}
\formula{c_0(\pi)}
\condition{\text{otherwise, if \(\gamma \in G_{\MPlus r}\)}}
\formula{0}
\lastcondition{\text{otherwise, if \(\gamma \notin G_{\MPlus r}\).}}
\end{array}
\right.
\]
Here $\eta \in \{1, \epsilon\}$
and \(c_0(\pi) = -\tfrac1 2(q + 1)q^{r - 1/2}\).
To obtain character values for a ramified representation $\pi = \pi(T^{\epsilon \varpi}, \psi)$,
interchange the roles of $\varpi$ and $\epsilon \varpi$
in the formulas above.

From Theorem~\ref{thm:exc-far} and Theorem~\ref{thm:exc-near},
for the representations $\pi^\pm = \pi^\pm(T^\epsilon, \psi_0^1)$ (see Definition~\ref{defn:depth-0-param} for explanation of notation), we have
\[
\Theta_{\pi^\pm}(\gamma)
= 
\begin{cases}
 \displaystyle \frac{\sgn_\varpi(\gamma + \gamma\inv + 2)}2
\Bigl\{
	H(\SSAddChar, \field_\epsilon)\frac{
		\sgn_\epsilon\bigl(\Im_\epsilon(\gamma)\bigr)
	}{
		\smabs{D_G(\gamma)}^{1/2}
	} - 1
\Bigr\} 
& \quad{\gamma \in T^\epsilon \setminus Z(G) T^\epsilon_{\MPlus0} }\\
 \displaystyle \frac1 2\Bigl\{
	\frac1{\smabs{D_G(\gamma)}^{1/2}} - 1
\Bigr\} & \quad{\gamma \in A_{\MPlus0}}\\
 \displaystyle\frac1 2\Bigl\{
	\pm H(\SSAddChar, \field_{\theta'})\frac{
		\sgn_{\theta'}\bigl(
			\eta\inv\Im_{\theta'}(\gamma)
		\bigr)
	}{
		\smabs{D_G(\gamma)}^{1/2}
	} -
	1
\Bigr\} 
& \quad \gamma \in T^{\theta', \eta}_{\MPlus0},
\\ 
0 & \quad\text{otherwise}.
\end{cases}
\]
Here \(\eta \in \sset{1, \epsilon}\) if \(\theta' \in \sset{\varpi, \epsilon\varpi}\) and
\(\eta \in \sset{1, \varpi}\) if \(\theta' = \epsilon\).

Our character formulas for
\(\pi^\pm = \bPi^\pm(T, \psi)\) agree with those of
\cite{sally-shalika:characters} for
\(\Pi^\pm(\SSAddChar_\pi, \psi, \field_\epsilon)\).
Similarly, our character formulas for
\(\pi = \bPi(T^\theta, \psi)\) agree with those of
\cite{sally-shalika:characters} for
\(\Pi(\SSAddChar_\pi, \psi, \field_\theta)\),
\textit{except} that, far from the identity
(see Theorem \ref{thm:far}), they differ by a sign
in the ramified case
when \(\sgn_\varpi(-1) = -1\)
and the conductor \(h = r+1/2\) is odd.
See Remark \ref{rem:far}.

\textsc{Acknowledgment}:
It is a pleasure to thank Jeffrey Hakim and an anonymous referee
for many helpful comments.

\section{Field extensions}
\label{sec:sqrt}

Since \(p \ne 2\), every quadratic extension of
\(\field\) is of the form
\(\indexmem{\field_\theta} \ldef \field(\sqrt\theta)\) for some
non-square \(\theta \in \field\mult\).
Let
$\map{\indexmem{\Norm_\theta}}{\field_\theta\mult}{\field\mult}$
and
$\map{\indexmem{\Tr_\theta}}{\field_\theta}\field$
denote the relevant norm and trace homomorphisms, respectively.  We write
\(\indexmem{C_\theta} = \ker \Norm_\theta\)
and
\(\indexmem{V_\theta} = \ker \Tr_\theta\),
and denote by \indexmem{\sgn_\theta} the unique, non-trivial
character of \(\field\mult\) that is trivial on
\(\Norm_\theta(\field_\theta\mult)\).
In particular,
\begin{equation}
\label{eq:sgn-un}
\sgn_\theta(x) = (-1)^{\ord(x)}
	\quad\text{for \(x \in \field\mult\)}
\end{equation}
if \(\field_\theta/\field\) is unramified,
and
\begin{equation}
\label{eq:sgn-ram}
\begin{aligned}
\sgn_\theta(\theta) & {}= \sgn_\theta(-1) \\
\sgn_\theta(x)      & {}= \sgn_{\pint\mult}(x)
	\quad\text{for \(x \in \pint\mult\),}
\end{aligned}
\end{equation}
where \(\sgn_{\pint\mult}\) is the quadratic character of
\(\pint\mult\),
if \(\field_\theta/\field\) is ramified.
Let \indexmem{\pint_\theta}, \indexmem{\pp_\theta},
and \indexmem{\resfld_\theta}
denote the analogues for \(\field_\theta\)
of \(\pint\), \(\pp\), and \resfld, respectively.

For \(\alpha = a + b\sqrt\theta \in \field_\theta\), we write
\(\indexmem{\Re_\theta}(\alpha) = a\)
and \(\indexmem{\Im_\theta}(\alpha) = b\).

Write \(\indexmem{\ord_\theta}\) for the valuation on
\(\field_\theta\) that extends the one on \(\field\).
In particular,
\(\ord_\theta(\sqrt\theta) = \tfrac1 2\ord(\theta)\).

\begin{lemma}
\label{lem:cayley-field}
The map
\(\map\cayley
	{\field_\theta \setminus \sset2}
	{\field_\theta \setminus \sset{-1}}
\)
given by $\cayley(x) = (1 + \frac{x}{2})/(1 - \frac{x}{2})$
is a \(\Gal(\field_\theta/\field)\)-equivariant bijection,
and restricts to a bijection
\(\abmap
	{V_\theta \cap \pp_\theta}
	{C_\theta \cap (1 + \pp_\theta)}\).
For \(Y \in \pp_\theta^n\), with \(n > 0\), we have that
\(\cayley(Y) \equiv 1 + Y 
	\pmod{\pp_\theta^{2n}}
\).
\end{lemma}

\begin{lemma}
\label{lem:Im-facts}
For \(\alpha \in \field_\theta\),
we have
\begin{equation}
\tag{$*$}
\label{eq:ord-Im}
\ord(\Im_\theta(\alpha))
\ge \ord_\theta(\alpha) - \tfrac1 2\ord(\theta).
\end{equation}
If \(\ord_\theta(\alpha) \ge \ord_\theta(\alpha - t)\)
for all \(t \in \field\),
then we have equality in \eqref{eq:ord-Im};
and if, further, \(\alpha \in \pp_\theta\), then
\[
\Im_\theta\bigl(\cayley(\alpha)\bigr)
\equiv \Im_\theta(\alpha) \pmod{1 + \pp}.
\]
\end{lemma}

\begin{proof}
Write \(d = \ord_\theta(\alpha)\),
and \(\alpha = a + b\sqrt\theta\).
The inequality \eqref{eq:ord-Im} follows from the fact that
\(\ord_\theta(\alpha) = \min \sset{
	\ord(a), \ord(b) + \tfrac1 2\ord(\theta)
}.\)

If \(\ord(b) > d - \tfrac1 2\ord(\theta)\), then
\(\ord_\theta(\alpha - a) = \ord_\theta(b\sqrt\theta) > d\).

If \(\ord_\theta(\alpha) \ge \ord_\theta(\alpha - t)\)
for all \(t \in \field\),
and \(\alpha \in \pp_\theta\) (i.e., \(d > 0\)),
then applying \eqref{eq:ord-Im} to
\(\cayley(\alpha) - (1 + \alpha)\)
gives (by Lemma \ref{lem:cayley-field})
\[
\ord\bigl(\Im_\theta(\cayley(\alpha)) - b\bigr)
\ge 2d - \tfrac1 2\ord(\theta)
= \ord(b) + d.
\]
Since \(d > 0\), the result follows.
\end{proof}

\begin{rem}
As an illustration of the result,
and reminder of our normalization of valuation,
recall that
\(\ord_\varpi(\sqrt\varpi) = \tfrac1 2\).
Now fix an odd integer \(n\) and put
\(\alpha = \sqrt\varpi{}^n\).
For all \(t \in \field\),
\(\ord_\varpi(\alpha - t) = \ord(t) < \ord_\varpi(\alpha)\)
if \(t \not\in \pp^{(n + 1)/2}\)
and
\(\ord_\varpi(\alpha - t) = \ord_\varpi(\alpha)\)
if \(t \in \pp^{(n + 1)/2}\).
Then \(\Im_\varpi(\alpha) = \varpi^{(n - 1)/2}\),
so that \(\ord(\Im_\varpi(\alpha)) = \tfrac 1 2(n - 1)\),
and \(\ord_\varpi(\alpha) = \tfrac1 2 n\);
in particular, we have the equality
\(\ord(\Im_\varpi(\alpha))
= \ord_\varpi(\alpha) - \tfrac1 2\ord(\varpi)\).
Finally,
\[
\cayley(\alpha)
= 1 + \sqrt\varpi{}^n + 2(\sqrt\varpi{}^n)^2
	+ 4(\sqrt\varpi{}^n)^3 + \dotsb,
\]
so that
\[
\Im_\varpi\bigl(\cayley(\alpha)\bigr)
= \varpi^{(n - 1)/2} + 4\varpi^{(3n - 1)/2} + \dotsb
\equiv \varpi^{(n - 1)/2} = \Im_\varpi(\alpha)
	\pmod{1 + \pp}.
\]
\end{rem}

The next two technical lemmas will come in handy in working
out the arithmetic of the character formulas.

\begin{notn}
\label{notn:field-quad}
Write \indexmem{\psi_0} for the quadratic character of
\(C_\epsilon\)
(so that \(\psi_0^2 = 1\), but \(\psi_0 \ne 1\)).
\end{notn}

We shall also use \(\psi_0\) later for quadratic characters
on related groups.

\begin{lemma}[%
	\cite{sally-shalika:orbital-integrals}*{Lemma A.3}%
]
\label{lem:quad-char}
For \(\lambda \in C_\epsilon\),
\[
\psi_0(\lambda) = \begin{cases}
-\sgn_\varpi(-1),                       & \lambda = -1      \\
\sgn_\varpi(\lambda + \lambda\inv + 2), & \text{otherwise.}
\end{cases}
\]
\end{lemma}

\begin{proof}
Clearly,
\[
\psi_0(1) = 1 = \sgn_\varpi(1 + 1\inv + 2).
\]
Now note that \(\psi_0\) takes the value
\(+1\) on squares in \(C_\epsilon\), and \(-1\) on
non-squares.
In particular, since \(C_\epsilon\) is the direct product of
a cyclic group of order \(q + 1\) and a pro-\(p\) group,
\[
\psi_0(-1) = (-1)^{(q + 1)/2}
= -(-1)^{(q - 1)/2} = -\sgn_\varpi(-1).
\]

%
%
Now fix
\(\lambda \in C_\epsilon \setminus \field\mult
	= C_\epsilon \setminus \sset{\pm1}\).
Since \(C_\epsilon \subseteq (\field_\epsilon\mult)^2\),
we may write \(\lambda = (c + d\sqrt\epsilon)^2\), with
\(c, d \in \field\mult\).
In particular, \(c^2 - d^2\epsilon \in \sset{\pm1}\);
and \(\lambda\) is a square in \(C_\epsilon\)
if and only if \(c^2 - d^2\epsilon = +1\).
Now
\(\lambda + \lambda\inv + 2
= 2\bigl((c^2 + d^2\epsilon) + 1\bigr)\).
If \(c^2 - d^2\epsilon = +1\), then
\[
\lambda + \lambda\inv + 2
= 2\bigl((c^2 + d^2\epsilon) + (c^2 - d^2\epsilon)\bigr)
= (2c)^2 \in (\field\mult)^2.
\]
If \(c^2 - d^2\epsilon = -1\), then
\[
\lambda + \lambda\inv + 2
= 2\bigl((c^2 + d^2\epsilon) - (c^2 - d^2\epsilon)\bigr)
= (2d)^2\epsilon \not\in (\field\mult)^2.\qedhere
\]
\end{proof}

Our next lemma discusses traces of norm-one elements,
for use in Theorem \ref{thm:pos-bad-ram}.

\begin{lemma}
\label{lem:norm-and-tr}
If \(\theta\) is a non-square
and \(X \in V_\theta \cap \pp_\theta\), then
\[
\ord\bigl(
	\Tr_\theta(\cayley(X)) - (2 + X^2)
\bigr)
\ge 4\ord_\theta(X).
\]
\end{lemma}

\begin{proof}
Note that \(X^2 \in \field\) and \(\bar X = -X\).
By direct computation and Lemma \ref{lem:cayley-field},
\begin{multline*}
\Tr_\theta\bigl(\cayley(X)\bigr)
= \cayley(X) + \overline{\cayley(X)}
= \cayley(X) + \cayley(\bar X)
= \cayley(X) + \cayley(-X) \\
= 2\cayley\bigl(X^2/2\bigr)
\equiv 2 + X^2 \pmod{
	\pp^{2\ord(X^2)} = \pp^{4\ord_\theta(X)}
},
\end{multline*}
where \(\abmapto Z{\bar Z}\) is the non-trivial element of
\(\Gal(\field_\theta/\field)\).  The result follows.
\end{proof}

\section{Tori}
\label{sec:tori}

\subsection{Standard tori and normalizers}
\label{sec:std-tori}

Every maximal \(k\)-torus in \bG is \(G\)-conjugate either to the
\(k\)-split torus
\(\bA \ldef \set{\begin{smallpmatrix}
a & 0 \\
0 & d
\end{smallpmatrix}}{a d = 1}\),
or to an elliptic \(k\)-torus (discussed below).
The quotient \(N_\bG(\bA)/\bA\) has order \(2\), with the
non-trivial coset represented by
\(\begin{smallpmatrix}
\phm0 & 1 \\
-1    & 0
\end{smallpmatrix} \in G\).

We start by defining a few model elliptic tori.

\begin{notn}
For \(\beta,\theta,\eta \in \field\mult\),
define
\[
X^{\theta,\eta}_\beta
= \begin{pmatrix}
0               & \beta\eta\inv \\
\beta\theta\eta & 0
\end{pmatrix}
\]
and
\[
\bT^{\theta,\eta} = \set{
\begin{pmatrix}
	a       & b\eta\inv \\
	b\theta\eta & a
\end{pmatrix}
}{a^2 - b^2\theta = 1}.
\]
Then the Lie algebra
\(\indexmem{\pmb\ttt^{\theta, \eta}}\) of \(\bT^{\theta, \eta}\)
is the \(1\)-dimensional vector space spanned by
\(X^{\theta, \eta}_1\).
\end{notn}

We will only use this notation for $\theta$ a non-square,
since otherwise $\bT^{\theta,\eta}$ is $\field$-split,
and thus $G$-conjugate to $\bA$.
We write \indexmem{\bT^\theta} for $\bT^{\theta,1}$.
There is a natural way to view the torus \(\bT^\epsilon\) as
an \resfld-group.  To emphasise when we are doing so, we shall
denote it by \indexmem{\sfT^\epsilon}.

We have that
\(T^{\theta,\eta}\) is isomorphic to \(C_\theta\),
and
its Lie algebra
\(\ttt^{\theta,\eta} = \set{X^{\theta,\eta}_b}{b \in \field}\)
is isomorphic to \(V_\theta\),
in each case via the map
\(\abmapto{\begin{smallpmatrix}
a           & b\eta\inv \\
b\theta\eta & a
\end{smallpmatrix}}{a + b\sqrt\theta}\).
We shall therefore freely write
\(\Im_\theta(\gamma)\) (respectively, \(\Im_\theta(Y)\))
for \(\gamma \in T^{\theta, \eta}\)
(respectively, \(Y \in \ttt^{\theta, \eta}\)).

To determine all $G$-conjugacy classes of elliptic maximal tori,
we use the fact
that such a torus is determined up to stable conjugacy,
or, equivalently, \(\GL_2(\field)\)-conjugacy,
by the \(\field\)-isomorphism class of its splitting field;
and
that the splitting field of $\bT^{\theta,\eta}$
is $\field_\theta$.
To represent all $\bG$-conjugacy classes of elliptic maximal tori,
we thus only need to consider values of $\theta$
in $\sset{\epsilon,\varpi,\epsilon\varpi}$.
We shall call a torus \term[standard torus]{standard}
if it is the split torus \bA, or else of the form
\(\bT^{\theta, \eta}\), with \(\theta\) as above.

If \(-1\) is not a norm from \(\field_\theta\), then
\(N_\bG(\bT^\theta) = \bT^\theta\).
If \(-1 = \Norm_\theta(a + b\sqrt\theta)\), then
\(N_\bG(\bT^\theta)/\bT^\theta\) has order \(2\), with the
non-trivial coset represented by
\(\begin{smallpmatrix}
a        & \phm b  \\
-b\theta & -a
\end{smallpmatrix} \in G\).
The first case applies if
\(\field_\theta/\field\) is ramified
and \(-1\) is not a square in \(\resfld\mult\);
and
the second otherwise (in particular, if
\(\field_\theta/\field\) is unramified).

Thus, there are two distinct
stable conjugacy classes of ramified tori,
represented by the tori $\bT^\theta$,
with $\theta \in \sset{\varpi, \epsilon\varpi}$. 
If \(-1\) is a square in \(\resfld\mult\), then the
stable conjugacy class of \(\bT^\theta\)
splits into two distinct \(G\)-conjugacy classes, represented
by \(\bT^{\theta,1}\)
and \(\bT^{\theta,\epsilon}\);
otherwise, it is a single \(G\)-conjugacy class.
In particular, \(\bT^{\theta, 1}\)
and \(\bT^{\theta, \epsilon}\) are \(G\)-conjugate in the
latter case.

There is a single stable conjugacy class of
unramified elliptic maximal \(\field\)-tori,
represented by \(\bT^\epsilon\).
It splits into two distinct \(G\)-conjugacy classes,
represented 
by
\(\bT^{\epsilon,1}\)
and
\(\bT^{\epsilon,\varpi}\).

\begin{notn}
\label{notn:torus-quad}
We write \indexmem{\psi_0^\eta} for the quadratic character
of \(T^{\epsilon, \eta}\), with \(\eta \in \sset{1, \varpi}\),
and \indexmem{\psi_0} for the quadratic character of
\(\sfT^\epsilon(\resfld)\).
\end{notn}

Recall that the notation \(\psi_0\) used below has already
been used for a character of a subgroup of \(C_\epsilon\)
(see Notation \ref{notn:field-quad}).  Since an isomorphism
of \(T^{\epsilon, \eta}\) with \(C_\epsilon\) intertwines
\(\psi_0^\eta\) and \(\psi_0\), this is not a serious
ambiguity.

\subsection{Torus filtrations}

Since our tori are isomorphic to
subgroups of multiplicative groups of fields, they carry
natural filtrations.
We specify them explicitly below,
since there are some normalization issues,
as well as a subtlety to be handled in the depth-zero case.

\begin{defn}
\label{defn:torus-filt}
We equip \(\field\mult\) with the filtration
defined by \((\field\mult)_0 = \pint\mult\)
and
\((\field\mult)_n = 1 + \pp^n\) for \(n \in \Z_{> 0}\),
and extend the indexing to \(r \in \R_{\ge 0}\)
by putting \((\field\mult)_r = (\field\mult)_{\rup r}\).
This may be transported, \textit{via} either isomorphism
\(A \iso \field\mult\), to a filtration
\(\set{A_r}{r \in \R_{\ge 0}}\) of \(A\).
In particular, \(A_0 = \SL_2(\pint) \cap A\) is the maximal
compact subgroup of \(A\).

Any elliptic torus \(T\) is conjugate to one of the form
\(T^{\theta, \eta}\), hence isomorphic to
\(C_\theta \subseteq \field_\theta\mult\).
We impose a filtration on \(\field_\theta\mult\) similar to
the one on \(\field\mult\),
except that we omit the maximal term
\((\field_\theta\mult)_0\),
and adjust the indexing to account for ramification.
Specifically, if \(\field_\theta/\field\) is ramified, then
we put \((\field_\theta\mult)_{n/2} = 1 + \pp_\theta^n\)
for \(n \in \Z_{> 0}\);
whereas, if \(\field_\theta/\field\) is unramified, then
we put \((\field_\theta\mult)_n = 1 + \pp_\theta^n\)
for \(n \in \Z_{> 0}\).
As in the split case, we extend the indexing of the
filtration to \(r \in \R_{> 0}\).
We obtain a filtration of \(C_\theta\) by intersection,
and then transport it to \(T\)
to obtain a filtration
\(\set{T_r}{r \in \R_{> 0}}\) of \(T\).

Note, however, that we have not yet defined the notation
\(T_0\).
If \(T\) is unramified
(i.e., we may take \(\theta = \epsilon\)),
then we put \(T_0 = T\), which is the maximal compact
subgroup of \(T\).
If \(T\) is ramified
(i.e., we may take \(\theta \in \sset{\varpi, \epsilon\varpi}\)),
then we put \(T_0 = \bigcup_{r > 0} T_r\).
(In the notation of the next paragraph, this is
\(T_{\MPlus0}\).)
This is no longer the maximal, compact subgroup of \(T\),
but rather an index-\(2\) subgroup.
Specifically, \(T = Z(G)T_0\).

For any torus \(H\) and real number \(r \ge 0\),
we write \(H_{\MPlus r}\) for \(\bigcup_{s > r} H_s\).

We define filtrations in a similar way (including the
adjustment for ramification) on the Lie algebras of maximal
tori in \(G\).  These filtrations are defined for all
\(r \in \R\) (not just \(r \ge 0\)), and the case \(r = 0\)
no longer needs to be treated separately.
\end{defn}

\begin{defn}
\label{defn:depth-char}
If \(T\) is a torus, and \(\psi\) a character of \(T\), then
the \term[depth (multiplicative character)]{depth} \(\depth(\psi)\)
of $\psi$
is the smallest index \(r \in \R_{\ge 0}\)
such that
the restriction of $\psi$ to $T_{\MPlus r}$ is trivial.
\end{defn}

\begin{defn}
\label{defn:depth-element}
Let \(T\) be a torus, with Lie algebra \(\ttt\).
If \(Y\) is a regular, semisimple element of \(\ttt\)
(respectively,
\(\gamma\) is a regular, semisimple element of \(T\)),
then we define the \term[depth (Lie algebra element)]{depth}
\(\depth^\gg(Y)\) of \(Y\)
(respectively, the \term[depth (group element)]{depth}
\(\depth^G(\gamma)\) of \(\gamma\))
to be the smallest index \(r \in \R\)
such that \(Y \not\in \ttt_{\MPlus r}\)
(respectively, \(\gamma \not\in T_{\MPlus r}\)).
If it is clear from the context, then we will drop the
superscript \(G\) or \(\gg\).
It is also convenient to define the
\term{maximal depth} \(\mdepth(\gamma)\) of \(\gamma\)
to be \(\max\set{\depth(\gamma z)}{z \in Z(G)}\).
\end{defn}

\begin{rem}
For example,
\(\depth^\gg(X^{\theta, \eta}_\beta)
= \ord(\beta) + \tfrac1 2\ord(\theta)\).
Thus, for all \(r \in \R\),
\begin{align*}
\ttt^{\epsilon, \eta}_r
& {}= \varpi^{\rup r}\pint\dotm X^{\epsilon, \eta}_1
\quad\text{%
	for \(\eta \in \sset{1, \varpi}\)%
} \\
\intertext{and}
\ttt^{\theta, \eta}_{r + 1/2}
& {}= \varpi^{\rup r}\pint\dotm X^{\theta, \eta}_1
\quad\text{%
	for \(\theta \in \sset{\varpi, \epsilon\varpi}\)
and \(\eta \in \sset{1, \epsilon}\).%
}
\end{align*}
\end{rem}

We will see below that the formulas for the supercuspidal
characters that we consider depend on the relative depths of
a (linear) character of a torus, and an element of that
torus.
Another basic function is the \term{Weyl discriminant}.

\begin{defn}
\label{defn:disc}
The functions
\(\map{D_G}G\field\) and \(\map{D_\gg}\gg\field\)
are defined by letting
\(D_G(\gamma)\) and \(D_\gg(Y)\)
be the coefficients of the degree-one terms in the
characteristic polynomials of
\(\Ad(\gamma) - 1\) and \(\ad(Y)\), respectively,
for \(\gamma \in G\) and \(Y \in \gg\).
Concretely,
\[
D_G\begin{pmatrix}
a & b \\
c & d
\end{pmatrix} = (a + d)^2 - 4
\qandq
D_\gg\begin{pmatrix}
a & \phm b  \\
c &     -a
\end{pmatrix} = 4(a^2 + bc).
\]
An element \(Y \in \gg\)
(respectively, \(\gamma \in G\)) is
\term{regular semisimple}
if and only if \(D_\gg(Y) \ne 0\)
(respectively, \(D_G(\gamma) \ne 0\)).
We write \indexmem{\gg\rss}
and \indexmem{G\rss} for the appropriate sets
of regular, semisimple elements.
\end{defn}

\begin{lemma}
\label{lem:disc-as-depth}
The discriminant of a regular, semisimple element
of \(\gg\) or \(G\) with eigenvalues \(\lambda\) and \(\lambda'\)
is \((\lambda - \lambda')^2\).
If \(Y \in \gg\rss\) and \(\gamma \in G\rss \cap G_0\), then
\[
\abs{D_\gg(Y)} = q^{-2\depth(Y)}
\qandq
\abs{D_G(\gamma)} = q^{-2\mdepth(\gamma)}.
\]
\end{lemma}

\begin{proof}
Note that
\begin{itemize}
\item neither \(D_\gg\) nor \(D_G\) is affected by passage to a
field extension;
\item \(D_\gg\) (respectively, \(D_G\)) is invariant under
the adjoint (respectively, conjugation) action of \(G\);
and
\item
neither \(\depth^\gg\),
nor the restriction of \(\depth^G\) to \(G_0\),
is affected by passage to a quadratic extension
\cite{adler-spice:good-expansions}*{%
	Lemma \xref{exp-lem:domain-field-ascent}%
}.
\end{itemize}
Thus, it suffices to prove the result for \(Y \in \mf a\)
and \(\gamma \in A_0\).
We shall only consider the group case; the Lie-algebra case
is similar.

By Definition \ref{defn:disc}, if
\(\gamma = \begin{smallpmatrix}
\lambda & 0\phinv     \\
0       & \lambda\inv
\end{smallpmatrix}\), then
\[
D_G(\gamma)
= (\lambda + \lambda\inv)^2 - 4
= (\lambda - \lambda\inv)^2,
\]
as desired.
If \(\gamma \in A_0\), then \(\lambda \in \pint\mult\),
so \(D_G(\gamma) \in \pint\);
and \(D_G(\gamma) \in \pp\) if and only if
\(\lambda \equiv \pm1 \pmod{1 + \pp}\),
in which case \(\gamma \in Z(G)A_{\MPlus0}\).
If \(\gamma \in A_{\MPlus0}\),
then we may write \(\lambda = \cayley(X)\) for some
\(X \in \pp\).
By Lemma \ref{lem:cayley-field} and direct calculation,
\[
\lambda - \lambda\inv
= \cayley(X) - \cayley(X)\inv
= \cayley(X) - \cayley(-X)
= \frac{2X}{1 - (X/2)^2},
\]
so that
\[
\ord\bigl(D_G(\gamma)\bigr)
= 2\ord(\lambda - \lambda\inv)
= 2\ord(X)
= 2\mdepth(\gamma).\qedhere
\]
\end{proof}

{
\def\AddChar{\SSAddChar}
\section{A principal-value integral}
\label{sec:SSPhi}

The character formulas of
\cite{sally-shalika:characters}
involve a quantity
\[
H(\AddChar, \field_\theta)
\ldef \lim_{m \to -\infty}
	\int_{\varpi^m\pint_\theta}
		\AddChar(\Norm_\theta(z))
	\textup d_{\AddChar}z
\]
(p.~1235 \loccit and \cite{shalika:thesis}*{p.~11}),
where
\(\AddChar\) is an additive character
and
\(\textup d_{\AddChar}z\) is the self-dual Haar measure on
\(\field_\theta\)
(with respect to the additive character \(\AddChar\)
and the trace pairing on \(\field_\theta\)),
in the sense of Definition \ref{defn:FT} below.
We compute the normalization of measure,
and the resulting integral,
below.

Write \(r\) for the depth of \(\AddChar\).

The evaluation of the integral will require, as is usual in
\(p\)-adic harmonic analysis, a fourth root of unity called
a \term{Gauss sum}.
We follow \cite{shalika:thesis}*{p.~5}
(see also
\cite{spice:sl2-mu-hat}*{Definition \xref{sl2-defn:G}})
in making the following definition.

\begin{defn}
\label{defn:gauss}
\[
\Gauss(\AddChar)
\ldef q^{-1/2}\sum_{X \in \pint/\pp}
	\AddChar_{(-\varpi)^r}(X^2).
\]
\end{defn}

In \cite{waldspurger:nilpotent}*{\S I.4},
Waldspurger works with an additive character \(\psi\)
of depth \(0\).
If \(r = \depth(\AddChar) = 0\),
then \cite{spice:sl2-mu-hat}*{Lemma \xref{sl2-lem:G-facts}}
gives that
\(\Gauss(\AddChar)\) is the fourth root of unity denoted by
\(\varepsilon(\AddChar)\)
in \cite{waldspurger:nilpotent}*{\S V.4}.

Considerable information about the transformation laws
for this root of unity are available in
\cite{spice:sl2-mu-hat}*{Lemma \xref{sl2-lem:G-facts}},
where our \(\Gauss(\AddChar)\) is denoted by
\(G_\varpi(\Phi)\).

\begin{lemma}
\label{lem:SSPhi}
\begin{align*}
\meas_{\textup d_{\AddChar}z}(\pint_\theta)
& {}= \begin{cases}
q^{r + 1},   & \theta = \epsilon \\
q^{r + 1/2}, & \theta = \varpi;
\end{cases}
\intertext{and}
H(\AddChar, \field_\theta)
& {}= \begin{cases}
(-1)^{r + 1},     & \theta = \epsilon \\
\Gauss(\AddChar), & \theta = \varpi.
\end{cases}
\end{align*}
\end{lemma}

By
\cite{spice:sl2-mu-hat}*{Lemma \xref{sl2-lem:Gamma-varpi}},
this formula agrees with the one given in
\cite{sally-shalika:characters}*{p.~1235}.

\begin{proof}
Note that \(H(\AddChar, \field_\theta)\) is just what is called
\(\mc H(\AddChar, Q)\) in \cite{shalika:thesis}*{p.~11},
where \(Q = \Norm_\theta\) is the norm form on
\(\field_\theta\).
In particular, its definition involves a lattice
in \(\field_\theta\) (although Lemma 1.5.2 \loccit shows
that the choice does not matter).
For definiteness, we take the lattice to be \(\pint_\theta\).
By Lemma 1.5.1 \loccit and the following exposition,
the computation of \(\mc H(\AddChar, Q)\) begins with
the identification of a \(Q\)-orthogonal \(\pint\)-basis for
\(\pint_\theta\).
In our setting,
\[
\sset{\mo x_1 = 1, \mo x_2 = \sqrt\theta}
\]
will do.  Note that we have
\[
l_1 \ldef Q(\mo x_1) = 1
\qandq
l_2 \ldef Q(\mo x_2) = -\theta.
\]

Our \(\pint\)-basis for \(\pint_\theta\) is also a
\(\field\)-basis for \(\field_\theta\),
hence furnishes a \(\field\)-isomorphism
\(\field_\theta \iso \field \oplus \field\).
By \cite{shalika:thesis}*{pp.~5, 11},
\(\textup d_{\AddChar}z\) is the
pull-back along the above isomorphism of
\(\textup dx_1 \oplus \textup dx_2\),
where
\begin{align*}
\meas_{\textup dx_1}(\pint)
    = q^{(\depth(\AddChar_{l_1}) + 1)/2}
& {}= q^{(r + 1)/2}
\intertext{and}
\meas_{\textup dx_2}(\pint)
    = q^{(\depth(\AddChar_{l_2}) + 1)/2}
& {}= \begin{cases}
q^{(r + 1)/2}, & \theta = \epsilon \\
q^{r/2},       & \theta = \varpi
\end{cases}
\end{align*}
(by \eqref{eq:depth-Phi-b}
and Definition \ref{defn:lots}).
Since \(\pint_\theta\) is the pull-back of
\(\pint \oplus \pint\), we have the indicated normalization.

By \cite{shalika:thesis}*{p.~11},
\(\mc H(\AddChar, Q)
= \mc H(\AddChar_{l_1})\mc H(\AddChar_{l_2})\),
where the notation is as in Lemma 1.3.2 \loccitthendot.

In the unramified case,
since \(\ord(l_1) = \ord(1) = 0\)
and \(\ord(l_2) = \ord(-\epsilon) = 0\),
we have by \eqref{eq:depth-Phi-b} that
\[
\depth(\AddChar_{l_1}) = \depth(\AddChar_{l_2}) = r.
\]
On the other hand,
\[
\sgn_\varpi(l_1) = \sgn_\varpi(1) = 1
\qandq
\sgn_\varpi(l_2) = \sgn_\varpi(-\epsilon)
= -\sgn_\varpi(-1).
\]
If \(\theta = \epsilon\)
and \(r = \depth(\AddChar)\) is even, then, by
\cite{shalika:thesis}*{Lemma 1.3.2}
and \cite{spice:sl2-mu-hat}*{Lemma \xref{sl2-lem:G-facts}},
\begin{align*}
\mc H(\AddChar, Q)
& {}= \sgn_\varpi(l_1)\Gauss(\AddChar)\dotm
\sgn_\varpi(l_2)\Gauss(\AddChar) \\
& {}= \sgn_\varpi(-\epsilon)\sgn_\varpi(-1) \\
& {}= -1 = (-1)^{r + 1}.
\end{align*}
If \(\theta = \epsilon\) and \(r\) is odd, then
\cite{shalika:thesis}*{Lemma 1.3.2} gives
\[
\mc H(\AddChar, Q) = 1\dotm1 = 1 = (-1)^{r + 1}.
\]

In the ramified case,
\(\ord(l_1) = \ord(1) = 0\)
and
\(\ord(l_2) = \ord(-\varpi) = 1\),
so
\[
\depth(\AddChar_{l_1}) = r
\qandq
\depth(\AddChar_{l_2}) = r - 1
\]
(meaning that exactly one is even);
whereas
\[
\sgn_\varpi(l_1) = \sgn_\varpi(1) = 1
\qandq
\sgn_\varpi(l_2) = \sgn_\varpi(-\varpi) = 1.
\]
Thus, using \cite{shalika:thesis}*{Lemma 1.3.2}
and \cite{spice:sl2-mu-hat}*{Lemma \xref{sl2-lem:G-facts}}
again, we obtain the desired formula for \(\mc H(\AddChar, Q)\)
(regardless of the parity of \(r\)).
\end{proof}
}

\section{The building and filtrations}
\label{sec:building}
In the 1990s,
Moy and Prasad~\cite{moy-prasad:k-types}
used Bruhat--Tits theory to initiate
a major advance in the study
of the harmonic analysis of reductive $p$-adic groups.
In particular, Adler~\cite{adler:thesis},
Yu~\cite{yu:supercuspidal},
and Kim~\cite{jkim:exhaustion}
built upon this foundation to provide a classification
of all supercuspidal representations of a reductive $p$-adic group
(under some tameness restrictions).
In this paper, we  will not assume any familiarity with Bruhat--Tits
theory,
but we shall use the notation of Moy and Prasad
for certain objects that we will describe explicitly
in \S\ref{sec:filt}.

\subsection{Lattices and filtrations}
\label{sec:filt}

Recall that a \term{lattice} in a finite-dimensional
\(\field\)-vector space is a compact, open
\(\pint\)-submodule.
For example, for each $r \in \R$, 
\[
\gl_2(\field)_{\xleft, r}
\ldef
\begin{pmatrix}
\pp^{\rup r} & \pp^{\rup r} \\
\pp^{\rup r} & \pp^{\rup r}
\end{pmatrix}
\qandq
\gl_2(\field)_{\xcentre, r}
\ldef \begin{pmatrix}
\pp^{\rup r}       & \pp^{\rup{r - 1/2}} \\
\pp^{\rup{r + 1/2}} & \pp^{\rup r}
\end{pmatrix}
\]
are examples of lattices in $\mathfrak{gl}_2(\field)$.
In each case,  if we allow $r$ to vary,
we obtain a filtration of $\mathfrak{gl}_2(\field)$.
Because these two filtrations have played
an important role in the representation
theory of both $\SL_2(\field)$ and $\GL_2(\field)$,
they were
often given their own special notations in the literature.
For example, when $r\in \Z$, the first lattice above was often
called $\mathfrak{k}_r$, and, when $r \in \frac12\Z$, the second
was often called $\mathfrak{b}_{2r}$.

We may view these two filtrations as part of a large family,
indexed by elements \(x\)
of the (reduced) Bruhat--Tits building \(\BB(\GL_2, \field)\)
\cite{bruhat-tits:reductive-groups-1}*{D\'efinition 7.4.2}.
We shall not concern ourselves with a description of the
building; for the case \(\bG = \GL_2\), a point in the building
essentially \emph{is} a filtration as above
\cite{bruhat-tits:reductive-groups-1}*{Proposition 10.2.10}.
Thus, we may regard \(\sset{\xleft, \xcentre}\) as a subset
of \(\BB(\GL_2, \field)\); it contains (up to
\(\GL_2(\field)\)-conjugacy) all the \emph{optimal points}
\cite{moy-prasad:k-types}*{\S6.1}, and so is, in a sense,
``all that we need''.
When we are dealing with \(\SL_2\), we shall also require the
point \(\xright\) with associated filtration
\[
\gl_2(\field)_{\xright, r}
\ldef \Int\begin{pmatrix}
1 & 0      \\
0 & \varpi
\end{pmatrix}\gl_2(\field)_{\xleft, r}
= \begin{pmatrix}
\pp^{\rup r}     & \pp^{\rup r - 1} \\
\pp^{\rup r + 1} & \pp^{\rup r}
\end{pmatrix},\quad r \in \R;
\]
it is \(\GL_2(\field)\)-, but not \(\SL_2(\field)\)-,
conjugate to \(\xleft\).
We put
\(\indexmem\BBpts = \sset{\xleft, \xcentre, \xright}\),
and, for \((x, r) \in \BBpts \times \R\), define
\[
\gg_{x, r} = \sl_2(\field)_{x, r}
= \sl_2(\field) \cap \gl_2(\field)_{x, r}.
\]

\begin{rem}
\label{rem:gl2-mult}
For \(x \in \BBpts\) and \(r, s \in \R\), we have that
\[
\gl_2(\field)_{x, r}\dotm\gl_2(\field)_{x, s}
\subseteq \gl_2(\field)_{x, r + s},
\]
where \(\dotm\) is the usual matrix multiplication.
\end{rem}

In particular,
the set of invertible elements of
\(\gl_2(\field)_{x, 0}\) forms a group.
This motivates the following definitions,
for pairs \((x, r) \in \BBpts \times \R\)
with \(r \ge 0\):
\begin{align*}
\GL_2(\field)_{x, r} &
{}\ldef \GL_2(\field) \cap \gl_2(\field)_{x, r},
	&& r = 0 \\
\GL_2(\field)_{x, r} &
{}\ldef \GL_2(\field) \cap (1 + \gl_2(\field)_{x, r}),
	&& r > 0 \\
G_{x, r} = \SL_2(\field)_{x, r} &
{}\ldef \SL_2(\field) \cap \GL_2(\field)_{x, r},
	&& r \ge 0.
\end{align*}

\begin{rem}
Our definitions here seem rather \textit{ad hoc},
but Moy and Prasad
\cites{moy-prasad:k-types,moy-prasad:jacquet}
have shown how to fit them into a uniform framework that
applies to all reductive, \(p\)-adic groups.
\end{rem}

Next, we define a family of \(G\)-domains in \(G\)
(i.e., open and closed subsets, invariant under the
conjugation action of \(G\))
as follows.
For \(r \in \R_{\ge 0}\), put
\[
G_r \ldef \bigcup_{x \in \BBpts} \Int(G)G_{x, r}.
\]
We define \(\gg_r\) similarly, for \(r \in \R\).
In all cases, replacing an index \(r\) by \(\MPlus r\)
indicates taking the union over all indices \(s > r\).
For example,
\[
G_{\MPlus r} \ldef \bigcup_{s > r} G_s
\qandq
\gg_{x, \MPlus r} \ldef \bigcup_{s > r} \gg_{x, s}.
\]

Note that  $\gg_{x,\MPlus r} \subseteq \gg_{x,r}$.
The quotient $\gg_{x,r}/ \gg_{x,\MPlus r}$ is always a finite-dimensional
$\resfld$-vector space,
and we will soon see
(Lemma \ref{lem:cayley}\pref{item:MP-map})
that
the quotient $G_{x,r}/G_{x,\MPlus r}$ is isomorphic to 
$\gg_{x,r}/ \gg_{x,\MPlus r}$ if $r>0$.
The quotient $G_{x,0}/G_{x,\MPlus0}$
is the group of $\resfld$-rational points of a reductive
$\resfld$-group $\sfG_x$.
If $x \in \sset{\xleft, \xright}$, then
$\sfG_x \iso {\SL_2}_{/\resfld}$.
If $x = \xcentre$, then
$\sfG_x \iso {\mathbb G_m}_{/\resfld}$, the
\(1\)-dimensional, multiplicative group scheme over \resfld.

For $r,s \in \R$ with $s  \geq r$,
it is convenient to define $\gg_{x,r:s}$ to be the
quotient $\gg_{x,r}/\gg_{x,s}$.
If $r$ and $s$ are
non-negative, then
set $G_{x,r:s} \ldef G_{x,r}/G_{x,s}$.
Below we will use similar notation to denote quotients
of other filtration groups.

\subsection{Group filtrations and torus filtrations}
\label{sec:gp-torus}

Recall that we have already defined filtrations on tori.
It is natural to wonder how they fit into the framework that
we have just described.
It turns out that the filtration of an elliptic torus is
`associated to' a unique point in the building of \bG
over \(\field\).
For standard tori (see \S\ref{sec:std-tori}),
that point lies in our preferred set \(\BBpts\).

\begin{defn}
\label{defn:xT}
For \bT a standard, elliptic torus, let \(x = x_\bT\) be
the unique element of \(\BBpts\) such that
\begin{align*}
T_r = T \cap G_{x, r}
\quad&\text{for all \(r \in \R_{> 0}\)}
\qandq
\ttt_r = \ttt \cap \gg_{x, r}
\quad&\text{for all \(r \in \R\).}
\end{align*}
Explicitly,
\begin{gather*}
x_{\bT^{\epsilon, 1}} = \xleft,
\quad
x_{\bT^{\epsilon, \varpi}} = \xright, \\
\text{and}\quad
x_{\bT^{\varpi, \eta}} = \xcentre
	\quad\text{for \(\eta \in \sset{1, \epsilon}\).}
\end{gather*}
\end{defn}

On the other hand, for the split torus, we have that
\[
A_r = A \cap G_{x, r}
\quad\text{for all \((x, r) \in \BBpts \times \R_{\ge 0}\).}
\]

Further, we have defined the depth of a
regular, semisimple element
(see Definition \ref{defn:depth-element}).
By \cite{adler-spice:good-expansions}*{%
	Lemma \xref{exp-lem:domain-field-ascent}%
}
and \cite{adler-debacker:bt-lie}*{%
	Lemmas 3.5.3 and 3.7.25%
},
if \(\depth(Y) = r\),
then \(Y \in \gg_r \setminus \gg_{\MPlus r}\);
and, if \(\depth(\gamma) = r\),
then \(\gamma \in G_r \setminus G_{\MPlus r}\).
(Alternatively, in our situation, one could use
\cite{adler-debacker:bt-lie}*{\S3.6}.)
That is, our definition of depth is a special case of the
usual one \cite{adler-debacker:bt-lie}*{\S\S3.3, 3.7.3}.
Actually, one must take a little care if \(\gamma\) lies in a
ramified, elliptic torus \(T\), but not in \(T_0 = T_{\MPlus0}\);
but, even in this situation, the definitions agree,
since \(T \subseteq G_{\xcentre, 0}\).

\subsection{Group filtrations and Lie-algebra filtrations}
\label{sec:cayley}

We will use the \term{Cayley transform}
defined on an open subset of \(\gg\) by
$\mapto{\cayley}{X}{\frac{1+X/2}{1-X/2}}$,
to relate the filtrations on $\gg$ to those on $G$.  
The normalizing factor \(\tfrac1 2\) makes sure that
\(\cayley\) satisfies
\cite{adler-spice:explicit-chars}*{%
	Hypothesis \xref{char-hyp:strong-mock-exp}%
}.

The biggest possible domain for \(\cayley\) is the set of
matrices \(X\) which do not have \(2\) as an eigenvalue.
Note that a trace-zero matrix with this property also does
not have \(-2\) as an eigenvalue.
However, we need this enlarged domain only once (in
the proof of Lemma \ref{lem:conj-out}).
Everywhere else, we shall be concerned only with the
restriction of \(\cayley\) to \(\gg_{\MPlus0}\).

\begin{lemma}
\label{lem:cayley}
The Cayley transform $\cayley$ has the following properties
for any \(x \in \BBpts\) and \(r, s \in \R\) with \(r > 0\):
\begin{enumerate}[(a)]
\item
\label{item:cayley-equivariant}
It is equivariant under the adjoint and conjugation
actions of $G$ on the domain and codomain.
\item
\label{item:bijpos}
It maps $\gg_{x,r}$ bijectively onto $G_{x,r}$.
\item
\label{item:MP-map}
If $0<s \leq r \leq 2s$,
then it induces an isomorphism
$\abmap{\gg_{x,s:r}}{G_{x,s:r}}$.
\item
\label{item:cayley-torus}
If \bT is a maximal \(\field\)-torus in \bG,
then \(\cayley\) restricts to a bijection
\(\abmap{\ttt_{\MPlus0}}{T_{\MPlus0}}\).
If \(T \iso C_\theta\), with \(\theta\) a non-square, then
\(\cayley\) agrees with the map of
Lemma \ref{lem:cayley-field}.
\item
\label{item:cayley-inv}
If \(Y\) is in the domain of \(\cayley\), then
\(\cayley(-Y) = \cayley(Y)\inv\).
\item
\label{item:cayley-conj}
If $X\in \gg_{x,r}$ and $Y\in \gg_{x,s}$,
then
$$
\Ad(\cayley(X))Y \equiv Y+[X,Y] \pmod{\gg_{x,r+2s}}.
$$
\item
\label{item:cayley-bracket}
If $X\in \gg_{x,r}$, $Y\in \gg_{x,s}$,
and $s>0$ (as well as $r>0$), then
$$
\bigl[\cayley(X),\cayley(Y)\bigr]
\equiv \cayley\bigl([X,Y]\bigr)
	\pmod{G_{r+s+\min\sset{r,s}}}.
$$
\item
\label{item:cayley-depth-disc}
If \(Y \in \gg_{\MPlus0}\), then
\(\depth(Y) = \depth(\cayley(Y))\)
and
\(\abs{D_\gg(Y)} = \abs{D_G(\cayley(Y))}\).
\end{enumerate}
\end{lemma}

\begin{proof}
By Lemma \ref{lem:disc-as-depth}
and Remark \ref{rem:gl2-mult}, the result follows from
straightforward calculations, using that
\[
\cayley(X)
= 1 + 2\sum_{i = 1}^\infty \Bigl(\frac 1 2 X\Bigr)^i,
\quad X \in \gg_{\MPlus0}.\qedhere
\]
\end{proof}

The isomorphisms of
Lemma \ref{lem:cayley}\pref{item:MP-map}
are often called \term{Moy--Prasad maps}.
These are the same isomorphisms that appear in Yu's
construction of tame, supercuspidal representations
\cite{yu:supercuspidal} (see Lemma 1.3 \loccit).

\section{Haar measure}
\label{sec:measure}

Especially with calculations involving several orbital
integrals (see Proposition \ref{prop:MK-exc}), it is
necessary to be very careful about Haar measures.  We
describe the ones that we use here.

Waldspurger \cite{waldspurger:nilpotent}*{\S I.4}
defines a canonical way to normalize the
measure on a reductive, \(p\)-adic group.
In our setting, this gives measures
\indexmem{\textup dg}, \indexmem{\textup dt^\theta},
and \indexmem{\textup dz}
on \(G\), \(T^\theta\) (for \(\theta\) a non-square),
and \(Z(G)\), respectively,
such that
\begin{align*}
\meas_{\textup dg}(\SL_2(\pint)) &
	{}= \frac{q^2 - 1}{q^{1/2}},
& \meas_{\textup dt^\epsilon}(T^\epsilon_0) &
	{}= \frac{q + 1}{q^{1/2}}, \\
\meas_{\textup dt^\varpi}(T^\varpi_0) &
	{}= 1,
& \meas_{\textup dz}(\sset1) &
	{}= 1.
\end{align*}
Similarly, the Haar measures
\(\textup dt^{\varpi, \epsilon}\),
\(\textup dt^{\epsilon\varpi, 1}\),
and
\(\textup dt^{\epsilon\varpi, \epsilon}\)
on the obvious tori \(T\)
all have \(\meas(T_0) = 1\).

Note that these normalizations have pleasant properties with
respect to Moy--Prasad filtrations.
For example, we have defined the measure on \(G\) so that
\(\meas_{\textup dg}(G_{\xleft, 0})
= \card{\sfG_\xleft(\resfld)}\dotm
\card{\Lie(\sfG_\xleft)(\resfld)}^{-1/2}\);
but it is in fact true that
\(\meas_{\textup dg}(G_{x, 0})
= \card{\sfG_x(\resfld)}\dotm
\card{\Lie(\sfG_x)(\resfld)}^{-1/2}\)
for all \(x \in \BBpts\),
and, indeed, this is the definition that Waldspurger offers.

The computations of \cite{spice:sl2-mu-hat} use a different
normalization of quotient measure.
Namely, they involve the measures
\indexmem{\textup d_\theta\dot g}
on \(G/T^\theta\) defined by
\[
\meas_{\textup d_\epsilon\dot g}(
	\SL_2(\pint)/T^\epsilon
) = \frac{q - 1}q
\qandq
\meas_{\textup d_\theta\dot g}(
	\SL_2(\pint)/T^\theta
) = \frac{q^2 - 1}{2q^2}
\quad\text{for \(\theta \in \sset{\varpi, \epsilon\varpi}\)}.
\]
Thus,
\[
\frac{\textup dg}{\textup dt^\epsilon}
= q\dotm\textup d_\epsilon\dot g
\qandq
\frac{\textup dg}{\textup dt^\varpi}
= q^{3/2}\dotm\textup d_\varpi\dot g.
\]
For \(\theta \in \sset{\epsilon, \varpi}\),
we write \(\textup d_\theta'\dot g\) for the measure on
\(G/Z(G)\) such that
\[
\meas_{\textup d_\theta'\dot g}(\SL_2(\pint)/Z(G))
= \meas_{\textup d_\theta\dot g}(\SL_2(\pint)/T^\theta).
\]
Thus,
\[
\frac{\textup dg}{\textup dz}
= \frac1 2 q^{1/2}(q + 1)\dotm\textup d_\epsilon'\dot g
\qandq
\frac{\textup dg}{\textup dz}
= q^{3/2}\dotm\textup d_\varpi'\dot g.
\]
To reduce to a minimum the symbol-juggling asked of the
reader, we will often abuse notation by writing
\(\textup d_\theta\dot g\) instead of \(\textup d_\theta'\dot g\)
when the context makes it clear which measure is needed.

\section{Duality, Fourier transforms, and orbital integrals}
\label{sec:dualityfour}

Let \(V\) be a finite-dimensional \(\field\)-vector space
equipped with a non-degenerate, symmetric, bilinear pairing
\(\vform\).
We shall use this pairing to
identify the dual vector space
\(\indexmem{V^*} \ldef \Hom_\field(V, \field)\) with \(V\).

\subsection{Duality}
\label{sec:duality}

\begin{notn}
If \(\LL\) is a lattice in \(V\), then we write
\[
\LL^\bullet
\ldef \set{X \in V}{\langle X, \LL\rangle \subseteq \pp}.
\]
(The requirement
that \(\langle X, \LL\rangle \subseteq \pp\), rather than, say,
that \(\langle X, \LL\rangle \subseteq \pint\), is a result
of our choice of a depth-zero additive character
\(\AddChar\).)
\end{notn}

For example, the reader can check directly that
\((\gg_{x, r})^\bullet = \gg_{x, \MPlus{(-r)}}\),
hence that
\((\gg_{x, \MPlus r})^\bullet = \gg_{x, -r}\),
for all \((x, r) \in \BBpts \times \R\).

For any lattices $\LL$ and $\MM$ in $\gg$,
we have an isomorphism
from
$\MM^\bullet/\LL^\bullet$ to
the Pontryagin dual $(\LL/\MM)\sphat$ of $\LL/\MM$,
given by
$\abmapto{X}{\indexmem{\chi_X\sss}}$,
where $\mapto{\chi_X\sss}Y{\AddChar(\langle X, Y\rangle)}$.
Suppose that for some $x\in \BBpts$ and positive $r\in \R$, we have that
$\gg_{x,r+} \subseteq \LL \subseteq \gg_{x,(r/2)+}$.
Lemma~\ref{lem:cayley}(\ref{item:bijpos})
implies that the Cayley transform $\cayley$
induces an isomorphism
$\abbimap{\LL/\gg_{x,r+}}{\cayley(\LL)/ G_{x,r+}}$
(which we will also denote by $\cayley$).
Thus, we have an isomorphism
\begin{equation}
\label{eq:duality}
\abbimap{\gg_{x,-r}/\LL^\bullet}{(\cayley(\LL)/G_{x,r+})\sphat}
\end{equation}
given by $\abmapto{X}{\chi_X\sss}$,
where
$$
\mapto{\chi_X\sss}{\cayley(Y)}{
	\AddChar\bigl(\Tr(X\cdot Y)\bigr)
}.
$$
In particular, we have isomorphisms
\begin{multline}
\label{eq:G-chars}
\abbimap{\gg_{x,-r}/\gg_{x,(-r)+}}{(G_{x,r}/G_{x,r+})\sphat} \\
\text{and}\quad
\abbimap{
	\gg_{x, \MPlus{(-r)}}/\gg_{x, \MPlus{(-r/2)}}
}{
	(G_{x, \MPlus{(r/2)}}/G_{x, \MPlus r})\sphat
}\,\,;
\end{multline}
and, similarly,
\begin{equation}
\label{eq:T-chars}
\abbimap{\ttt_{-r}/\ttt_{\MPlus{(-r)}}}{(T_{r}/T_{\MPlus r})\sphat}
\qandq
\abbimap{
	\ttt_{\MPlus{(-r)}}/\ttt_{\MPlus{(-r/2)}}
}{
	(T_{\MPlus{(r/2)}}/T_{\MPlus r})\sphat
}
\end{equation}
for any maximal torus $T$ with Lie algebra $\ttt$.

\subsection{Fourier transforms and orbital integrals}
\label{sec:mu-hat}

\begin{defn}
\label{defn:FT}
Fix a Haar measure \(\textup dv\) on \(V\).
The \term[Fourier transform (function)]{Fourier transform}
of a function \(f \in \Hecke(V)\)
is the function \(\hat f \in \Hecke(V)\) defined by
\[
\hat f(w)
= \int_V f(v)\AddChar(\langle v, w\rangle)\textup dv,
\quad w \in V.
\]
We say that \(\textup dv\) is \term{self-dual}
(with respect to the additive character \(\AddChar\)
and pairing \(\vform\))
if
\[
\Hat{\Hat f}(v) = f(-v)
\quad\text{for all \(f \in \Hecke(V)\) and \(v \in V\).}
\]
If \(T\) is a distribution on \(V\)
(i.e., a linear functional on \(\Hecke(V)\)),
then the
\term[Fourier transform (distribution)]{Fourier transform}
of \(T\) is the distribution \(\widehat T\) on \(V\) defined by
\[
\widehat T(f) = T(\hat f),\quad f \in \Hecke(V).
\]
\end{defn}

Every vector space supports a unique self-dual Haar measure.

We leave to the reader the verification that
the Fourier transform does, indeed, carry \(\Hecke(V)\)
to itself.  In fact, \(C(\LL/\MM)\) is carried to
\(C(\MM^\bullet/\LL^\bullet)\)
\cite{adler-debacker:bt-lie}*{p.~282}.

Now suppose that \(V = \hh\) is the Lie algebra of a
reductive subgroup \(H\) of \(G = \SL_2(\field)\).
Then we may, and do, take the pairing on \(V\) to be the
trace form, given by
\begin{equation}
\label{eq:trace-pairing}
\langle X, Y\rangle \ldef \Tr(X\dotm Y),
\quad X, Y \in \gg.
\end{equation}

\begin{defn}
\label{defn:mu-hat}
If \(X \in \hh\) is regular and semisimple,
say with \(T = C_H(X)\),
then there is an \(H\)-invariant measure \(\textup d\dot h\)
on \(H/T\)
(which we may also view as a measure on the
\(H\)-adjoint orbit of \(X\) in \(\hh\)).
The \term[orbital integral (distribution)]{orbital integral}
of \(X\) is the
distribution \(\mu^H_X\) on \(\hh\) defined by
\[
\mu^H_X(f) = \int_{H/T} f(\Ad(h)X)\textup d\dot h,
\quad f \in \Hecke(\hh).
\]
As a special case of Definition \ref{defn:FT}, we define
\[
\hat\mu^H_X(f) \ldef \mu^H_X(\hat f)
= \int_{H/T} \hat f(\Ad(h)X)\textup d\dot h,
\quad f \in \Hecke(\hh).
\]
\end{defn}

The integral defining \(\mu^H_X\)
converges because the \(H\)-adjoint orbit of
\(X\) is closed in \(\hh\).
(In fact, it can be shown to converge under weaker conditions
\cite{ranga-rao:orbital}*{Theorem 2};
but we do not need this fact.)

Let \(\textup dY\) be the self-dual Haar measure on \(\hh\).
By \cite{hc:queens}*{Theorem 1.1},
there is a
\term[orbital integral (Fourier transform, function)]\relax
function on \(\hh\), which we shall again denote
by \indexmem{\hat\mu^H_X}, such that
\[
\hat\mu^H_X(f)
= \int_\hh f(Y)\hat\mu^H_X(Y)\textup dY,
\quad f \in \Hecke(\hh).
\]

Although we have specified a choice of Haar measure on
\(\hh\), it is unimportant here, since the \emph{function}
\(\hat\mu^H_X\) does not depend on the choice
(although the \emph{distribution} does).
The function \emph{does} depend on the normalization chosen
for the measure \(\textup d\dot h\) on \(H/T\)
(as well as on \(\AddChar\)), so we
shall be careful to specify this normalization.

The case \(\bH = \bG\) is covered in
\cite{spice:sl2-mu-hat}.  The only other case of interest to
us is handled by the lemma below.

\begin{lemma}
\label{lem:torus-mu-hat}
If \(\bH = \bT\) is a torus,
and
\(\meas_{\textup d\dot h}(H/T) = 1\),
then
\[
\hat\mu^T_X(Y) = \AddChar(\langle X, Y\rangle),
\quad X, Y \in \ttt.
\]
In particular, \(\hat\mu^T_X(0) = 1\).
\end{lemma}

Note that, in this setting, \(H/T\) is a singleton.

\begin{proof}
Since \(\mu^T_X(f) = f(X)\) for all
\(f \in \Hecke(\ttt)\) and \(X \in \ttt\),
this follows immediately from Definition \ref{defn:mu-hat}.
\end{proof}

\section{Unrefined minimal $K$-types}
\label{sec:ktypes}

\begin{defn}[\cite{moy-prasad:k-types}*{Definition 5.1}]
\label{defn:k-type}
Suppose $x\in\BBpts$, $r \geq 0$, and let $\chi$ be an irreducible
representation of $G_{x,r}$, trivial on $G_{x,r+}$.
We say that the pair $(G_{x,r},\chi)$ is
an \emph{(unrefined) minimal $K$-type of depth $r$} if
\begin{enumerate}[(a)]
\item
$r=0$ and $\chi$ is the inflation to $G_{x, 0}$ of a cuspidal representation
of $\sfG_x(\resfld) = G_{x, 0:\MPlus0}$; or
\item
\label{defn:k-type:pos}
$r>0$ and the coset $\Sigma\in\gg_{x,-r}/\gg_{x,(-r)+}$
corresponding to $\chi$
(see \S\ref{sec:duality})
contains no nilpotent elements.
\end{enumerate}
Two minimal $K$-types $(G_{x,r},\chi)$
and $(G_{y,s},\xi)$ are \emph{associate} if $r=s$ and
\begin{enumerate}[(a)]
\item
$r = 0$, $x = y$, and $\chi$ is equivalent to $\xi$;
or
\item
$r>0$, and the $G$-orbit of the coset that realizes $\chi$
intersects the coset that realizes $\xi$.
\end{enumerate}
\end{defn}

For arbitrary reductive, \(p\)-adic groups, one must call
depth-zero \(K\)-types associate even under some
circumstances when \(x \ne y\); but working with
\(\bG = \SL_2\), and restricting to \(x \in \BBpts\),
avoids this complication.

\begin{thm}[%
\cite{moy-prasad:k-types}*{Theorem 5.2}
and
\cite{moy-prasad:jacquet}*{Theorem 3.5}%
]
\label{thm:mp}
Let $(\pi,V)$ be an irreducible admissible representation of $G$.
Then there is a non-negative, rational number $r$ with the
following properties:
\begin{enumerate}[(1)]
\item
For some $x\in\BBpts$, the space $V^{G_{x,r+}}$ of $G_{x,r+}$-fixed
vectors is non-zero, and $r$ is the smallest non-negative
real number with this property.
\item
For any $y\in\BBpts$, if $W\ldef V^{G_{y,r+}}\neq\{0\}$, then
\begin{enumerate}[(a)]
\item
\label{thm:mp:zero}
if $r=0$, then every irreducible $G_{y,r}$-submodule of $W$
contains an unrefined minimal $K$-type of depth zero.
\item
\label{thm:mp:pos}
if $r>0$, then every irreducible $G_{y,r}$-submodule of $W$
is an unrefined minimal $K$-type.
\end{enumerate}
\end{enumerate}
Moreover, any two unrefined minimal $K$-types contained in $\pi$
are associate.
\end{thm}

\begin{defn}
\label{defn:depth-rep}
The number $r$ in Theorem \ref{thm:mp}
(denoted by \(\rho(\pi)\) in \cite{moy-prasad:k-types})
is called the \term[depth (representation)]{depth}
\(\depth(\pi)\) of $\pi$.
\end{defn}

\begin{rem}
\label{rem:K-type-leftist}
If the representation \(\pi\) of \(G\)
contains an unrefined minimal \(K\)-type of the
form \((G_{\xright, 0}, \chi\textsub R)\),
then \(\pi \circ \Int\begin{smallpmatrix}
1 & 0      \\
0 & \varpi
\end{smallpmatrix}\) contains the unrefined minimal \(K\)-type
\((G_{\xleft, 0}, \chi\textsub L)\),
where
\(\chi\textsub L
= \chi\textsub R \circ \Int\begin{smallpmatrix}
	1 & 0      \\
	0 & \varpi
\end{smallpmatrix}\).
\end{rem}

In order to classify the representations of $G$,
we start by listing the unrefined minimal
$K$-types for $G$,
and check which items on the list are associate.

\section{Representations of depth zero}
\label{sec:depthzero}

\subsection{Cuspidal representations of
	\protect{$\SL_2(\resfld)$}}
\label{sec:finite-params}

Write \(\sfG = {\SL_2}_{/\resfld}\).
The torus \(\sfT^\epsilon\) of \S\ref{sec:std-tori} is,
up to \(\SL_2(\resfld)\)-conjugacy,
the unique maximal elliptic \resfld-torus in
\(\sfG\);
and its (relative) Weyl group
is \(\sset{1, \sigma_\epsilon}\),
where \(\sigma_\epsilon(\psi) = \psi\inv\)
for any character \(\psi\) of \(\sfT^\epsilon(\resfld)\).

\begin{defn}
\label{defn:DL}
For any character $\psi$ of $\sfT^\epsilon(\resfld)$
such that $\psi\neq\psi\inv$,
we have from
\cite{deligne-lusztig:finite}*{Theorems 6.8 and 8.3}
that the Deligne--Lusztig virtual representation
$R^\sfG_{\sfT^\epsilon,\psi}$
is irreducible and cuspidal.
Let
$\smabs{R^\sfG_{\sfT^\epsilon,\psi}} = -R^\sfG_{\sfT^\epsilon, \psi}$ 
denote the corresponding
(non-virtual) representation.
\end{defn}

\begin{defn}
\label{defn:exc-rep-finite}
By \cite{deligne-lusztig:finite}*{Theorem 6.8},
\(R^\sfG_{\sfT^\epsilon, \psi_0}\)
is a sum of two inequivalent, irreducible (virtual)
representations, which we will denote by
\(R^\pm_{\sfT^\epsilon, \psi_0}\).
By \cite{fulton-harris:rep-theory}*{pp.~70--73},
where \(R^\sfG_{\sfT^\epsilon, \psi_0}\)
is denoted by \(X_{\psi_0}\) and its two components by
\(X'\) and \(X''\), we may give an explicit description of
the virtual representations.
It is convenient to choose signs in such a way that
\begin{equation}
\label{eq:Rpm}
R^\pm_{\sfT^\epsilon, \psi_0}
= \tfrac1 2 R^\sfG_{\sfT^\epsilon, \psi_0} \mp
\tfrac1 2 q^{1/2}\Gauss(\AddChar)\inv\dotm\lsup0f,
\end{equation}
where \(\lsup0f\) is as in Definition \ref{defn:lusztig}
below, and, as usual,
we have identified the finite-group representation
\(R^\pm_{\sfT^\epsilon, \psi_0}\) with its character.
\end{defn}

\begin{rem}
\label{rem:exhaust-and-equiv-finite}
An explicit computation shows that
\[
\sum_u \lsup0f(u) = 0,
\]
where the sum runs over \(\set{\begin{smallpmatrix}
1 & n \\
0 & 1
\end{smallpmatrix}}{n \in \resfld}\).  By
\eqref{eq:Rpm},
\cite{deligne-lusztig:finite}*{(8.3.2)},
and \cite{carter:finite}*{Corollary 9.1.2},
\(R^\pm_{\sfT^\epsilon, \psi_0}\) is cuspidal.

By
\cite{deligne-lusztig:finite}*{Proposition 8.2}
and
\cite{carter:finite}*{p.~457},
all irreducible, cuspidal representations of $\SL_2(\resfld)$
arise in this way.
Moreover, by \cite{deligne-lusztig:finite}*{Theorem 6.8},
if \(\psi\) and \(\psi'\)
satisfy \(\psi \ne \psi\inv\)
and \(\psi' \ne \smash{\psi'}\inv\), then
$\smabs{R^\sfG_{\sfT^\epsilon,\psi}}$
is not isomorphic to
\(\smabs{R^\pm_{\sfT^\epsilon,\psi_0}}\),
and is isomorphic to
$\smabs{R^\sfG_{\sfT^\epsilon,\psi'}}$
if and only if $\psi'=\psi^{\pm1}$.
\end{rem}

\subsection{Lifting finite-field representations to
	depth-zero representations}

\begin{prop}
\label{prop:depthzero-sc-induced}
Let $\pi$ be an irreducible,
depth-zero,
supercuspidal representation of $G$.
Then $\pi$ contains an unrefined minimal $K$-type 
$(G_{x,0},\sigma)$,
where \(x \in \sset{\xleft, \xright}\),
and $\pi$ is equivalent to
$\Ind_{G_{x,0}}^G\sigma$.
\end{prop}

\begin{proof}
This follows from \cite{moy-prasad:jacquet}*{Proposition 6.6}
upon noting that \(G_{x, 0}\) is self-normalizing in \(G\)
for \(x \in \sset{\xleft, \xright}\).
\end{proof}

Suppose \(x \in \sset{\xleft, \xright}\),
so that \(\sfG_x(\resfld) \iso \SL_2(\resfld)\).
There is an unramified, elliptic torus \(T\) in \(G\)
(unique up to \(G_{x, \MPlus0}\)-conjugacy)
such that the image of \(T_0 \subseteq G_{x, 0}\)
in \(G_{x, 0:\MPlus0}\) is \(\sfT^\epsilon(\resfld)\)
\cite{debacker:tori-unram}*{Lemma 2.2.2}.
Specifically, we take \(T = T^{\epsilon, \eta}\),
where \(\eta = 1\) if \(x = \xleft\)
and \(\eta = \varpi\) if \(x = \xright\).
Thus, a character of
\(\sfT^\epsilon(\resfld) \iso T_{0:\MPlus0} = T/T_{\MPlus0}\)
may be viewed in a natural way as a character of \(T\) trivial on
\(T_{\MPlus0}\), i.e., a depth-zero character
(in the sense of Definition \ref{defn:depth-char}).
Note that the character \(\psi_0\) of \(\sfT^\epsilon\)
inflates to the character \(\psi_0^\eta\) of \(T\)
(with notation as in Notation \ref{notn:torus-quad}).

\begin{defn}
\label{defn:depth-0-param}
A \term{depth-zero, supercuspidal parameter} is
\begin{itemize}
\item
a pair \((T, \psi)\),
where \(T = T^{\epsilon, \eta}\),
with \(\eta \in \sset{1, \varpi}\),
and \(\psi\) is a depth-zero character of \(T\)
such that \(\psi \ne \psi\inv\);
or
\item
a triple \((T, \psi, \pm)\),
where \(T = T^{\epsilon, \eta}\)
with \(\eta \in \sset{1, \varpi}\),
and
\(\psi = \psi_0^\eta\).
\end{itemize}
Given such a datum, put
\[
\bPi(T, \psi) = \Ind_{G_{x, 0}}^G \smabs{R^G_{T, \psi}};
\quad\text{respectively, }
\bPi^\pm(T, \psi)
= \Ind_{G_{x, 0}}^G \smabs{
	R^\pm_{T, \psi}
},
\]
where \(\smabs{R^G_{T, \psi}}\)
(respectively,
\(\smabs{R^\pm_{T, \psi}}\))
is the inflation to \(G_{x, 0}\) of the appropriate
finite-field representation.
We call the various \(\bPi(T, \psi)\)
(and all the positive-depth, supercuspidal representations,
which we will construct later)
\term{ordinary} (see \S\ref{sec:ordinary}),
and the four possible \(\bPi^\pm(T, \psi)\)
\term{exceptional} (see \S\ref{sec:exceptional}).
\end{defn}

The distinction between `ordinary' and `exceptional' is just
an \emph{ad hoc} one reflecting the different techniques needed
in their character computations.

The following \textit{ad hoc} definition allows us to state
Proposition \ref{prop:MK} uniformly.
All that is important for us is that
\begin{itemize}
\item
\(X_\pi \in \gg_{x, 0}\);
\item
the image \(\overline{X_\pi}\) of \(X_\pi\) in
\(\gg_{x, 0:\MPlus0} = \Lie(\sfG_x)(\resfld)\) satisfies
\(C_{\sfG_x}(\overline{X_\pi}) = \sfT^\epsilon\);
and
\item
\(C_\bG(X_\pi) = \bT^{\epsilon, \eta}\).
\end{itemize}
(Recall that \(\eta = 1\) if \(x = \xleft\),
and \(\eta = \varpi\) if \(x = \xright\).)

\begin{notn}
\label{notn:depth-0-X}
If \(\pi\) is a
depth-zero, supercuspidal representation, then we write
\(\indexmem{X_\pi} \ldef X^{\epsilon, \eta}_1\).
Put \(\SSAddChar_\pi = \AddChar\).
\end{notn}

\begin{rem}
\label{rem:depth-0-leftist}
Note that, if \((T, \psi)\)
(respectively, \((T, \psi, \pm)\))
and \((T', \psi')\)
(respectively, \((T', \psi', \pm)\))
are depth-zero, supercuspidal parameters,
and \(g \in \GL_2(\field)\) is such that
\(T' = \Int(g)T\) and \(\psi' = \psi \circ \Int(g)\), then
\(\pi' = \pi \circ \Int(g)\),
where \(\pi = \bPi(T, \psi)\) and \(\pi' = \bPi(T', \psi')\)
(respectively, \(\pi = \bPi^\pm(T, \psi)\)
and \(\pi' = \bPi^\pm(T', \psi')\)).
In particular,
if \(g \in G\), then
\(\pi(g)\) intertwines \(\pi\) and \(\pi'\).
Further, \(\Ad(g)X_\pi = X_{\pi'}\).

Note that this applies in particular when
\(T = T^{\epsilon, \varpi}\);
\(g = \begin{smallpmatrix}
\varpi & 0 \\
0      & 1
\end{smallpmatrix}\);
and
\(T' = T^{\epsilon, 1}\).
In this setting, \(x_{\bT'} = \xleft\).
\end{rem}

\section{Representations of positive depth}
\label{sec:posdepth}
\subsection{Unrefined minimal $K$-types of positive depth}
\label{sec:posktypes}

Now that we have classified the representations of depth zero,
we turn to those of positive depth
(Definition \ref{defn:k-type}\pref{defn:k-type:pos}).
Theorem \ref{thm:mp} suggests that we start by
classifying the unrefined minimal $K$-types that they contain.
As before, we may confine our attention
to the three filtrations associated to 
elements of $\BBpts$.
We begin by listing the $K$-types associated
to the filtration coming from $\xleft$.

Let $r\in\Z$,
since otherwise the quotient $\gg_{{\xleft},-r}/\gg_{{\xleft},\MPlus{(-r)}}$
is trivial.
The quotient $G_{\xleft, 0:\MPlus0} = \sfG_{\xleft}(\resfld)$
is isomorphic to $\SL_2(\resfld)$.
Every coset in $\gg_{\xleft,-r}/\gg_{\xleft,(-r)+}$
can be written in the form
$$
\varpi^{-r} X + \gg_{\xleft,(-r)+},
$$
where
$X$ has one of the following forms
(up to \(G_{\xleft, 0}\)-conjugacy):
$$
\begin{pmatrix}
0 & \beta \\ 0 & 0
\end{pmatrix},
\qquad
X^{\text{split}}_\beta =
\begin{pmatrix}
\beta & \phm 0 \\ 0 & -\beta
\end{pmatrix},
\qquad
X^{\epsilon,1}_\beta = 
\begin{pmatrix}
\phantom{\epsilon}0 & \beta \\ \epsilon\beta & 0
\end{pmatrix},
$$
where $\beta\in \pint\mult$.
Since $\beta$ is determined only modulo $\pp$,
we will think of it as lying in $\resfld \mult$.
In the first example above, $X$ is nilpotent,
so the corresponding coset does not correspond to
an unrefined minimal $K$-type.

Now consider the $K$-types arising from the Iwahori filtration
(i.e., the filtration associated to the point $\xcentre$).
If $r\in\Z$,
then every coset in the quotient space
$\gg_{\xcentre,-r} / \gg_{\xcentre,(-r)+}$
has the form
$\varpi^{-r} X^{\text{split}}_\beta + \gg_{\xcentre,\MPlus{(-r)}}$.
As before, we may take $\beta$ to lie in $\resfld\mult$.
Similarly, if $r\in\Z+\frac12$, then
every coset in the quotient space
$\gg_{\xcentre,-r} / \gg_{\xcentre,\MPlus{(-r)}}$
either contains a nilpotent element or
has the form (up to conjugation by $G_{\xcentre}$)
$\varpi^{-\rup r}
X^{\theta,\eta}_\beta +\gg_{\xcentre,(-r)+}$,
where $\theta\in\{\varpi,\epsilon\varpi\}$
and
$\eta\in\{1,\epsilon\}$.

Suppose
$\chi\unrsplit_{\beta,r}$
is the character of
$G_{\xleft,r}/G_{\xleft,r+}$ corresponding to
$\varpi^{-r}X^{\text{split}}_\beta+\gg_{\xleft,(-r)+}$,
and
$\chi^{\text{split}}_{\beta,r}$
is the character of
$G_{\xcentre,r}/G_{\xcentre,r+}$ corresponding to
$\varpi^{-r}X^{\text{split}}_\beta+\gg_{\xcentre,(-r)+}$.
(See \S\ref{sec:duality}.)
Then any representation $\pi$ of $G$ that contains
$\chi\unrsplit_{\beta,r}$
must contain
$\chi^{\text{split}}_{\beta,r}$,
as the latter is the restriction of the former
to $G_{\xcentre,r}$.
From \cite{yu:supercuspidal}*{Corollary 17.3},
a representation of $G$ that contains such an unrefined minimal
$K$-type
cannot be supercuspidal.
Therefore, we may ignore this family of $K$-types.

The situation for $\xright $ is the same as that for $\xleft$,
after conjugation by $\begin{smallpmatrix}
1 & 0      \\
0 & \varpi
\end{smallpmatrix}$.

Thus, every irreducible supercuspidal representation of $G$
of positive depth must contain
a $K$-type whose corresponding coset has the form
$$
\varpi^{-\rup r} X + \gg_{x,(-r)+},
$$
where
the possibilities for
\(X\), \(x\), and \(r\)
are as follows:
\begin{equation}
\label{eq:typelist}
\setlength{\extrarowheight}{.5em}
\begin{array}{lcl}
X & x & r \\
\hline
X_\beta^{\epsilon,1}
	& \xleft & \Z_{> 0} \\
X_\beta^{\epsilon,\varpi}
	& \xright  & \Z_{> 0} \\
X_\beta^{\theta,\eta}
	& \xcentre & \Z_{\ge 0}+\frac12
\end{array}
\end{equation}
Here, $\beta$ ranges over $\resfld\mult$,
$\theta\in\{\varpi,\epsilon\varpi\}$,
and $\eta\in\{1,\epsilon\}$.
Thus, we have six families of $K$-types, each parametrized
by $\resfld \mult$ and $\Z_{> 0}$ (or $\Z_{\ge 0}+\frac12$).

Next we will determine
which of the $K$-types in the six families above
are associate.

Given two cosets $\Sigma_1$ and $\Sigma_2$ of the form above,
when does there exist $g\in G$ such that
$\Int(g)\Sigma_1\cap\Sigma_2\neq \emptyset$?
In each case, one can answer this through
direct computation.
A less cumbersome method requires appealing to
a few general theorems, each of which
is easy to prove in our special situation.

\begin{prop}
\label{prop:assoc-conj}
For $i=1,2$,
let $\Sigma_i = \varpi^{-\rup{r_i}} X_i + \gg_{x_i,(-r_i)+}$
be a coset listed in \eqref{eq:typelist}.
Then $\Sigma_1$ and $\Sigma_2$ are associate if and only if
$r_1=r_2$, and
$X_1$ and $X_2$ are $G$-conjugate.
\end{prop}

\begin{proof}
If $r_1=r_2$
and
$X_1$ and $X_2$ are $G$-conjugate,
then 
$\varpi^{-\rup{r_1}} X_1$
is conjugate to
$\varpi^{-\rup{r_2}} X_2$,
so $\Sigma_1$ and $\Sigma_2$ are associate.

Conversely,
suppose that
$\Sigma_1$ and $\Sigma_2$ are associate.
Then $r_1=r_2$ by the definition of `associate'.
By
\cite{adler-roche:intertwining}*{Proposition 9.3},
$\varpi^{-\rup{r_1}} X_1$
is conjugate to
$\varpi^{-\rup{r_2}} X_2$.
Therefore,
$X_1$ and $X_2$ are conjugate.
\end{proof}

Thus, we only need to check which of
our chosen coset representatives in \eqref{eq:typelist}
are conjugate.
From easy calculations, we have the following conjugacy relations,
and no other non-trivial ones:
\begin{equation}
\label{eq:conjlist}
\begin{aligned}
X_\beta^{\epsilon,\eta} &\sim X_{-\beta}^{\epsilon,\eta}
	&& (\eta \in \sset{1,\varpi}) \quad
	&& \text{via an element of $G_x$} \\
X_\beta^{\theta,\eta} &\sim X_{-\beta}^{\theta,\eta}
	&& (\theta \in \sset{\varpi,\epsilon\varpi},
		\quad\eta \in \sset{1,\epsilon})\quad
	&& \text{if and only if $q\equiv 1 \pmod 4$} \\
X_\beta^{\theta,1} &\sim X_{-\beta}^{\theta,\epsilon}
	&& (\theta \in \sset{\varpi,\epsilon\varpi})\quad
	&& \text{if and only if $q\equiv 3 \pmod 4$.}
\end{aligned}
\end{equation}

\subsection{From positive-depth $K$-types to inducing data}
\label{sec:posreps}

Fix a triple $(X,x,r)$ from the list in~\eqref{eq:typelist},
let $\Sigma$ be the corresponding coset,
and let
$\chi_\Sigma\sss = \chi_X\sss\in(G_{x,r}/G_{x,r+})\sphat\,\,$ 
be the character corresponding to $\Sigma$
(see \S\ref{sec:duality}).
We will describe all irreducible representations of $G$
that contain $(G_{x,r},\chi_\Sigma\sss)$.

Let $T=C_G(X)$,
and $\ttt=C_\gg(X)=\Lie(T)$.
These are independent of $\beta$.

Let
$$
\ttt^\perp = \sett{Y\in\gg}{$\Tr(Y\dotm Z)=0$ for all $Z\in\ttt$},
$$
and
$\ttt^\perp_s = \ttt^\perp\cap\gg_{x,s}$.
Then it is easy to verify that for all $s\in\R$,
$$
\gg_{x,s} = \ttt_s\sss \oplus \ttt^\perp_s
\qandq
\gg_{x, \MPlus s} = \ttt_{\MPlus s}\sss \oplus \ttt^\perp_{\MPlus s}.
$$

Define lattices $\JJ$ and $\JJ_+$ in $\gg$ by
$$
\JJ = \ttt_r+\ttt^\perp_{r/2}
\qandq
\JJ_+ = \ttt_r+\ttt^\perp_{(r/2)+}.
$$
By Lemma \ref{lem:cayley}(%
	\ref{item:bijpos},
	\ref{item:MP-map}%
), the images
\(J = \cayley(\JJ)\) and \(J_+ = \cayley(\JJ_+)\)
of these lattices under the Cayley transform
are groups,
and we have isomorphisms
\(\abbimap{J/G_{x, r}}{\JJ/\gg_{x, r}}\)
and
\(\abbimap{J_+/G_{x, \MPlus r}}{\JJ_+/\gg_{x, \MPlus r}}\),
which we will again denote by \(\cayley\).
In particular, by \eqref{eq:duality},
\begin{equation}
\label{eq:J-duality}
(J_+/G_{x, \MPlus r})\sphat\
\iso (\JJ_+/\gg_{x, \MPlus r})\sphat\
\iso \gg_{x, -r}/(\ttt\sss_{-r} + \ttt^\perp_{-r/2}).
\end{equation}

Since $\chi_\Sigma$ is trivial on $\ttt^\perp_r$, we may extend it to
a character $\bar\chi$ of $\JJ_+/\gg_{x,r+}$ (or $J_+/G_{x,r+}$)
by setting
$\bar\chi$ to be trivial on $\ttt^\perp_{(r/2)+}$.
By inflation, we may regard $\chi_\Sigma$ as a character of $G_{x,r}$,
and $\bar\chi$ as a character of $J_+$.
Explicitly,
\begin{equation}
\label{eq:ol-chi}
\bar\chi\bigl(\cayley(Y)\bigr)
= \AddChar\bigl(\Tr(X\dotm Y)\bigr),
\quad Y \in \JJ_+.
\end{equation}
In terms of \eqref{eq:J-duality},
this extension--inflation process corresponds to following
the projection
\[
\abmap{
	\gg_{x, -r}/(\ttt\sss_{-r} + \ttt^\perp_{-r/2})
}{
	\gg_{x, -r:-r/2} \iso (G_{x, \MPlus{(r/2)}:\MPlus r})\sphat
}\,\,.
\]

\begin{prop}
Any representation of $G$ that contains $\chi$
must contain $\bar\chi$.
\end{prop}

\begin{proof}
This is a special case of Corollary 6.5 of \cite{adler-roche:intertwining},
though earlier versions exist.
\end{proof}

We have shown that in order to classify
the irreducible representations of $G$ that contain $\chi$,
it is enough to classify the irreducible representations
containing each character $\bar\chi$.

\begin{prop}
\label{prop:heisenberg}
There exists a unique irreducible representation $\rho_\Sigma$ of $J$
that contains $\bar\chi$.
Moreover, $\res{\rho_\Sigma}to{J_+}$ is a sum of $\indx J{J_+}^{1/2}$
copies of $\bar\chi$.
The character of $\rho_\Sigma$ is given by
$$
\Theta_{\rho_\Sigma}(g) =
\begin{cases}
\indx J{J_+}^{1/2}\bar\chi(g) & g\in J_+ \\
0 & g \in J \setminus J_+.
\end{cases}
$$
\end{prop}

The proof will require two lemmas along the way.
If $J=J_+$, then there is nothing to prove,
so assume that $J\neq J_+$.
Define an alternating form $\vform$ on $J/J_+$
by $\langle a,b\rangle=\bar\chi([a,b])$
for all \(a, b \in J\).

\begin{lemma}[%
	special case of \cite{adler:thesis}*{Lemma 2.6.1}%
]
\label{lem:nondegen}
The form $\vform$ is non-degenerate.
\end{lemma}

Let $N=\ker\bar\chi_\Sigma$.
It follows from Lemma~\ref{lem:nondegen}
that $J/N$ is a two-step nilpotent group,
and that its center and derived group are both
$J_+/N$.  
The representation theory of such groups is well known.  
(The corresponding result over \(\R\) is called the
Stone--von Neumann theorem.)

\begin{lemma}[%
	\cite{gerardin:weil}*{Lemma 1.2}%
]
\label{lem:heisenberg}
Let $H$ be a finite group, let $A$ be the center of $H$,
and suppose that $A$ is also the derived group of $H$.
Let $\xi$ be a non-trivial character of $A$.
Then there exists a unique (up to equivalence) representation
$\rho_\xi$ of $H$ with central character $\xi$.
Moreover, $\dim\rho_\xi = \indx H A^{1/2}$,
and the character of $\rho_\xi$ is supported on $A$.
\end{lemma}

\begin{proof}[Proof of Proposition~\ref{prop:heisenberg}]
This follows from Lemma~\ref{lem:heisenberg},
setting $H=J/N$, $A=J_+/N$, and $\xi=\bar\chi$.
\end{proof}

\begin{cor}
Every representation of $G$ that contains $\chi_\Sigma$
also contains $\rho_\Sigma$.
\end{cor}

Thus, in order to determine the postive-depth, supercuspidal
representations,  it is enough to classify the irreducible
representations of $G$ that contain
$\rho_\Sigma$.  We start by
noting that $T$ normalizes $J$,
and classify the irreducible representations of
$TJ$ that contain $\rho_\Sigma$.

\begin{prop}
\label{prop:extension}
The representation $\rho_\Sigma$ extends to $TJ$.
There is an explicit bijection 
between the set of 
characters of $T$ that contain $\res{\chi_\Sigma}to{T_r}$
and the set of such extensions.
\end{prop}

If $J=J_+$, then $TJ/N\iso T/T_r$, so there
is nothing to prove.
The only case for which $J\neq J_+$ is where
$T$ is unramified and $r$ is even.
Since we are interested in computing characters explicitly,
it will be convenient
to imitate a method of Moy~\cite{moy:thesis-original} instead.
We defer this to \S\ref{sec:inducing}.  

\begin{rem}
\label{rem:inducing-data}
For future reference, we note that the group $TJ$ is equal
to $TG_{x,s}$; and that $\rho_\Sigma$
has dimension $q$ if
$r \in 2\Z$,
and dimension $1$ otherwise.
\end{rem}

\subsection{Constructing positive-depth, supercuspidal
	representations}
\label{sec:param}

Remember that we have constructed representations
\(\bPi(T^{\epsilon, \eta}, \psi)\),
where \(\psi\) is a depth-zero character of
\(T^{\epsilon, \eta}\) such that \(\psi \ne \psi\inv\),
and \(\bPi^\pm(T^{\epsilon, \eta}, \psi_0^\eta)\)
(see Definition \ref{defn:depth-0-param}).
We now complete our construction of the supercuspidal
representations of \(G\) by defining representations
\(\bPi(T, \psi)\) when \(T\) is any
maximal, standard, elliptic torus,
and \(\psi\) is a positive-depth character of \(T\).

\begin{defn}
\label{defn:pos-depth-param}
A \term{positive-depth, supercuspidal parameter}
is a pair \((T, \psi)\),
where \(T = T^{\theta, \eta}\) is a standard torus
and \(\psi\) is a positive-depth character of \(T\).
Given such a parameter, put \(r = \depth(\psi)\).
Using Lemma \ref{lem:cayley}\pref{item:MP-map}
and \eqref{eq:duality},
we may deduce from the restriction to \(\psi\) of
\(T_{\MPlus{(r/2)}}\) an element of
\begin{equation}
\label{eq:torus-iso}
(T_{\MPlus{(r/2)}:\MPlus r})\sphat\
\cong (\ttt_{\MPlus{(r/2)}:\MPlus r})\sphat\
\cong \ttt_{-r:-r/2}.
\end{equation}
Thus, there exists \(\beta \in \pint\mult\)
(uniquely determined only modulo some power of the prime ideal)
such that
\begin{equation}
\label{eq:psi-log}
\psi\bigl(\cayley(Y)\bigr)
= \AddChar\bigl(\Tr(
	\varpi^{-\rup r}X^{\theta, \eta}_\beta\dotm Y
)\bigr)
\quad\text{for all \(Y \in \ttt_{\MPlus{(r/2)}}\).}
\end{equation}
We then put
\(\Sigma
= \varpi^{-\rup r}X^{\theta, \eta}_\beta + \gg_{x, \MPlus{(-r)}}\),
which depends only on the image of
\(\beta\) in \(\resfld\mult\),
and construct
\(\chi_\Sigma\),
\(J\), \(J_+\), \(\bar\chi\),
and \(\rho_\Sigma\) as in \S\ref{sec:posreps}.
By Proposition \ref{prop:extension}, the extensions of
\(\rho_\Sigma\) to \(T J\) are parametrized by characters of 
\(T\) containing \(\res{\chi_\sigma}to{T_r}\).
By definition, \(\psi\) is such a character, hence affords
an extension \indexmem{\sigma(T, \psi)} of \(\rho_\Sigma\)
to \(T J\).
Put \(\bPi(T, \psi) \ldef \Ind_{T J}^G \sigma(T, \psi)\).
The positive-depth, supercuspidal representations
constructed this way (as well as some of the depth-zero,
supercuspidal representations constructed in Definition
\ref{defn:depth-0-param}) are called \term{ordinary}.
\end{defn}

\begin{rem}
By a direct computation
(or, given Proposition \ref{prop:yu-equiv}, by
\cite{yu:supercuspidal}*{Remark 3.6}),
\(r = \depth(\psi)\) is also the depth of \(\pi\), in the
sense of Definition \ref{defn:depth-rep}.
\end{rem}

\begin{notn}
\label{notn:pos-depth-X}
In the notation of Definition \ref{defn:pos-depth-param}
(specifically, \eqref{eq:torus-iso}),
if \(\pi = \bPi(T, \psi)\), then put
\(X_\pi = \varpi^{-\rup r}X^{\theta, \eta}_\beta\)
and
\(\SSAddChar_\pi
= \AddChar_{\varpi^{-\rup r}\beta\theta}\).
\end{notn}

\begin{prop}
\label{prop:sc}
Let \((T, \psi)\) be a positive-depth, supercuspidal
parameter.
Then \(\bPi(T, \psi)\) is supercuspidal;
and an irreducible, smooth representation
$\pi$ of $G$ contains $\sigma(T, \psi)$ if and only if
it is equivalent to $\bPi(T, \psi)$.
\end{prop}

\begin{proof}
From \cite{adler:thesis}*{\S2.5}
or \cite{yu:supercuspidal}*{Theorem 15.1},
$\bPi(T, \psi)$ is irreducible. 
Mautner observed~\cite{mautner:spherical-2}*{Theorem 9.1}
that it is therefore supercuspidal if
it has a non-zero, compactly supported matrix coefficient.
The function
\begin{equation*}
\abmapto g{
\begin{cases}
\langle \sigma (\psi) (g) v, w \rangle & \qquad \text{$g \in TJ$}\\
0 & \qquad \text{$g \in G \setminus TJ$,}
\end{cases}
}
\end{equation*}
where $v$ and $w$ are any non-zero vectors in the space of
\(\psi\),
and
$\vform$ is a non-trivial
$TJ$-invariant pairing,
is one such matrix coefficient.

By Frobenius reciprocity, any irreducible smooth representation
of $G$ that contains $\sigma(T, \psi)$ is equivalent to
$\bPi(T, \psi)$.
\end{proof}

\begin{rem}
\label{rem:pos-depth-leftist}
As in Remark \ref{rem:depth-0-leftist}, we observe that, if
\((T, \psi)\) and \((T', \psi')\) are positive-depth,
supercuspidal parameters, and \(g \in \GL_2(\field)\)
is such that \(T' = \Int(g)T\)
and \(\psi' = \psi \circ \Int(g)\),
then \(\pi' = \pi \circ \Int(g)\),
where \(\pi = \bPi(T, \psi)\)
and \(\pi' = \bPi(T', \psi')\).
In particular, if \(g \in G\), then \(\pi(g)\) intertwines
\(\pi\) and \(\pi'\).
Further, all relevant data (\(\chi\), \(J\), etc.\@) behave well
with respect to the conjugation,

Note that this applies in particular
when \(T = T^{\theta, \eta}\);
with \(\theta = \epsilon\) and \(\eta = \varpi\)
or \(\theta \in \sset{\varpi, \epsilon\varpi}\)
	and \(\eta = \epsilon\);
\(g = \begin{smallpmatrix}
\eta & 0 \\
0    & 1
\end{smallpmatrix}\); and \(T' = T^{\theta, 1}\).
In this setting, \(x_{\bT'} \in \sset{\xleft, \xcentre}\).
\end{rem}

\section{Parametrization of supercuspidal representations}

\begin{thm}
\label{thm:sc}
Every supercuspidal representation of \(G\) is of the form
\(\bPi(T, \psi)\) or \(\bPi^\pm(T, \psi)\)
for some (depth-zero or positive-depth)
supercuspidal parameter \((T, \psi)\) or \((T, \psi, \pm)\).

The only non-trivial isomorphisms among supercuspidal
representations are as follows.
We have \(\bPi(T, \psi) \isorep \bPi(T, \psi\inv)\)
if
\begin{itemize}
\item
\(T = T^{\epsilon, \eta}\) with \(\eta \in \sset{1, \varpi}\),
or
\item
\(T = T^{\theta, \eta}\)
with \(\theta \in \sset{\varpi, \epsilon\varpi}\)
and \(\eta \in \sset{1, \epsilon}\),
and
\(q \equiv 1 \pmod4\).
\end{itemize}
If \(q \equiv 3 \pmod4\), then we have
\(\bPi(T^{\theta, 1}, \psi^1)
\isorep \bPi(T^{\theta, \epsilon}, \psi^\epsilon)\),
where \(\theta \in \sset{\varpi, \epsilon\varpi}\)
and
\(\psi^\epsilon = \psi^1 \circ \Int\begin{smallpmatrix}
	\sqrt{-\epsilon} & 0                    \\
	0                & \sqrt{-\epsilon}\inv
\end{smallpmatrix}\).
\end{thm}

\begin{proof}
By Proposition \ref{prop:depthzero-sc-induced} and
\S\ref{sec:finite-params}
(in the depth-zero case),
and Proposition \ref{prop:sc} and \S\ref{sec:posktypes}
(in the positive-depth case),
we have the desired exhaustion.

We now identify equivalences.
That all the stated equivalences hold follows from
\eqref{eq:conjlist} and
Remarks \ref{rem:depth-0-leftist}
and \ref{rem:pos-depth-leftist}.
To show that there are no others, we
use Theorem \ref{thm:mp}.

No depth-zero, supercuspidal
representation is equivalent to any positive-depth,
supercuspidal representation.

That the stated isomorphisms among depth-zero
representations are the only ones follows from
\S\ref{sec:finite-params}.

Now suppose that \((T, \psi)\) and \((T', \psi')\)
are two positive-depth, supercuspidal parameters,
and put \(\pi = \bPi(T, \psi)\)
and \(\pi' = \bPi(T', \psi')\).
Let the data \(X\), \(X'\) and \(\psi\), \(\psi'\)
be as in Definition \ref{defn:pos-depth-param}.

We have that \(\pi \isorep \pi'\) only if \(X\) is 
\(G\)-conjugate to \(X'\).
(Remember that \(X\) and \(X'\) are well defined only
modulo some Moy--Prasad filtration lattice; so we are really
claiming that there exist \(G\)-conjugate elements in the
appropriate cosets.)
By Remark \ref{rem:pos-depth-leftist},
it therefore suffices to assume that
\(\pi \isorep \pi'\) and \(X = X'\),
and show that
$\psi = \psi'$.
Note (under our assumption) that \(T = T'\);
that the \(J\)-groups for \(\pi\) and \(\pi'\)
are the same,
as are the \(J_+\)-groups;
and, by \eqref{eq:ol-chi},
that \(\bar\chi = \bar\chi'\).

Now
$$
\Hom_G(\pi,\pi')
=
\bigoplus_{g\in TJ \backslash G/TJ}
	\Hom_{\lsup{g\cap}(TJ)}\bigl(
		\sigma(T, \psi),
		\sigma(T, \psi') \circ \Int(g)\inv
\bigr),
$$
where \(\lsup{g\cap}(T J) \ldef T J \cap \Int(g)(T J)\).
Since the left-hand side is non-zero,
there exists \(g \in G\) such that the corresponding summand
on the right-hand side is non-zero;
i.e., \(g\) intertwines \(\sigma(T, \psi)\)
and \(\sigma(T', \psi')\).
Since \(\res{\sigma(T, \psi)}to{J_+}\) is \(\bar\chi\)-isotypic
and \(\res{\sigma(T, \psi')}to{J_+}\) is \(\bar\chi'\)-isotypic,
$g$ intertwines $\bar\chi$ and $\bar\chi'$;
that is,
$\bar\chi = \bar\chi' \circ \Int(g)\inv$
on $\lsup{g\cap}(TJ)$.
Since $\bar\chi = \bar\chi'$,
a calculation shows that we must have  $g\in JTJ = TJ$.
Then \(\lsup{g\cap}(T J) = T J\), so the irreducible
representations $\sigma(T, \psi)$ and $\sigma(T, \psi')$
intertwine, and hence are equivalent.
It follows from Proposition \ref{prop:extension}
that $\psi=\psi'$.
\end{proof}

\section{Inducing representations}
\label{sec:inducing}

Recall from Proposition \ref{prop:heisenberg}
that a character $\psi'$ of $T$ of depth $r>0$
uniquely determines a representation $\rho_\chi$ of $J$
(there denoted by \(\rho_\Sigma\)),
where $\chi = \res{\psi'}to{T_r}$
is the restriction to $T_r$ of $\psi'$.
In this section,
imitating a method of Moy \cite{moy:thesis-original},
we give an explicit bijection 
between the set of characters of $T$ that contain $\chi$
and the set of irreducible representations of $TJ$ that contain
$\rho_\chi$,
thus proving Proposition~\ref{prop:extension}.
We also show that the resulting parametrization of representations
of $TJ$ agrees with that in Yu's construction (\cite{yu:supercuspidal}).

Suppose first that $J=J_+$.
As remarked before, there is nothing to prove in this case:
given a character $\psi$ 
of $T$ that extends $\chi$,
we clearly have a corresponding character of $TJ$,
since $TJ/N\iso T/T_r$
(where $N=\ker\bar\chi$, as before).

The character of \(T J\) is easy to describe.
As in Definition \ref{defn:pos-depth-param},
given \(\psi\),
there exists $X\in\ttt_{-r}$ (well defined modulo $\ttt_{-r/2}$)
such that
$$
\psi\bigl(\cayley(Y)\bigr) = \AddChar\bigl(\Tr(X\cdot Y)\bigr)
\qquad\text{for all $Y\in \ttt_{\MPlus{(r/2)}}$}.
$$
The character $\sigma(T, \psi)$ of $TJ$ that corresponds to $\psi$
is then given by the formula
\begin{equation}
\label{eq:char-noheisenberg}
\abmapto{t\dotm\cayley(Y)}{
	\psi(t)\cdot \AddChar\bigl(\Tr(X\cdot Y)\bigr)
},
\qquad t\in T,\, Y\in \JJ.
\end{equation}

For the rest of this section,
assume that $J\neq J_+$.
That is,
$T$ is unramified
and $r$ is even.
Note that then \(\indx J{J_+} = q^2\).

Put \(Z = Z(G)\).
Consider the following diagram of four groups:
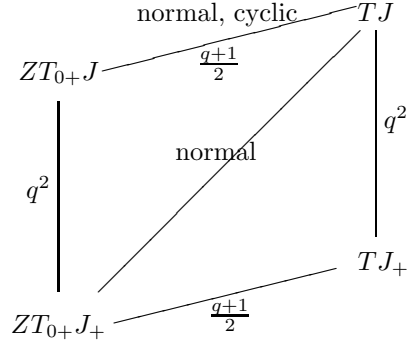
\begin{center}
\setlength{\unitlength}{0.04pt}
\begin{picture}(3968,3366)(0,-10)
\drawline(1049,246)(3149,771)
\drawline(899,546)(3374,3021)
\drawline(940,2634)(3340,3246)
\put(524,2496){\makebox(0,0)[b]{$ZT_{\MPlus0}J$}}
\put(3596,696){\makebox(0,0)[b]{$TJ_+$}}
\put(3524,3110){\makebox(0,0)[b]{$TJ$}}
\put(449,1386){\makebox(0,0)[rb]{$q^2$}}
\put(3599,2046){\makebox(0,0)[lb]{$q^2$}}
\put(2144,171){\makebox(0,0)[b]{$\frac{q+1}{2}$}}
\put(2024,3060){\makebox(0,0)[b]{normal, cyclic}}
\put(2039,2554){\makebox(0,0)[b]{$\frac{q+1}{2}$}}
\drawline(3524,1071)(3524,3021)
\put(2012,1791){\makebox(0,0)[b]{normal}}
\drawline(524,546)(524,2346)
\put(524,96){\makebox(0,0)[b]{$ZT_{\MPlus0}J_+$}}
\end{picture}
\end{center}
Given a character $\psi$ of $T$ that extends $\chi$,
we want to construct a corresponding
representation $\sigma(T, \psi)$ of $TJ$ and compute its character.
Using formula~\eqref{eq:char-noheisenberg},
one can define an extension
of $\psi$ to $TJ_+$;
we will denote it again by $\psi$.
It is clear that such extensions are parametrized by the characters
of $T$ that extend $\chi$.

Meanwhile, applying Lemma~\ref{lem:heisenberg}
(with $H$ and $A$ the quotients of $ZT_{\MPlus0}J$ and
$ZT_{\MPlus0}J_+$, respectively, by their common normal
subgroup \(\ker(\res{\psi}to{ZT_{\MPlus0}J_+})\)),
we see that
there is a unique irreducible
representation $\kappa_\psi$ of $ZT_{\MPlus0}J$ containing
$\res\psi to{ZT_{\MPlus0}J_+}$.
Moreover, $\res{\kappa_\psi}to{ZT_{\MPlus0}J_+}$ is a sum of $q$ copies of
$\res\psi to{ZT_{\MPlus0}J_+}$,
and the character of $\kappa_\psi$ is supported on $ZT_{\MPlus0}J_+$.

\begin{rem}
\label{rem:kappa-extends}
Note that \(\kappa_\psi\) contains \(\res\psi to{J_+} = \bar\chi\),
in the notation of \S\ref{sec:posktypes};
hence, by Proposition \ref{prop:heisenberg},
contains the representation \(\rho_\chi\).
Since \(\kappa_\psi\) and \(\rho_\chi\) both have dimension
\(\indx{Z T_{\MPlus0}J}{Z T_{\MPlus0}J_+}^{1/2}
= q = \indx J{J_+}^{1/2}\),
we actually have that \(\kappa_\psi\) extends \(\rho_\chi\).
\end{rem}

\begin{lemma}
\label{lem:fixed}
The character $\res\psi to{ZT_{\MPlus0}J_+}$ is fixed under conjugation by $TJ$.
\end{lemma}

\begin{proof}
Since $\res\psi to Z$ is clearly fixed, it is enough to show
that $\res\psi to{T_{\MPlus0}J_+}$ is, too.

We show first that $G_{x,r/2}$, which contains $J$, stabilizes
$\res\psi to{T_{\MPlus0}J_+}$.
Since
$$
[G_{x,r/2},G_{x,(r/2)+}]\subseteq G_{x,r+}\subseteq \ker(\psi),
$$
the group $G_{x,r/2}$ fixes $\res\psi to{G_{x,(r/2)+}}$.
Suppose $g\in G_{x,r/2}$
and $t\in T_{\MPlus0}$.  Then, putting
\(Y = \cayley\inv(t)\)
and \(W = \cayley\inv(g)\),
we have by Lemma \ref{lem:cayley}\pref{item:cayley-bracket}
that
\begin{multline*}
\psi\bigl(\Int(g)t\bigr)
=
\psi(t) \psi([t\inv,g]) \\
=
\psi(t) \AddChar\bigl(\Tr(X \cdot \cayley\inv([t\inv,g]))\bigr)
=
\psi(t) \AddChar\bigl(\Tr([X,Y]\cdot W)\bigr)
=
\psi(t).
\end{multline*}

Finally, we show that $T$ stabilizes $\res\psi to{T_{\MPlus0}J_+}$.
Let $e\in T$, $t\in T_{\MPlus0}$, and $k\in J_+\subseteq G_{x,(r/2)+}$.
Write \(W = \cayley\inv(k)\).
Then, using
Lemma~\ref{lem:cayley}(%
	\ref{item:cayley-equivariant},
	\ref{item:cayley-bracket}%
),
\begin{multline*}
\psi(\lsup e (tk))
=
\psi(t) \psi(\lsup e k)
=
\psi(t) \AddChar\bigl(\Tr(X\cdot \cayley\inv(\Int(e)k)) \bigr)
=
\psi(t) \AddChar\bigl(\Tr(X\cdot \Ad(e)W) \bigr) \\
=
\psi(t) \AddChar\bigl(\Tr(\Ad(e)X  \cdot W) \bigr)
=
\psi(t) \psi(k)
=
\psi(tk). \qedhere
\end{multline*}
\end{proof}

Since the character $\res\psi to{ZT_{\MPlus0}J_+}$ is fixed under
conjugation by $TJ$,
the representation $\kappa_\psi$ is also fixed by $TJ$.
Therefore, $\kappa_\psi$ may be extended to a representation of $TJ$.
The number of such extensions is $\indx{TJ}{ZT_1J}=\frac{q+1}{2}$.

We will show how our choice of $\psi$ picks out one of these
extensions.

\begin{lemma}
\label{lem:conj-out}
If $g\in TJ\setminus TJ_+$,
then
$\lsup{g\cap}(TJ_+) = ZT_{\MPlus0}J_+$.
\end{lemma}

As in the proof of Theorem \ref{thm:sc},
we have written \(\lsup{g\cap}(T J_+)\) for
\(T J_+ \cap \Int(g)(T J_+)\).

\begin{proof}
Since $ZT_{\MPlus0}J_+$ is normal in $TJ$,
the right-hand side is contained in the left-hand side.

Let $g=kt'$, with $k\in J\setminus J_+$ and $t'\in T$.
Then
$$
\Int(g)(TJ_+) = \Int(kt')(TJ_+) = \Int(k)(TJ_+)
= \Int(k)(T) \cdot J_+.
$$
Now let $t\in T\setminus ZT_{\MPlus0}$.
It will be enough to show that $\Int(k)t\notin TJ_+$.
Since $t\notin -1\cdot T_{\MPlus0}$, one can check directly that
there is some $a\in\ttt\setminus \ttt_{\MPlus0}$
such that $\cayley(a) = t$.
Put \(b = \cayley\inv(k) \in \JJ \setminus \JJ_+\).
From Lemma~\ref{lem:cayley}(\ref{item:cayley-conj}),
$\Int(k)a \equiv a+[b,a] \pmod{\gg_{x,(r/2)+}}$.
A calculation shows that $[b,a]\in \JJ \setminus \JJ_+$.
Therefore,
$\Int(k)a \in (\ttt+\JJ) \setminus (\ttt+\JJ_+)$,
so, by Lemma \ref{lem:cayley}\pref{item:cayley-equivariant},
$\Int(k)t \notin TJ_+$.
\end{proof}

\begin{cor}
\label{cor:double-coset}
$TJ_+$ has $2q-1$ double cosets in $TJ$.
\end{cor}

\begin{proof}
If $g\in TJ$, then
$$
\card{TJ_+gTJ_+ / TJ_+}
= \card{TJ_+/ (\lsup g(TJ_+) \cap TJ_+)}
=
\begin{cases}
\tfrac1 2(q+1)	& \text{if $g\notin TJ_+$}, \\
1		& \text{if $g\in TJ_+$}.
\end{cases}
$$
Therefore, the number of double cosets is $1+m$,
where $1+\tfrac1 2(q+1)m = \indx{TJ}{TJ_+} = q^2$.
The result follows.
\end{proof}

Consider the representation $I_\psi=\Ind_{TJ_+}^{TJ} \psi$.
Every irreducible component of $I_\psi$ must contain $\res\psi to{T_{\MPlus0}J_+}$,
and therefore $\kappa_\psi$.
That is, as a $TJ$-module,
$I_\psi$ is a sum of extensions of $\kappa_\psi$, with various multiplicities.
Let $a_1,a_2,\ldots,a_{(q+1)/2}$
denote these multiplicities.
Then
$$
\sum a_i = (\dim I_\psi)\dotm(\dim \kappa_\psi)\inv = q,
$$
and so
\begin{equation*}
\begin{split}
\sum a_i^2
&= \dim \Hom_{TJ}( I_\psi, I_\psi )\\
&= \sum_{g\in TJ_+\backslash TJ/TJ_+}
	\dim \Hom_{\lsup{g\cap}(TJ_+)}(\psi,
	\psi \circ \Int(g)\inv).
\end{split}
\end{equation*}
From Lemma~\ref{lem:fixed},
each term in this last sum is $1$.
From Corollary~\ref{cor:double-coset},
the sum is $2q-1$.

\begin{lemma}
The multiplicities $a_i$ are all equal to $2$, except for one of them,
which is equal to $1$.
\end{lemma}

\begin{proof}
Apply \cite{moy:thesis-original}*{Lemma 3.5.4}
with $\Delta = \tfrac1 2(q+1)$, and $r= q$.
\end{proof}

Let $\sigma(T, \psi)$ denote the unique extension of $\kappa_\psi$ to
$TJ$ having multiplicity one in $I_\psi$.

\begin{prop}
\label{prop:inducing-char}
We have
$$
\Theta_{\sigma(T,\psi)}(g) =
\begin{cases}
-\psi(g) & g \in TJ_+\setminus ZT_1J_+ \\
q\cdot \psi(g) & g\in ZT_1J_+ \\
0 & \text{if $g$ is not conjugate to an element of $TJ_+$}.
\end{cases}
$$
\end{prop}

\begin{proof}
Since $ZT_{\MPlus0}J$ is normal in $TJ$ and $TJ/ZT_{\MPlus0}J$ is cyclic, we have that
the $\tfrac1 2(q+1)$ extensions of $\kappa_\psi$ are of the form
$\sigma(T,\psi) \otimes \nu$ for
$\nu \in (TJ/ZT_{\MPlus0}J)\sphat\ \iso (T/ZT_{\MPlus0})\sphat$\,\,.
Thus, in the Grothendieck ring, we have
\begin{equation}
\tag{$*$}
\label{eq:Ipsi-rep}
\begin{aligned}
I_\psi= \Ind_{TJ_+}^{TJ} \psi
& {}= 2\Bigl(\sum_{\nu\neq 1} \sigma(T,\psi)\otimes \nu\Bigr)+\sigma(T,\psi)
	\\
& {}= 2\Bigl(\sum_{\nu} \sigma(T,\psi)\otimes \nu\Bigr)-\sigma(T,\psi).
\end{aligned}
\end{equation}
Let $\dot\psi$ denote the function on $TJ$ that is
equal to $\psi$ on $TJ_+$ and is zero on $TJ\setminus TJ_+$.
Then for all $g\in TJ$,
\begin{equation}
\tag{$**$}
\label{eq:Ipsi-char}
\begin{aligned}
\Theta_{I_\psi}(g)
& {}=\!\!\sum_{s\in TJ /TJ_+}
	\!\!\dot\psi\bigl(\Int(s)g\bigr) \\
& {}=
\begin{cases}
q^2 \psi(g) &\text{if $g \in ZT_{\MPlus0}J_+$,}\\
\psi(g) &\text{if $g\in TJ_+ \setminus ZT_{\MPlus0}J_+$,} \\
0 &\text{if $g$ is not conjugate to an element of $TJ_+$.}
\end{cases}
\end{aligned}
\end{equation}
Combining \eqref{eq:Ipsi-rep} and \eqref{eq:Ipsi-char}
gives the desired result.
\end{proof}

\begin{proof}[Proof of Proposition \ref{prop:extension}]
Proposition~\ref{prop:inducing-char} shows that
$\abmapto\psi{\sigma(T,\psi)}$  is an injective map from the
set of characters of $T$ that extend $\res\chi to{T_r}$
to the set of representations of $TJ$;
and, together with Proposition~\ref{prop:heisenberg},
that the image of the map is contained in the set
of representations of $TJ$ that extend $\rho_\chi$.
(This latter fact can also be observed directly; see
Remark \ref{rem:kappa-extends}.)
We show that it is a surjection by a counting argument.
\textit{A priori}, we do not know how many extensions there
are of \(\rho_\chi\) from \(J\) to \(T J\), but certainly
there are no more than
$$
\indx{TJ}J
= \indx T{T \cap J}
= \indx T{T_r}.
$$
Since there are \emph{exactly}
\(\indx T{T_r}\) extensions of \(\res\chi to{T_r}\) to \(T\),
the (injective) map must be surjective.
\end{proof}

It remains to show that our parametrization
agrees with that of Yu \cite{yu:supercuspidal}.

\begin{prop}
\label{prop:yu-equiv}
The representation $\sigma(T, \psi)$ of $TJ$
is equivalent to the inducing representation
constructed in \cite{yu:supercuspidal}*{\S4}
from the datum
\(\bigl((\bT, \bG), x_\bT, (\psi, 1)\bigr)\)
(see \S3 \loccit).
\end{prop}

Our argument uses results from \S\S\ref{sec:SSPhi}, \ref{sec:roots}
that will only be proven later; but the reader can check that there
is no circularity involved.

\begin{proof}
We write \(\sigma'(T, \psi)\) for Yu's inducing datum; it is
denoted in \cite{yu:supercuspidal} by \(\rho_1\).

All extensions of $(TJ_+,\psi)$ to $TJ$ agree
on $ZT_{\MPlus0}J_+$, and all have characters supported on
the conjugacy classes that meet $TJ_+$
(see \cite{yu:supercuspidal}*{\S11}).
Thus, $\sigma'(T, \psi) = \sigma(T, \psi')$,
where $\psi'$ agrees with $\psi$ on $ZT_{\MPlus0}J_+$.
To determine $\psi'$, it is enough
to compute the character of $\sigma'(T, \psi)$
at an element of $TJ_+ \setminus ZT_{\MPlus0}J_+$.
On that domain,
the explicit construction of \cite{yu:supercuspidal}*{\S4}
shows that
$\sigma'(T, \psi) = \psi \otimes \tilde\psi$
(i.e., that
\(\sigma'(T, \psi)(t j_+)
= \psi(t)\tilde\psi(t \ltimes j_+)\)
for \(t \in T\) and \(j_+ \in J_+\)),
where $\tilde\psi$ is a representation of
$T\ltimes J_+$ that is trivial on
$T_{\MPlus0} \ltimes \sset1$
and $\bar\chi$-isotypic on $1 \ltimes J_+$
(see Theorem 11.5 \loccit).
From
\cite{adler-spice:explicit-chars}*{%
	Proposition \xref{char-prop:theta-tilde-phi}%
},
the character of $\tilde\psi$ at $g$
is $\varepsilon(\psi,g)$.
From Proposition \ref{prop:sl2-weil}
and Lemma \ref{lem:SSPhi},
we see that $\varepsilon(\psi,\gamma) = (-1)^{r+1} = -1$.
Thus,
$\sigma(\psi)$ and $\sigma'(\psi)$ have the same character.
\end{proof}

\section{Murnaghan--Kirillov theory}
\label{sec:MK}

Our calculations of character values near the identity
(see Theorems \ref{thm:near} and \ref{thm:exc-near})
rely on Murnaghan--Kirillov theory, i.e., on asymptotic
descriptions of the character in terms of Fourier transforms
of orbital integrals (see \S\ref{sec:mu-hat}).
In the ordinary case, we have a `single-orbit' theory, i.e.,
only one orbital integral is involved
(see Proposition \ref{prop:MK}); but, in the
exceptional case, the situation is more complicated
(see Proposition \ref{prop:MK-exc}).

In the depth-zero case, we use results of
\cite{debacker-reeder:depth-zero-sc},
which require Hypothesis
\ref{hyp:debacker-reeder:depth-zero-sc:res-12.4.1(2)}.

\subsection{Lusztig's generalized Green functions}

Put \(\sfG = {\SL_2}_{/\resfld}\), and write \(\ms U\) for the
set of unipotent elements in \(\sfG\).
Lusztig \cite{lusztig:char-sheaves-V}*{Lemma 25.4} has
described the space of class functions on \(\ms U(\resfld)\)
(or, rather, on the set of unipotent elements in
any finite group of Lie type)
in terms of \term{generalized Green functions}
\cite{lusztig:char-sheaves-II}*{(8.3.1)}.
In our setting, there are only two such functions that we
need to consider.

\begin{defn}[%
	\cite{deligne-lusztig:finite}*{Definition 4.1}%
]
\label{defn:green}
The \term{(elliptic) Green function} \(Q^\sfG_{\sfT^\epsilon}\)
is the restriction to \(\ms U(\resfld)\) of \(R^\sfG_{\sfT^\epsilon, \psi}\)
for any character \(\psi\) of \(\sfT^\epsilon(\resfld)\)
(for example, \(\psi = 1\)).  
\end{defn}

\begin{defn}
\label{defn:lusztig}
The \term{Lusztig function} \indexmem{\lsup0f}
is defined by
$$
\lsup{0}f \ldef \left[
	\Int(\SL_2(\resfld)) \begin{pmatrix}
		0 & 1 \\ 
		0 &0
	\end{pmatrix}  \right]  -
	\left[ \Int(\SL_2(\resfld)) \begin{pmatrix}
		0 &  \epsilon \\ 
		0  & 0 
	\end{pmatrix} \right]
$$
(i.e., the difference of the characteristic functions of the
two non-trivial nilpotent orbits in \(\sl_2(\resfld)\)).
We will view this as a function only on the set of nilpotent
elements, or on all of \(\sl_2(\resfld)\), as convenient.
See \cite{waldspurger:nilpotent}*{p.~7}.
By abuse of notation, we will also write \(\lsup0f\) for the
function \(\abmapto u{\lsup0f(\cayley\inv(u))}\) on
\(\ms U(\resfld)\)
.
\end{defn}

When convenient,
we think of a generalized Green function on \(\ms U(\resfld)\)
as a function on $\sfG(\resfld)$ by setting it equal to zero
off the set of unipotent elements.

The definition of the Fourier transform
(Definition \ref{defn:FT})
also makes sense for vector spaces \(V\) over finite fields;
in that setting, the self-dual Haar measure assigns
to \(V\) measure \(\card V^{1/2}\).
We also need a choice of non-trivial additive character on
\(\resfld = \pint/\pp\).
Since \(\AddChar\) is trivial on \(\pp\), but non-trivial on
\(\pint\), it induces such a character in a natural way.
The Fourier transform in the next lemma is taken with respect to the
specified measure and character;
and the constant \(\Gauss(\AddChar)\) is as in
Definition \ref{defn:gauss}.

\begin{lemma}
\label{lem:lusztig-FT}
\(
\widehat{\lsup0f}
= \sgn_\varpi(-1)\Gauss(\AddChar)\dotm\lsup 0f.
\)
\end{lemma}

\begin{proof}
This is~\cite{waldspurger:nilpotent}*{Proposition~V.8}
in the symplectic case, with $k =1$.
(See also \cite{lusztig:ft-ss}*{Corollary 10}.)
\end{proof}

\subsection{Lifting from finite to local fields}

Suppose that \(\sigma\) is a cuspidal representation of
\(\sfG(\resfld) = \SL_2(\resfld)\) \cite{carter:finite}*{\S9.1}.
Then we may write the restriction to \(\ms U(\resfld)\)
of the character \(\chi_\sigma\) of \(\sigma\)
as a \(\C\)-linear combination
\begin{equation}
\label{eq:sigma-as-green}
\res{\chi_\sigma}to{\ms U(\resfld)}
= \sum_\Green c_\sigma(\Green)\Green,
\end{equation}
where the sum runs over
\(\sset{Q^\sfG_{\sfT^\epsilon}, \lsup0f}\).
(As with all the results of this section, an appropriate
analogue of this statement holds in a very general setting;
we shall exhibit an explicit linear combination in all cases
of interest.)

We will show how to lift this finite-field formula (over \resfld)
to the local-field setting (over \(\field\)).
We warn the reader that our discussion will involve both
generalized Green functions, denoted as above by \(\Green\),
and a fourth root of unity, denoted by \(\Gauss(\AddChar)\).
Context should make clear which is meant.

For the remainder of this section, put \(x = \xleft\),
so that \(G_{x, 0:\MPlus0} \iso \SL_2(\resfld)\).

\begin{notn}
\label{notn:dot}
If \(f\) is any function on
\(\SL_2(\resfld) = G_{x, 0:\MPlus0}\), then we
denote by \(\dot f\) the function on \(G\) defined by
\[
\dot f(g) = \begin{cases}
f(\bar g), &
	\text{%
		if \(g \in G_{x, 0}\) has image \(\bar g \in \sfG_x(\resfld)\)%
	} \\
0, &
	\text{otherwise.}
\end{cases}
\]
Similarly, if \(f\) is any function on \(\sl_2(\resfld)\),
then we inflate and extend it to a function \(\dot f\) on
\(\gg\).
(This is the function denoted by \(f_{\sset x}\) in
\cite{debacker-kazhdan:mk-zero}*{p.~3}.)
\end{notn}

\begin{notn}
Let \(\textup d\ell\) be the Haar measure on \(G_{x, 0}\)
that assigns it total mass \(1\).
(We shall preserve the notation \(\textup dg\) for the Haar
measure on \(G\) specified in \S\ref{sec:measure}.)
\end{notn}

Now we may inflate \(\sigma\) to a representation
\(\dot\sigma\) of \(\SL_2(\pint) = G_{x, 0}\).
From, for example,
Proposition~\ref{prop:depthzero-sc-induced}, the representation
$\pi = \Ind_{G_{x, 0}}^G \dot\sigma$ is an irreducible,
hence supercuspidal,
representation of $G$,
so its character may be computed by Harish-Chandra's
integral formula \cite{hc:harmonic}*{p.~94}:
\begin{equation}
\label{eq:hc:harmonic:thm-12}
\Theta_\pi(\gamma)
= \frac{\deg_{\textup dg/\textup dz}(\pi)}{\chi_\sigma(1)}
\int_{G/Z(G)} \int_{G_{x, 0}}
	\dot\chi_\sigma\bigl(\Int(g\ell)\gamma\bigr)
\textup d\ell\,\frac{\textup dg}{\textup dz},
\quad\gamma \in G\rss.
\end{equation}
We will use the decomposition \eqref{eq:sigma-as-green} to
evaluate this integral formula on the topologically
unipotent set \(G_{\MPlus0}\)
(see \cite{debacker-kazhdan:mk-zero}*{(5)}),
but first we need a convergence result.
Fix \(\gamma \in G\rss \cap G_{\MPlus0}\),
and write \(\gamma = \cayley(Y)\)
with \(Y \in \gg\rss \cap \gg_{\MPlus0}\).

The intersection with \(G_{x, 0}\) of the \(G\)-orbit
of \(\gamma\) projects to \(\ms U(\resfld)\)
in \(G_{x, 0:\MPlus0}\), so that the expression
\(\dot\Green\bigl(\Int(g\ell)\gamma\bigr)\)
makes sense
(see Notation \ref{notn:dot}).

\begin{lemma}
\label{lem:supp}
If
\(\Green
\in \sset{Q^\sfG_{\sfT^\epsilon}, \lsup0f}\),
then
\[
\abmapto g{\int_{G_{x, 0}}
	\dot\Green\bigl(\Int(g\ell)\gamma\bigr)
\textup d\ell}
\]
is a compactly supported function on \(G\).
\end{lemma}

\begin{proof}
From~\cite{hc:harmonic}*{Lemma 23}, we know that both
\[
\abmapto g{\int_{G_{x, 0}}
	\dot{\chi}_\sigma\bigl(\Int(g\ell)\gamma\bigr)
\textup d\ell}
\qandq
\abmapto g{\int_{G_{x, 0}}
	\dot{R}^\sfG_{\sfT^\epsilon,\psi}\bigl(\Int(g\ell)\gamma\bigr)
\textup d\ell}
\]
are compactly supported functions on $G$
(when \(\psi \ne \psi\inv\)).
The result for \(\Green = Q^\sfG_{\sfT^\epsilon}\) follows
(see also \cite{debacker-reeder:depth-zero-sc}*{Lemma 10.0.6}).

For $\Green = \lsup 0f$,
choose a cuspidal representation $\sigma$ for which
$c_\sigma(\lsup 0f) \neq 0$
(namely, \(\sigma = R^\pm_{\sfT^\epsilon, \psi_0}\)).
Since
$$
\lsup 0f
=
\dfrac{1}{c_\sigma(\lsup 0f)} \bigl(
	\chi_\sigma - c_\sigma(Q^\sfG_{\sfT^\epsilon}) Q^\sfG_{\sfT^\epsilon}
	\bigr),
$$
the result follows.
\end{proof}

By \eqref{eq:sigma-as-green} and Lemma \ref{lem:supp},
since
\[
\frac{\deg_{\textup dg/\textup dz}(\pi)}{\chi_\sigma(1)}
= \meas_{\textup dg/\textup dz}(\SL_2(\pint)/Z(G))\inv
= \frac{2q^{1/2}}{q^2 - 1},
\]
\eqref{eq:hc:harmonic:thm-12} becomes
\begin{equation}
\label{eq:pi-as-green}
\Theta_\pi(\gamma)
= \frac{2q^{1/2}}{q^2 - 1}
\sum_\Green c_\sigma(\Green)\int_{G/Z(G)} \int_{G_{x, 0}}
	\dot\Green\bigl(\Int(g\ell)\gamma\bigr)
\textup d\ell\,\frac{\textup dg}{\textup dz},
\end{equation}
so we wish to understand these integrals of Green functions.

\begin{lemma}
\label{lem:MK-green}
Under
Hypothesis \ref{hyp:debacker-reeder:depth-zero-sc:res-12.4.1(2)},
\begin{equation*}
\frac{2q^{1/2}}{q^2 - 1}
\int_{G/Z(G)} \int_{G_{x, 0}}
	\dot Q^\sfG_{\sfT^\epsilon}(\Int(g \ell) \gamma)
\textup d\ell\,\frac{\textup dg}{\textup dz}
= - \hat{\mu}^G_{X^{\epsilon, 1}_1} (Y),
\end{equation*}
where the orbital integral is taken with respect to the
measure \(\textup dg/\textup dt^\epsilon\) on
\(G/T^\epsilon\).
\end{lemma}

\begin{proof}
This follows from
\cite{debacker-reeder:depth-zero-sc}*{\S9.2 and Lemma 12.4.3}
upon noting that the measure with respect to \(\textup dz\)
of \(Z(G)_0 = \sset1\) (there denoted by \(Z_J\)) is \(1\);
that the measure with respect to \(\textup dg\)
of \(G_{x, 0} = \SL_2(\pint)\) is \(q^{-1/2}(q^2 - 1)\);
and that
\[
2\int_{G/Z(G)} f(g)\frac{\textup dg}{\textup dz}
= \int_G f(g)\textup dg,
\quad f \in \Hecke(G/Z(G)).\qedhere
\]
\end{proof}

\begin{lemma}
\label{lem:MK-lusztig}
\begin{multline*}
\frac{2q^{1/2}}{q^2 - 1}\int_{G/Z(G)} \int_{G_{x, 0}}
	(\lsup{0}f)\spdot(\Int(g \ell) \gamma)
\textup d\ell\,\frac{\textup dg}{\textup dz} \\
= \dfrac{\Gauss(\AddChar)}{2 q} \Bigl[
	\bigl(
		\hat\mu^G_{X^{\varpi, 1}_1} -
		\hat\mu^G_{X^{\varpi, \epsilon}_1}
	\bigr) +
	\bigl(
		\hat\mu^G_{X^{\epsilon\varpi, 1}_1} -
		\hat\mu^G_{X^{\epsilon\varpi, (\epsilon\inv)}_1}
	\bigr)
\Bigr](Y),
\end{multline*}
where the orbital integrals are taken with respect to the
measures \(\textup dg/\textup dt^{\theta, \eta}\)
on \(G/T^{\theta, \eta}\),
with \(\theta \in \sset{\varpi, \epsilon\varpi}\)
and \(\eta \in \sset{1, \epsilon}\).
\end{lemma}

\begin{proof}
Following the proof of
\cite{debacker-kazhdan:mk-zero}*{Lemma 5.3.1},
define the distribution $D_{\lsup{0} f}$ on \(\gg\) by
$$
D_{\lsup{0}f} (F)
= \int_\gg \, \int_G \, \int_{G_{x,0}}
	(\lsup{0}f)\spdot\bigl(\Ad(g \ell)Y\bigr) F(Y)
\textup d\ell \, \textup dg \, \textup dY,
\quad F \in \Hecke(\gg),
$$
where \(\textup dY\) is the self-dual Haar measure on \(\gg\)
(so that \(\meas_{\textup dY}(\sl_2(\pint)) = q^{3/2}\)).

By Lemma \ref{lem:supp}, the innermost integral defines a
compactly supported function on
\(\gg \times G\), so that we may switch the order of
the outer two integrals (over \(\gg\) and \(G\)).
For fixed \(g \in G\),
the integrand is compactly supported as a function on
\(G_{x, 0} \times \gg\),
so we may switch the (now) inner two integrals
(over \(\gg\) and \(G_{x, 0}\)),
obtaining
$$
D_{\lsup{0}f} (F)
= \int_G \, \int_{G_{x,0}} \, \int_\gg
	(\lsup{0}f)\spdot\bigl(\Ad(g \ell)Y\bigr) F(Y)
\textup dY \, \textup d\ell \, \textup dg,
\quad F \in \Hecke(\gg).
$$
From~\cite{waldspurger:nilpotent}*{Lemme III.2},
the innermost integral defines a compactly supported function on
$G \times G_{x,0}$,
so that we may switch the order of the outer integrals.
Absorbing the integral over $G_{x,0}$ into that over $G$,
and recalling that $x = \xleft$,
we find that our distribution $D_{\lsup{0}f}$  
is equal to 
\[
\abs{\sfG (\resfld )}   q^{-3/2}  \phi_{(1,0, \emptyset, \emptyset, \emptyset)}
= \frac{q^2 - 1}{q^{1/2}}\phi_{(1, 0, \emptyset, \emptyset, \emptyset)},
\]
where $\phi_{(1,0, \emptyset, \emptyset, \emptyset)}$ is as defined by Waldspurger~\cite{waldspurger:nilpotent}*{p.~53}.

By Lemma \ref{lem:lusztig-FT},
for all $F \in \Hecke(\gg)$, we have 
\begin{align*}
\sgn_\varpi(-1)\Gauss(\AddChar)\, D_{\lsup 0f} (F) &=  \int_G \,  \int_\gg
	\widehat{(\lsup{0}f)\spdot}\bigl(\Ad(g )Y\bigr) F(Y)
\textup dY \, \textup dg \\
&=
  \int_G \,  \int_\gg
	(\lsup{0}f)\spdot \bigl(\Ad(g )Y\bigr) \hat{F}(Y)
\textup dY \, \textup dg \\
&=  D_{\lsup 0f} (\hat{F}) \\
&= \frac{q^2 - 1}{q^{1/2}}\phi_{(1,0, \emptyset, \emptyset, \emptyset)}(\hat{F})
\end{align*}
By \cite{spice:sl2-mu-hat}*{Lemma \xref{sl2-lem:G-facts}},
the above equality becomes
\[
\frac{q^{1/2}}{q^2 - 1}D_{\lsup0f}
= \Gauss(\AddChar)
	\hat\phi_{(1, 0, \emptyset, \emptyset, \emptyset)}.
\]
In the notation of~\cite{waldspurger:nilpotent}*{\S I.9}, we have
$\widehat{\mc H} = \Hecke(\gg\textsub{tn}) = \Hecke(\gg_{\MPlus0})$.
The result now follows
from~\cite{waldspurger:nilpotent}*{p.~70 and Proposition~IV.3},
upon adjusting (as in Lemma \ref{lem:MK-green})
for the difference between integration over \(G\)
and \(G/Z(G)\).
(See also \cite{waldspurger:nilpotent}*{p.~7},
where the constants are not completely explicated.) 
\end{proof}

Lemmas \ref{lem:MK-green} and \ref{lem:MK-lusztig} allow us
to re-write \eqref{eq:pi-as-green} as
\begin{equation}
\label{eq:pi-as-green-no-really}
\begin{aligned}
\Theta_\pi(\gamma) ={} &
	-c_\sigma(Q^\sfG_{\sfT^\epsilon})
		\hat\mu^G_{X^{\epsilon, 1}_1}(Y) \\
	& \qquad+ c_\sigma(\lsup0f)\frac{\Gauss(\AddChar)}{2q}\bigl[
		(\hat\mu^G_{X^{\varpi, 1}_1} -
			\hat\mu^G_{X^{\varpi, \epsilon}_1}) +
		(\hat\mu^G_{X^{\epsilon\varpi, 1}_1} -
			\hat\mu^G_{X^{\epsilon\varpi, (\epsilon\inv)}_1})
	\bigr](Y).
\end{aligned}
\end{equation}

\subsection{Character expansions}

\begin{prop}
\label{prop:MK}
Suppose that \(\pi\) is an ordinary supercuspidal
representation of \(G\), and put \(r = \depth(\pi)\).
If
\begin{itemize}
\item
\(\gamma \in G\rss \cap G_{\MPlus r}\)
	and Hypothesis \ref{hyp:debacker-reeder:depth-zero-sc:res-12.4.1(2)}
holds,
or
\item
\(r > 0\) and \(\gamma \in G\rss \cap G_r\),
\end{itemize}
then
\[
\Theta_\pi(\gamma)
= \deg(\pi)\dotm
\hat\mu^G_{X_\pi}\bigl(
	\cayley\inv(\gamma)
\bigr).
\]
\end{prop}

Here, \(X_\pi\) is the regular, semisimple element
associated to \(\pi\) in
Notations \ref{notn:depth-0-X} and \ref{notn:pos-depth-X}.
Put \(T = C_G(X_\pi)\).

Recall from Definition \ref{defn:mu-hat} that
\(\hat\mu^G_{X_\pi}\) depends on a choice of measure on
\(G/T\),
and from \S\ref{sec:deg-pi} that \(\deg(\pi)\) depends on a
choice of measure on \(G/Z(G)\).
The actual choices of measure in the above proposition are
unimportant; they need only be \emph{consistent},
in the sense that the measure of a subset of \(G/T\)
is the same as the measure of its pull-back (along the
natural projection) in \(G/Z(G)\);
i.e., that the `quotient measure' on \(T/Z(G)\) assigns it
total mass \(1\).
(Note that the measures \(\textup dg/\textup dt^\epsilon\)
and \(\textup dg/\textup dz\) of
\S\ref{sec:measure} are \emph{not} consistent in
this sense.)

\begin{proof}
If \(r > 0\), then this follows from
\cite{adler-debacker:mk-theory}*{%
	Theorem 6.3.1 and Remark 6.3.2%
}
or
\cite{adler-spice:explicit-chars}*{%
	Corollary \xref{char-cor:char-tau|pi-1-germ}%
}.
(There are other, similar results in the literature, but most
of them make stronger assumptions on $p$
\cites{
	jkim-murnaghan:charexp,
	jkim-murnaghan:gamma-asymptotic,
	murnaghan:chars-u3,
	murnaghan:chars-sln,
	murnaghan:chars-classical,
	murnaghan:hc,
	murnaghan:chars-gln
}.)

If \(r = 0\), then recall from
Remark \ref{rem:depth-0-leftist} that it suffices to
consider the case where
\(\pi\) is induced from the inflation to
\(G_{\xleft, 0}\) of the representation
\(\sigma = -R^{\sfG_\xleft}_{\sfT^\epsilon, \psi}\) of
\(\SL_2(\resfld) \iso \sfG_{\xleft}(\resfld)\).
In this case, by Definition \ref{defn:green},
\(c_\sigma(Q^{\sfG_\xleft}_{\sfT^\epsilon}) = -1\)
and
\(c_\sigma(\lsup0f) = 0\);
so \eqref{eq:pi-as-green-no-really}
and \cite{debacker-reeder:depth-zero-sc}*{\S5.3}
give
\begin{align*}
\Theta_\pi(\gamma)
& {}= \hat\mu^G_{X_\pi}\bigl(\cayley\inv(\gamma)\bigr) \\
& {}= \frac{2q^{1/2}}{q + 1}\dotm
\deg_{\textup dg/\textup dz}(\pi)
	\hat\mu^G_{X_\pi}\bigl(\cayley\inv(\gamma)\bigr),
\end{align*}
where the orbital integral is computed with respect to the
measure \(\textup dg/\textup dt^\epsilon\)
on \(G/T^\epsilon\).
(As in the proof of Lemma \ref{lem:MK-green},
we adapt the results of
\cite{debacker-reeder:depth-zero-sc}*{\S5.3}
to account for the fact that we are working on \(G/Z(G)\),
whereas they work with \(G/Z\),
where \(\bZ = \sset1\) is the maximal \(\field\)-split torus
in \(Z(G)\).)
As mentioned above, the measures \(\textup dg/\textup dz\)
and \(\textup dg/\textup dt^\epsilon\) are not consistent;
in fact,
\[
\meas_{\textup dt^\epsilon/\textup dz}(T^\epsilon/Z(G))
= \frac{q + 1}{2q^{1/2}}.
\]
Adjusting either measure to achieve consistency thus gives
the desired result.
\end{proof}

\begin{prop}
\label{prop:MK-exc}
If \(\gamma \in G\rss \cap G_{\MPlus0}\)
and Hypothesis \ref{hyp:debacker-reeder:depth-zero-sc:res-12.4.1(2)}
holds, then
\[
\Theta_{\pi^\pm}(\gamma)
= \tfrac1 2\hat\mu^G_{X^{\epsilon, 1}_1}\bigl(
	\cayley\inv(\gamma)
\bigr) \pm{}
\tfrac1 4 q^{-1/2}\Bigl[
	\bigl(
		\hat\mu^G_{X^{\varpi, 1}_1} -
		\hat\mu^G_{X^{\varpi, \epsilon}_1}
	\bigr) +
	\bigl(
		\hat\mu^G_{X^{\epsilon\varpi, 1}_1} -
		\hat\mu^G_{X^{\epsilon\varpi, \epsilon}_1}
	\bigr)
\Bigr],
\]
where \(\pi^\pm = \bPi^\pm(T^{\epsilon, 1}, \psi_0^1)\).

\end{prop}

\begin{proof}
By Definitions \ref{defn:exc-rep-finite}
and \ref{defn:depth-0-param},
\(\pi^\pm\) is induced from the inflation to
\(G_{\xleft, 0}\) of the representation
\[
\sigma
= -\tfrac1 2 R^{\sfG_\xleft}_{\sfT^\epsilon, \psi_0} \pm
\tfrac1 2 q^{1/2}\Gauss(\AddChar)\inv\dotm\lsup0f
\]
of \(\SL_2(\resfld) \iso \sfG_\xleft(\resfld)\).
In this setting, by Definition \ref{defn:green},
\(c_\sigma(Q^{\sfG_\xleft}_{\sfT^\epsilon}) = -\tfrac1 2\)
and
\(c_\sigma(\lsup0f)
= \pm\tfrac1 2 q^{1/2}\Gauss(\AddChar)\inv\);
so \eqref{eq:pi-as-green-no-really}
gives the desired formula.
\end{proof}

\section{`Ordinary' supercuspidal characters}
\label{sec:ordinary}

Let \(\pi\) be an ordinary supercuspidal representation.
By Definitions \ref{defn:depth-0-param}
and \ref{defn:pos-depth-param},
\(\pi = \pi(T, \psi)\) for some
(depth-zero or positive-depth)
supercuspidal parameter \((T, \psi)\).

\begin{defn}
\label{defn:lots}
By Remarks \ref{rem:depth-0-leftist}
and \ref{rem:pos-depth-leftist},
we may, and do, assume that \(T = T^{\theta, 1}\)
for some
\(\theta \in \sset{\epsilon, \varpi, \epsilon\varpi}\).
Since the character formulas for the case
\(\theta = \epsilon\varpi\) are, \textit{mutatis mutandis},
the same as those for the case \(\theta = \varpi\),
we further restrict to the cases
\(\theta \in \sset{\epsilon, \varpi}\).
We say that we are in the \term{unramified case}
if \(\theta = \epsilon\),
and in the \term{ramified case}
if \(\theta = \varpi\).

Put
\begin{itemize}
\item
\(\indexmem r = \depth(\psi)\) and \(\indexmem s = r/2\);
\item
\(\indexmem x = x_\bT \in \sset{\xleft, \xcentre}\)
(see Definition \ref{defn:xT});
and
\item
\(\indexmem X = X_\pi\)
and
\(\indexmem{\SSAddChar} = \SSAddChar_\pi\)
(see Notations \ref{notn:depth-0-X}
and \ref{notn:pos-depth-X}).
\end{itemize}
Finally, write \indexmem h for the \term{conductor}
of \(\psi\), in the sense of
\cite{sally-shalika:characters}*{p.~1232}.
Then \(h = \rup r = r + 1 - \tfrac1 2\ord(\theta)\),
so \(h = r + 1\) in the unramified case
and \(h = r + \tfrac1 2\) in the ramified case.
By \eqref{eq:depth-Phi-b},
\(\depth(\SSAddChar)
= r - \tfrac1 2\ord(\theta) = h - 1\).
\end{defn}

Note that \(\SSAddChar\) does not bear the same relationship
to \(\AddChar\) as does \(\Phi'\) to \(\Phi\)
in \cite{spice:sl2-mu-hat}*{%
	Notation \xref{sl2-notn:depth-and-scdepth}%
}.

\subsection{Indices}
\label{sec:indices}

\begin{lemma}
\label{lem:indices}
Suppose that \(i > 0\).
In the unramified case,
\[
\indx{T G_{x, i}}{T G_{x, \MPlus i}}
= \begin{cases}
q^2, & i \in \Z      \\
1,   & i \not\in \Z.
\end{cases}
\]
In the ramified case,
\[
\indx{T G_{x, i}}{T G_{x, \MPlus i}}
= \begin{cases}
q, & 2i \in \Z      \\
1, & 2i \not\in \Z.
\end{cases}
\]
\end{lemma}

\begin{proof}
By Lemma \ref{lem:cayley}\pref{item:bijpos},
we have compatible isomorphisms
\(G_{x, i:\MPlus i} \iso \gg_{x, i:\MPlus i}\)
and
\(T_{i:\MPlus i} \iso \ttt_{i:\MPlus i}\),
so it suffices to compute
\[
\card{\gg_{x, i:\MPlus i}}\dotm\card{\ttt_{i:\MPlus i}}\inv
= \exp_q\bigl(
	\dim_\resfld(\gg_{x, i:\MPlus i}) -
	\dim_\resfld(\ttt_{i:\MPlus i})
\bigr).
\]
We use the explicit descriptions of \S\ref{sec:filt}.

In the unramified case, both quotients are trivial unless
\(i \in \Z\), in which case the \resfld-vector spaces
\(\gg_{x, i:\MPlus i} \iso \gg_{x, 0:\MPlus0}\)
and
\(\ttt_{i:\MPlus i} \iso \ttt_{0:\MPlus0}\) are
isomorphic to \(\sl_2(\resfld)\) and a Cartan subalgebra,
hence are \(3\)- and \(1\)-dimensional, respectively.
(Explicitly, they are spanned by
\(\varpi^i\dotm\sset{
	\begin{smallpmatrix}
	1 & \phm0 \\
	0 & -1
	\end{smallpmatrix},
	\begin{smallpmatrix}
	0 & 1 \\
	0 & 0
	\end{smallpmatrix},
	\begin{smallpmatrix}
	0 & 0 \\
	1 & 0
	\end{smallpmatrix}
}\)
and
\(\sset{\varpi^i\dotm\smash{X^\epsilon}}\)).

In the ramified case, both quotients are trivial unless
\(2i \in \Z\).
If \(i \in \Z\), then
\(\gg_{x, i:\MPlus i} \iso \gg_{x, 0:\MPlus0}\)
and
\(\ttt_{i:\MPlus i} \iso \ttt_{0:\MPlus0}\),
are isomorphic to a
\emph{split} Cartan subalgebra \(\mf a(\resfld)\)
of \(\sl_2(\resfld)\),
and an \emph{elliptic} Cartan subalgebra
\emph{of \(\mf a(\resfld)\)},
hence are \(1\)- and \(0\)-dimensional, respectively.
(An explicit basis for \(\gg_{x, i:\MPlus i}\) is
\(\sset{\varpi^i\begin{smallpmatrix}
1 & \phm0 \\
0 & -1
\end{smallpmatrix}}\).)
If \(i \in \Z + \tfrac1 2\), then
\(\gg_{x, i:\MPlus i}\) has basis
\(\varpi^i\dotm\sset{
	\begin{smallpmatrix}
	0 & 1 \\
	0 & 0
	\end{smallpmatrix},
	\varpi\begin{smallpmatrix}
	0 & 0 \\
	1 & 0
	\end{smallpmatrix}
}\),
and
\(\ttt_{i:\MPlus i}\) has basis
\(\sset{\varpi^i\cdot \smash{X^\varpi}}\).
\end{proof}

\begin{lemma}
\label{lem:disc-as-index}
If \(\gamma \in T \setminus Z(G)T_r\)
and \(d = \mdepth(\gamma)\),
in the sense of Definition \ref{defn:depth-element},
then
\[
\bigindx{
	T_{\MPlus0}G_{x, (r - d)/2}
}{
	T_{\MPlus0}G_{x, s}
}\dotm\bigindx{
	T_{\MPlus0}G_{x, \MPlus{((r - d)/2)}}
}{
	T_{\MPlus0}G_{x, \MPlus s}
}
= \abs{D_G(\gamma)}\inv.
\]
\end{lemma}

\begin{proof}
Note that we could replace the first index by
\(\bigindx{
	T G_{x, (r - d)/2}
}{
	T G_{x, s}
}\),
and similarly for the second.
By Lemma \ref{lem:indices},
in the unramified case, the product is
\begin{align*}
&\Bigl(\prod_{\substack{
	(r - d)/2 \le i < s \\
	i \in \Z
}} q^2\Bigr)\dotm\Bigl(\prod_{\substack{
	(r - d)/2 < i \le s \\
	i \in \Z
}} q^2\Bigr) \\
&\qquad= \exp_{q^2}\Bigl(
	\bigcard{\Z \cap [\tfrac{r - d}{2}, s)} +
	\bigcard{\Z \cap (\tfrac{r - d}{2}, s]}
\Bigr) \\
&\qquad= \exp_{q^2}
	\bigcard{
		\Z \cap \Bigl(
			[\tfrac{r - d}{2}, s) \cup [-s, -\tfrac{r - d}{2})
		\Bigr)
	} \\
&\qquad= \exp_{q^2}
	\bigcard{\Z \cap [s - d, s)};
\end{align*}
and, in the ramified case, it is
\begin{align*}
&\Bigl(\prod_{\substack{
	(r - d)/2 \le i < s \\
	2i \in \Z
}} q\Bigr)\dotm\Bigl(\prod_{\substack{
	(r - d)/2 < i \le s \\
	2i \in \Z
}} q\Bigr) \\
&\qquad= \exp_q\Bigl(
	\bigcard{\tfrac1 2\Z \cap [\tfrac{r - d}{2}, s)} +
	\bigcard{\tfrac1 2\Z \cap (\tfrac{r - d}{2}, s]}
\Bigr) \\
&\qquad= \exp_q\Bigl(
	\bigcard{\Z \cap [r - d, r)} +
	\bigcard{\Z \cap (-r, -(r - d)]} 
\Bigr) \\
&\qquad= \exp_q
	\bigcard{\Z \cap [r - 2d, r)}.
\end{align*}
We translated
\(\Z \cap [-s, -\tfrac{r - d}2)\) by \(r - d \in \Z\)
(in the first case)
and \(\Z \cap (-r, -(r - d)]\) by \(2(r - d) \in \Z\)
(in the second case)
without changing any cardinalities.
In either case, we use the fact that
\(\card{\Z \cap [a - n, a)} = n\)
for any \(a \in \R\) and \(n \in \Z\)
(with \(n = d\) in the unramified case,
and \(n = 2d\) in the ramified case)
to conclude by Lemma \ref{lem:disc-as-depth}
that the index is
\[
q^{2d} = \abs{D_G(\gamma)}\inv,
\]
as desired.
\end{proof}

\subsection{Formal degrees}
\label{sec:deg-pi}

The Schur orthogonality relations for finite groups
have an analogue for supercuspidal (or even discrete-series)
representations of \(p\)-adic groups;
see \cite{hc:harmonic}*{Theorem 1}.
The analogue of the dimension of a finite-group
representation is the so called \term{formal degree}
of a supercuspidal representation \(\pi\).
This definition depends on a choice of Haar measure
\(\textup d\dot g\) on \(G/Z(G)\), so we denote it by
\(\deg_{\textup d\dot g}(\pi)\).
We have that
\(\deg_{c\dotm\textup d\dot g}(\pi)
= c\inv\deg_{\textup d\dot g}(\pi)\)
for \(c \in \R_{> 0}\).

\begin{lemma}
\label{lem:deg-pi}
In the unramified case,
\[
\deg_{\textup d_\epsilon\dot g}(\pi)
= q^{r + 1}.
\]
In the ramified case,
\[
\deg_{\textup d_\varpi\dot g}(\pi)
= q^{h + 1}.
\]
\end{lemma}

\begin{proof}
If \(r = 0\), then \(T = T^\epsilon\), and
\cite{debacker-reeder:depth-zero-sc}*{\S5.3} gives
\[
\deg_{\textup dg/\textup dz}(\pi)
= \frac{q^{1/2}}{(q + 1)/2}
= \frac{2q^{1/2}}{q + 1}.
\]
By \S\ref{sec:measure},
\(\displaystyle\textup d_\epsilon\dot g
= \frac 2{q^{1/2}(q + 1)}\dotm\frac{\textup dg}{\textup dz}\),
so that
\[
\deg_{\textup d_\epsilon\dot g}(\pi)
= \frac{q^{1/2}(q + 1)}2\dotm\frac{2q^{1/2}}{q + 1}
= q = q^{r + 1}.
\]

Now suppose that \(r > 0\).
Write \(\textup d\dot g\) for \(\textup d_\epsilon\dot g\)
or \(\textup d_\varpi\dot g\), as appropriate;
and recall the notation \(K = T J\)
and \(\rho = \rho_\chi\) from \S\ref{sec:posreps}.
By Remark \ref{rem:inducing-data},
\(K = T G_{x, s}\)
and
\(\dim(\rho) = q^{1 - (\rup s - \rdown s)}\).
Since \(\pi = \Ind_K^G \rho\), we have that
\begin{equation}
\tag{$*$}
\label{eq:deg-as-indx}
\begin{aligned}
\deg_{\textup d\dot g}(\pi)
& {}= \meas_{\textup d\dot g}(K/Z(G))\inv\dotm
\dim(\rho) \\
& {}= \indx{\SL_2(\pint)}K\dotm
\meas_{\textup d\dot g}(\SL_2(\pint)/Z(G))\inv\dotm
q^{1 - (\rup s - \rdown s)}.
\end{aligned}
\end{equation}

In the unramified case,
\(\bT = \bT^\epsilon\), \(x = \xleft\),
and \(G_{x, s} = G_{x, \rup s}\).
A direct computation shows that
\[
\indx{\SL_2(\pint)}{T G_{x, 1}}
= \indx{\SL_2(\resfld)}{\ms T(\resfld)}
= q(q - 1),
\]
so Lemma \ref{lem:indices} gives
\begin{equation}
\tag{${**}\textsub{un}$}
\label{eq:indx-un}
\indx{\SL_2(\pint)}K
= q(q - 1)\dotm\indx{T G_{x, 1}}{T G_{x, \rup s}}
= q(q - 1)q^{2(\rup s - 1)}.
\end{equation}

In the ramified case,
\(\bT = \bT^\varpi\), \(x = \xcentre\),
and \(G_{x, s} = G_{x, \rup r/2}\).
Further, \(r \in \Z + \tfrac1 2\), so
\(\rup r = r + \tfrac1 2\)
and
\(1 - (\rup s - \rdown s) = 0\).
This time, a direct computation shows that
\[
\indx{\SL_2(\pint)}{T G_{x, 1/2}}
= \indx{\SL_2(\resfld)}{Z(\SL_2)(\resfld)\ms U(\resfld)}
= \tfrac1 2(q^2 - 1),
\]
where \(\ms U(\resfld) = \set{\begin{smallpmatrix}
1 & a \\
0 & 1
\end{smallpmatrix}}{a \in \resfld}\).
Now Lemma \ref{lem:indices} gives
\begin{equation}
\tag{${**}\textsub{ram}$}
\label{eq:indx-ram}
\begin{aligned}
\indx{\SL_2(\pint)}K
& {}= \tfrac1 2(q^2 - 1)\dotm
\indx{T G_{x, 1/2}}{T G_{x, \rup r/2}} \\
& {}= \tfrac1 2(q^2 - 1)q^{\rup r - 1} \\
& {}= \tfrac1 2(q^2 - 1)q^{r - 1/2}.
\end{aligned}
\end{equation}

Combining \eqref{eq:deg-as-indx},
\eqref{eq:indx-un} \emph{or} \eqref{eq:indx-ram},
and \S\ref{sec:measure}
gives the desired result.
\end{proof}

\subsection{Roots of unity}
\label{sec:roots}

The character formulas of
\cite{adler-spice:explicit-chars} involve a number of
roots of unity, defined in terms of roots (i.e., weights for
the adjoint action of a maximal torus).

\begin{notn}
\label{notn:roots}
Write \(\indexmem{\alpha_+}\) for the element
\[
\abmapto{\begin{pmatrix}
a       & b \\
b\theta & a
\end{pmatrix}}{(a^2 + b^2\theta) + 2a b\sqrt\theta}
\]
of \(\Hom_{\field_\theta}(\bT, \GL_1)\),
and \(\alpha_- = -\alpha_+\).
Then the set \indexmem{\Roo(\bG, \bT)} of absolute roots
of \bT in \bG is \(\sset{\alpha_\pm}\),
and the set \indexmem{\dot\Roo(\bG, \bT)}
of orbits of \(\Gal(\field\textsup{sep}/\field)\)
on \(\Roo(\bG, \bT)\) is a singleton.
\end{notn}

\begin{rem}
We have that \(\alpha_\pm(\gamma) = \gamma^{\pm2}\) for
\(\gamma \in T \iso C_\theta\) (see \S\ref{sec:std-tori}).
\end{rem}

\begin{notn}
\label{notn:adler-spice:explicit-chars:root-constants}
Throughout this section, we adopt
\cite{adler-spice:explicit-chars}*{%
	Notation \xref{char-notn:root-constants}%
}.
In particular,
\[
F_\alpha = \field_\theta,
\quad
F_{\pm\alpha} = \field,
\qandq
\bG_{\alpha_+} = \bG_{\alpha_-} = \bG;
\]
and, if \(\theta = \epsilon\), then
\[
\resfld_\alpha^1 = \ker \Norm_{\resfld_\alpha/\resfld_{\pm\alpha}}
= \ker \Norm_{\resfld_\theta/\resfld} \rdef \resfld_\theta^1.
\]
The notations on the left are as in
\cite{adler-spice:explicit-chars},
and those on the right are ours.
\end{notn}

Two roots of unity enter into the character formulas
in \cite{adler-spice:explicit-chars}*{%
	Corollary \xref{char-cor:char-tau|pi-1}%
}, namely, \(\varepsilon(\psi, \gamma)\)
and \(\mf G(\psi, \gamma)\).
(There is an unfortunate near-conflict between the notation
\(\epsilon\), for an element of \(\pint\mult\), and
\(\varepsilon\), which we use to stand for various signs.
We hope that context will
allow the reader to distinguish them.)
In this section, we compute these quantities
(and, more importantly, their product, in
Corollary \ref{cor:root})
in our special case (i.e., the group \(\bG = \SL_2\)).

We begin by computing
the `depth-zero sign'
\(\varepsilon(\psi, \gamma)\).

\begin{notn}
\label{notn:adler-spice:explicit-chars:weil}
Adopt
\cite{adler-spice:explicit-chars}*{%
	Notation \xref{char-notn:weil}%
}.
In particular, \(\Xi^1(\psi)\) is the set of
(absolute) roots of \bT in \bG that
``occur in the filtration
(of \(\sl_2(\field_\theta)\))
associated to \(x\) at depth \(s\)'',
and whose value at \(\gamma\) is not a principal unit.
By our explicit description of Moy--Prasad filtrations (see
\S\ref{sec:filt}), we have that,
if \(x = \xleft\), then
\begin{align*}
\Xi^1(\psi) & {}= \begin{cases}
\Roo(\bG, \bT), &
	\text{\(s \in \Z\) and \(\gamma \not\in Z(G)T_{\MPlus0}\)} \\
\emptyset, &
	\text{otherwise;}
\end{cases} \\
\intertext{and, if \(x = \xcentre\), then}
\Xi^1(\psi) & {}= \begin{cases}
\Roo(\bG, \bT), & s \in \Z + \tfrac1 2 \\
\emptyset,      & \text{otherwise.}
\end{cases}
\end{align*}
\end{notn}

\begin{prop}
\label{prop:sl2-weil}
Suppose that \(\gamma \in T \setminus Z(G)T_r\).
Put \(d = \mdepth(\gamma)\),
in the notation of Definition \ref{defn:depth-element}.
The root of unity \(\varepsilon(\psi, \gamma)\) defined in
\cite{adler-spice:explicit-chars}*{%
	Proposition \xref{char-prop:theta-tilde-phi}%
} is given by
\[
\varepsilon(\psi, \gamma)
= \begin{cases}
H(\SSAddChar, \field_\theta)
\sgn_\theta\bigl(\Im_\theta(\gamma)\bigr), &
	d = 0 \\
1, &
	d > 0.
\end{cases}
\]
\end{prop}

The factor \(\sgn_\theta\bigl(\Im_\theta(\gamma)\bigr)\)
will be shown in the course of the proof to be \(1\), so
that we could leave it out; but we find it convenient to
include it, for consistency with Proposition
\ref{prop:sl2-gauss}.

\begin{proof}
As remarked after
\cite{adler-spice:explicit-chars}*{%
	Proposition \xref{char-prop:theta-tilde-phi}%
},
since all roots are symmetric,
we may use \cite{gerardin:weil}*{Corollary 4.8.1},
rather than Theorem 4.9.1 \loccit,
to obtain the alternate formula
\begin{equation}
\tag{$*$}
\label{eq:sl2-weil}
 \varepsilon(\psi, \gamma)
= (-1)^{\card{\dot\Xi^1\symm(\psi, \gamma)}}
\prod_{\alpha \in \dot\Xi\symm^1(\psi, \gamma)}
	\sgn_{\resfld_\alpha^1}(\alpha(\gamma)),
\end{equation}
where \(\sgn_{\resfld_\alpha^1}\) is the unique
(non-trivial) order-\(2\) character of
\(\resfld_\alpha^1\).
(The notation \(\dot\Xi\)
for a set of orbits is as in Notation \ref{notn:roots}.)

We have \(\Xi^1(\psi, \gamma) = \emptyset\) if
\(d > 0\), i.e., \(\gamma \in Z(G)T_{\MPlus0}\);
in particular, this holds whenever
\(\theta \ne \epsilon\).
Then \eqref{eq:sl2-weil} becomes
\(\varepsilon(\psi, \gamma) = 1\), as desired.
Therefore, we focus on the case where \(d = 0\).
In particular,
\(\Xi^1(\psi, \gamma) = \Xi(\psi)\).

By Lemma \ref{lem:Im-facts},
\(\ord\bigl(\Im_\epsilon(\gamma)\bigr) = 0\),
so
\(\sgn_\epsilon\bigl(\Im_\epsilon(\gamma)\bigr) = 1\).
Thus, by Lemma \ref{lem:SSPhi}, it suffices to show that
\(\varepsilon(\psi, \gamma) = (-1)^{r + 1}\).

If \(s \not\in \Z\), i.e., \(r \not\in 2\Z\),
then \(\Xi(\psi, \gamma) = \emptyset\),
so that \eqref{eq:sl2-weil} becomes
\(\varepsilon(\psi, \gamma) = 1 = (-1)^{r + 1}\).

If \(s \in \Z\), i.e., \(r \in 2\Z\), then
\(\Xi^1\symm(\psi, \gamma) = \Xi(\psi, \gamma)
= \Roo(\bG, \bT)\),
whence \(\dot\Xi^1\symm(\psi, \gamma)\) is a singleton,
and \(\Xi\nosymm(\psi, \gamma) = \emptyset\),
so \eqref{eq:sl2-weil} again becomes
\[
\varepsilon(\psi, \gamma)
= (-1)^1\dotm1\dotm
	\sgn_{\resfld_\epsilon^1}(\alpha_+(\gamma))
= -\sgn_{\resfld_\epsilon^1}(\gamma^2)
= -1 = (-1)^{r + 1}.\qedhere
\]
\end{proof}

Now we compute the `positive-depth sign'
\(\mf G(\psi, \gamma)\).

\begin{notn}
\label{notn:adler-spice:explicit-chars:gauss}
Adopt
\cite{adler-spice:explicit-chars}*{%
	Notation \xref{char-notn:gauss}%
}.
In particular, \(\Upsilon(\psi, \gamma)\) is
empty if \(d = 0\),
and otherwise is the set of
(absolute) roots of \bT in \bG that
``occur in the filtration
(of \(\sl_2(\field_\theta)\))
associated to \(x\) at depth \((r - d)/2\)''.
(Recall that \(d\) is the (maximal) depth of \(\gamma\).)
As in Notation \ref{notn:adler-spice:explicit-chars:weil},
we have that, if \(x \in \sset{\xleft, \xright}\), then
\[
\Upsilon(\psi, \gamma) = \begin{cases}
\Roo(\bG, \bT), & (r - d)/2 \in \Z  \\
\emptyset,      & \text{otherwise;}
\end{cases}
\]
and, if \(x = \xcentre\), then
\[
\Upsilon(\psi, \gamma) = \begin{cases}
\Roo(\bG, \bT), & r - d \in \Z \\
\emptyset,      & \text{otherwise.}
\end{cases}
\]
We set
$\Upsilon\symmunram(\psi, \gamma)$
(respectively, $\Upsilon\symmram(\psi, \gamma)$)
equal to 
$\Upsilon(\psi, \gamma)$
in the unramified (respectively, ramified) case,
and to the empty set otherwise.
\end{notn}


Our calculation of \(\mf G(\psi, \gamma)\) in the ramified
case (Proposition \ref{prop:sl2-gauss})
will involve the quantity \(S(\psi)\) defined in
\cite{sally-shalika:characters}*{p.~1234}.
The measure used in its definition is not specified there, but
the statement (on p.~1235 \loccit) that
\(S(\psi)^2 = \sgn_\varpi(-1)\) holds only for the
normalization chosen below.

\begin{lemma}
\label{lem:S-psi}
In the ramified case,
\begin{align*}
S(\psi)
& {}\ldef q^{1/2}\int_{\pint\mult}
	\sgn_\varpi(Y)\psi\Bigl(\frac{
		1 + \varpi^{h - 1}\sqrt\varpi Y
	}{
		1 - \varpi^{h - 1}\sqrt\varpi Y
	}\Bigr)\textup dY \\
& {}\phantom{:}= \sgn_\varpi(-1)^{h - 1}H(\SSAddChar, \field_\varpi),
\end{align*}
where \(\textup dY\) is the Haar measure on \(\field\)
such that \(\meas_{\textup dY}(\pint) = 1\).
\end{lemma}

\begin{proof}
By the definitions of
\(\cayley\)
(see Lemma \ref{lem:cayley-field})
and
\(X\), \(\beta\), and \(\SSAddChar\)
(see Definitions \ref{defn:pos-depth-param}
and \ref{defn:lots}),
\begin{multline*}
\psi\Bigl(
	\frac{
		1 + \varpi^{h - 1}\sqrt\varpi Y
	}{
		1 - \varpi^{h - 1}\sqrt\varpi Y
	}
\Bigr)
= \psi(\cayley(2\varpi^{h - 1}\sqrt\varpi Y)) \\
= \AddChar\bigl(
	\Tr(X\dotm2\varpi^{h - 1}\sqrt\varpi Y\rangle)
\bigr)
= \AddChar(2\beta\varpi\dotm2\varpi^{h - 1}Y)
= \SSAddChar_{4\varpi^{h - 1}}(Y).
\end{multline*}

Since \(\meas_{\textup dY}(1 + \pp) = q\inv\),
and since the restriction
of the additive Haar measure \(\textup dY\)
to \(\pint\mult\) is a multiplicative Haar measure,
\begin{align*}
S(\psi)
& {}= q^{-1/2}\sum_{Y \in \pint\mult/(1 + \pp)}
	\sgn_\varpi(Y)\SSAddChar_{4\varpi^{h - 1}}(Y) \\
& {}= \Gauss(\SSAddChar_{4\varpi^{h - 1}}).
\end{align*}
By
\cite{spice:sl2-mu-hat}*{Lemma \xref{sl2-lem:G-facts}}
and
Lemma \ref{lem:SSPhi},
\begin{align*}
S(\psi)
& {}= \sgn_\varpi(-1)^{h - 1}\Gauss(\SSAddChar) \\
& {}= \sgn_\varpi(-1)^{h - 1}H(\SSAddChar, \field_\varpi).
	\qedhere
\end{align*}
\end{proof}

\begin{prop}
\label{prop:sl2-gauss}
Suppose that \(\gamma \in T \setminus Z(G)T_r\).
Put \(d = \mdepth(\gamma)\),
in the notation of Definition \ref{defn:depth-element}.
The root of unity \(\mf G(\psi, \gamma)\) defined in
\cite{adler-spice:explicit-chars}*{%
	Proposition \xref{char-prop:gauss-sum}%
} is given by
\[
\mf G(\psi, \gamma) =
\begin{cases}
H(\SSAddChar, \field_\theta)
\sgn_\theta\bigl(\Im_\theta(\gamma)\bigr), &
	d > 0 \\
1, &
	d = 0.
\end{cases}
\]
\end{prop}

\begin{proof}
We use the following formula from the cited proposition
(adapted to our situation, per
Notation \ref{notn:adler-spice:explicit-chars:root-constants})
to compute \(\mf G(\psi, \gamma)\):
\begin{equation}
\tag{$*$}
\label{eq:sl2-gauss}
\begin{aligned}
\mf G(\psi, \gamma)
={} &
(-1)^{\card{\dot\Upsilon\symm(\psi, \gamma)}}
\bigl(-\Gauss(\SpecAddChar)\bigr)^{
	f(\dot\Upsilon\symmram(\psi, \gamma))
}\times{} \\
&\qquad\prod_{
	\alpha \in \dot\Upsilon\symmram(\psi, \gamma)
}
	\sgn_\resfld\bigl[
		\Norm_\varpi(w_\varpi)
		\textup d\alpha(X_\SpecAddChar)
		(\alpha(\gamma) - 1)
	\bigr],
\end{aligned}
\end{equation}
where
\begin{itemize}
\item \(\SpecAddChar\) is a certain (additive) character of
\(\field\)
(specified in
\cite{adler-spice:explicit-chars}*{%
	\S\xref{char-sec:generalities}%
}
and denoted there by $\Lambda$);
\item
the notation \(f(\cdot)\) is defined by
\(f(\dot\Roo(\bG, \bT)) = f(\field_{\alpha_+}/\field)
= 1\)
and
\(f(\dot\emptyset) = 0\);
\item \(X_\SpecAddChar\) is an element such that
\[
\psi(\cayley(Y))
= \SpecAddChar\bigl(
	\Tr(X_\SpecAddChar\dotm Y)
\bigr)
\quad\text{for all \(Y \in \ttt_r\)};
\]
\item \(w_\varpi = \sqrt\varpi^{r - d}\);
and
\item the argument of \(\sgn_\resfld\), which lies in
\(\pint_\varpi\mult\), is implicitly regarded as an element of
\(\resfld\mult\).
\end{itemize}
(The notation \(\dot\Upsilon\)
for a set of orbits is as in
Notation \ref{notn:roots}.)
The condition on \(w_\varpi\) can be satisfied only if
\(r - d \in \Z\); but this is always the case when
\(\Upsilon\symmram(\psi, \gamma) \ne \emptyset\).
This formula differs from the one in
\cite{adler-spice:explicit-chars} in several ways.
\begin{itemize}
\item
The original formula had \(\mf G_\SpecAddChar(\resfld)\) in place of
\(\Gauss(\SpecAddChar)\); but
\cite{adler-spice:explicit-chars}*{%
	Definition \xref{char-defn:basic-gauss-sum}
	and Lemma \xref{char-lem:well-known-gauss}%
}
and \cite{spice:sl2-mu-hat}*{Lemma \xref{sl2-lem:G-facts}}
show that they are equal (since \(\depth(\SpecAddChar) = 0\)).
\item
The argument of \(\sgn_{\resfld_\alpha}\) in the
original formula had a factor of \(\tfrac1 2 e_\alpha\),
where \(e_\alpha\) is the ramification degree of
\(\field_\alpha/\field\);
but this factor collapses to \(1\) whenever
\(\alpha \in \Upsilon\symmram(\psi, \gamma)\).
\item
The original formula had \(\textup d\alpha^\vee(X^*)\)
in place of \(\textup d\alpha(X_\SpecAddChar)\); but these
are the same once we taken into account our identification
of \(\gg\) and \(\gg^*\).
\item The product included an extra factor
\(\sgn_\field(\bG_{\pm\alpha})\), defined to be \(+1\)
if \(\bG_{\pm\alpha}\) is \(\field\)-split and \(-1\)
otherwise.
Since \(\bG_{\pm\alpha} = \SL_2\), this factor is \(1\).
\end{itemize}

Note that \eqref{eq:sl2-gauss} collapses to \(1\) unless
\(d > 0\), so we assume that.

In the unramified case,
\(\Upsilon\symmram(\psi, \gamma) = \emptyset\),
so \eqref{eq:sl2-gauss} becomes
\begin{equation}
\tag{$*\textsub{un}$}
\mf G(\psi, \gamma) = (-1)^{\card{
	\dot\Upsilon\symm(\psi, \gamma)
}}.
\end{equation}
If \((r - d)/2 \in \Z\), then
\(\dot\Upsilon\symm(\psi, \gamma) = \dot\Roo(\bG, \bT)\)
is a singleton, so
\begin{equation}
\tag{$\dag\textsub{even}$}
\label{eq:sl2-gauss-unram-even}
\mf G(\psi, \gamma) = -1 = (-1)^{r + 1}(-1)^d.
\end{equation}
If \((r - d)/2 \not\in \Z\), then
\(\dot\Upsilon\symm(\psi, \gamma) = \emptyset\),
so
\begin{equation}
\tag{$\dag\textsub{odd}$}
\label{eq:sl2-gauss-unram-odd}
\mf G(\psi, \gamma) = 1 = (-1)^{r + 1}(-1)^d.
\end{equation}
In either case, Lemma \ref{lem:Im-facts} shows that
\(\sgn_\epsilon\bigl(\Im_\epsilon(\gamma)\bigr) = (-1)^d\);
so that, by Lemma \ref{lem:SSPhi},
\eqref{eq:sl2-gauss-unram-even} and
\eqref{eq:sl2-gauss-unram-odd} both
simplify to the desired formula.

In the ramified case, recall that \(r - d \in \Z\);
in particular,
\(\dot\Upsilon\symmram(\psi, \gamma)
= \dot\Upsilon\symm(\psi, \gamma) = \dot\Roo(\bG, \bT)\)
is a singleton,
and
\(f(\dot\Upsilon\symmram(\psi, \gamma))
= f(\field_\varpi/\field) = 1\).
Thus, \eqref{eq:sl2-gauss} becomes
\begin{equation}
\tag{$*\textsub{ram}$}
\label{eq:sl2-gauss-ram}
\mf G(\psi, \gamma)
= \Gauss(\SpecAddChar)\sgn_\resfld\bigl(
	\Norm_\varpi(w_\varpi)\textup d\alpha_+(X_\SpecAddChar)
	(\alpha_+(\gamma) - 1)
\bigr).
\end{equation}

Since \(\SpecAddChar\) and \(\AddChar\) are both non-trivial,
additive characters of \(\field\), we have
\(\AddChar = \SpecAddChar_b\) for some \(b \in \field\mult\).
Then we may take
\begin{equation}
\tag{$\ddag$}
\label{eq:X-Lambda}
X_\SpecAddChar = b X = b\beta\begin{pmatrix}
0      & 1 \\
\varpi & 0
\end{pmatrix};
\end{equation}
and we note for future reference that, by
\cite{spice:sl2-mu-hat}*{Lemma \xref{sl2-lem:G-facts}},
\begin{equation}
\tag{$\ddag\ddag$}
\label{eq:Gs}
\sgn_\varpi(b\beta\varpi)\Gauss(\SpecAddChar)
= \sgn_\varpi(\beta\varpi)\Gauss(\AddChar)
= \Gauss(\AddChar_{\beta\varpi})
= \Gauss(\SSAddChar).
\end{equation}

Since \(\alpha_+(\gamma) = \gamma^2\), we have that
\(\alpha_+(\gamma) - 1 = (\gamma - \gamma\inv)\gamma\).
Similarly, we can calculate explicitly from Notation
\ref{notn:roots} and \eqref{eq:X-Lambda} that
\(\textup d\alpha_+(X_\SpecAddChar)
= 2b\beta\sqrt\varpi\),
so that
\begin{align*}
\sgn_\resfld\bigl[
	\Norm_\varpi(w_\varpi)
	\textup d\alpha(X_\SpecAddChar)
	(\alpha_+(\gamma) - 1)
\bigr]
& {}= \sgn_\resfld\bigl[
	\Norm_\varpi(w_\varpi)\dotm
	4b\beta\varpi\dotm
	\tfrac1 2\sqrt\varpi\inv(\gamma - \gamma\inv)\dotm
	\gamma
\bigr] \\
& {}= \sgn_\resfld\bigl[
	\Norm_\varpi(w_\varpi)\dotm
	4b\beta\varpi\dotm
	\Im_\varpi(\gamma)
\bigr] \\
& {}= \sgn_\varpi\bigl[
	\Norm_\varpi(w_\varpi)\dotm
	4b\beta\varpi\dotm
	\Im_\varpi(\gamma)
\bigr] \\
& {}= \sgn_\varpi(b\beta\varpi)\dotm
\sgn_\varpi\bigl(\Im_\varpi(\gamma)\bigr).
\end{align*}
(The crucial point in the transition from the second to the
third line is that the argument lies in \(\pint\mult\), not
just \(\pint_\varpi\mult\).)
Thus, by \eqref{eq:Gs} and Lemma \ref{lem:SSPhi},
\eqref{eq:sl2-gauss-ram} becomes
\begin{align*}
\mf G(\psi, \gamma)
& {}= \sgn_\varpi(b\beta\varpi)\Gauss(\SpecAddChar)\dotm
\sgn_\varpi\bigl(\Im_\varpi(\gamma)\bigr) \\
& {}= H(\SSAddChar, \field_\varpi)\sgn_\varpi\bigl(
	\Im_\varpi(\gamma)
\bigr).\qedhere
\end{align*}
\end{proof}

\begin{cor}
\label{cor:root}
With the notation of
Propositions \ref{prop:sl2-weil} and \ref{prop:sl2-gauss},
\[
\varepsilon(\psi, \gamma)\mf G(\psi, \gamma)
= H(\SSAddChar, \field_\theta)
\sgn_\theta\bigl(\Im_\theta(\gamma)\bigr).
\]
\end{cor}

\subsection{Character values far from the identity}
\label{sec:far}

\begin{thm}
\label{thm:far}
If
\(\gamma \not\in Z(G)G_{\MPlus0}\),
or
\(r > 0\) and \(\gamma \not\in Z(G)G_r\),
then
\(\Theta_\pi(\gamma) = 0\) unless some \(G\)-conjugate of
\(\gamma\) lies in \(T^\theta\).
If \(\gamma \in T^\theta\), then
\[
\Theta_\pi(\gamma)
= \frac1 2\sgn_\epsilon\bigl(\Im_\epsilon(\gamma)\bigr)
\frac{
	\psi(\gamma) + \psi(\gamma\inv)
}{
	\smabs{D_G(\gamma)}^{1/2}
}
\bigl[(-1)^{r + 1} + H(\SSAddChar, \field_\epsilon)\bigr]
\]
in the unramified case,
and
\begin{align*}
\Theta_\pi(\gamma)
= \frac{
	\sgn_\varpi\bigl(\Im_\varpi(\gamma)\bigr)
}{
	2 \smabs{D_G(\gamma)}^{1/2}
}\Bigl\{&
	\psi(\gamma)\bigl[
		\sgn_\varpi(-1)^{h - 1}S(\psi) +
		H(\SSAddChar, \field_\varpi)
	\bigr] +{} \\
	&\qquad\psi(\gamma\inv)\bigl[
		\sgn_\varpi(-1)^h S(\psi) +
		H(\SSAddChar, \field_\varpi)
	\bigr]
\Bigr\}
\end{align*}
in the ramified case.
\end{thm}

Recall that \(\SSAddChar\) is as in
Definition \ref{defn:lots},
so that \(H(\SSAddChar, \field_\theta)\) is computed in
Lemma \ref{lem:SSPhi}.  We shall use this in the proof.

\begin{proof}
First suppose that \(r = 0\).
By \cite{debacker-reeder:depth-zero-sc}*{Lemma 9.3.1},
we have that \(\Theta_\pi(\gamma) = 0\) unless
\(\gamma \in Z(G)G_0\); so we assume that
\(\gamma \in G_0\).
Then it has a \term{topological Jordan decomposition}
 \(\gamma = \gamma\tsemi\gamma\tunip\),
with \(\gamma\tsemi\) topologically semisimple
and \(\gamma\tunip\) topologically unipotent
(see \S7 \loccit);
and, if \(\gamma \not\in Z(G)G_{\MPlus0}\),
then \(\gamma\tsemi\) is regular.
By \cite{debacker-reeder:depth-zero-sc}*{Lemma 10.0.4},
\(\Theta_\pi(\gamma) = 0\)
unless \(\gamma\tsemi\) is \(G\)-conjugate to an element of
\(T^\epsilon\), so we assume that
\(\gamma\tsemi \in T^\epsilon\).
Thus the subgroup
\(\bG_{\gamma\tsemi} \ldef C_\bG(\gamma)\conn
= \CC\bG{\MPlus0}(\gamma)\)
of \cite{debacker-reeder:depth-zero-sc}*{p.~802}
is just \(\bT = \bT^\epsilon\),
and the set
\[
\widehat{\mc T}(\gamma\tsemi)
\ldef \set{
	(T' = \Int(g)T^\epsilon, \psi' = \psi \circ \Int(g)\inv)
}{
	\gamma\tsemi \in T'
}
\]
of \S10 \loccit is just
\[
\sett{(T^\epsilon, \psi')}{%
	\(\psi' = \psi \circ \Int(n)\inv\)
	for some \(n \in N_G(T^\epsilon)\)%
}
= \sset{(T^\epsilon, \psi), (T^\epsilon, \psi\inv)}.
\]
Further,
\(\gamma\tunip \in G_{\gamma\tsemi} = T^\epsilon\),
so that \(\gamma = \gamma\tunip\gamma\tsemi \in T^\epsilon\)
as well.

Recall that \(\psi \ne \psi\inv\).
Now combining Lemmas 9.3.1 and Lemma 10.0.4 \loccit
gives
\begin{equation}
\tag{$*_0$}
\label{eq:zero-far}
\begin{aligned}
\Theta_\pi(\gamma)
& {}= \varepsilon(\sG_\xleft, \sfT)
\sum_{(T', \psi')}
	\psi'(\gamma\tsemi)
	R(T, T', 1)(\gamma\tunip) \\
& {}= -\bigl[
	\psi(\gamma\tsemi) + \psi(\gamma\tsemi\inv)
\bigr]
	R(T, T, 1)(\gamma\tunip),
\end{aligned}
\end{equation}
where
\begin{itemize}
\item the sum is taken over the orbits
in \(\widehat{\mc T}(\gamma\tsemi)\) under the natural
(trivial) action of \(T\);
\item
\(
\varepsilon(\sG_\xleft, \sfT)
= \varepsilon({\SL_2}_{/\resfld}, \sfT)
= (-1)^{
	\operatorname{rk}_\resfld(\SL_2) -
	\operatorname{rk}_\resfld(\sfT)
} = -1
\)
is the Kottwitz sign defined on p.~802 \loccit;
and
\item
\(R(T, T, 1)\) is the function defined in \S9.2 \loccitthendot.
\end{itemize}
Since \(\depth(\psi) = 0\), we have that \(\psi\) is trivial
on \(\gamma\tunip \in T_{\MPlus0}\),
so that
\[
\psi(\gamma\tsemi) = \psi(\gamma)
\qandq
\psi(\gamma\tsemi\inv) = \psi(\gamma\inv).
\]
By \S5.1 \loccit,
and Lemma \ref{lem:torus-mu-hat},
since \(R^{\sfT^\epsilon}_{\sfT^\epsilon}(1) = 1\),
we have that
\[
R(T, T, 1)(\gamma\tunip)
= \varepsilon(\bT, Z(\bG))
\hat\mu^T_{X^*}\bigl(
	\cayley\inv(\gamma\tunip)
\bigr)
= 1,
\]
where \(\varepsilon(\bT, Z(\bG))\) is the Kottwitz sign, as
above.
By Lemmas \ref{lem:disc-as-depth} and \ref{lem:Im-facts},
\(\abs{D_G(\gamma)} = 1\)
and
\(\sgn_\epsilon\bigl(\Im_\epsilon(\gamma)\bigr) = 1\).
Thus, since \(r = 0\), we have
by Lemma \ref{lem:SSPhi}
that \eqref{eq:zero-far}
simplifies to the desired formula (in this case).

Now suppose that \(r > 0\), and put \(d = \mdepth(\gamma)\),
in the notation of Definition \ref{defn:depth-element}.
Since \(\gamma\) does not lie in the \(G\)-domain
\(Z(G)G_r\), neither do any of its \(G\)-conjugates;
so, with the notation of
\cite{adler-spice:good-expansions}*{%
	\S\xref{exp-sec:normal}
	and Definition \xref{exp-defn:fancy-centralizer-no-underline}%
},
we have for all \(\gamma' \in \Int(G)\gamma\)
that
\[
\gamma'_{< r} = \gamma'
\qandq
\gamma'_{\ge r} = 1,
\]
so that \(\CC\bG r(\gamma')\) is the unique torus containing
\(\gamma'\).
In particular, if \(\gamma' \in T \cap \Int(G)\gamma\),
then
\begin{gather*}
\odc{\gamma'; x, r} = T_{\MPlus0}G_{x, (r - d)/2},\quad
\odc{\gamma'; x, \MPlus r} = T_{\MPlus0}G_{x, \MPlus{((r - d)/2)}}, \\
\odc{\gamma'; x, r}_T = T_{\MPlus0},
\qandq
\odc{\gamma'; x, \MPlus r}_T = T_{\MPlus0}.
\end{gather*}
(There is a minor notational inconvenience
if \(r = 1\) in the unramified case,
or \(r = 1/2\) in the ramified case.
In those cases, \(\gamma'_{< r} = \gamma'\tsemi\),
and we use Lemma \ref{lem:torus-mu-hat}
to notice that
\(\hat\mu^T_{X_\pi}(\cayley\inv(\gamma'_{\ge r}))
= \psi(\gamma'_{\ge r})\).)
With the notation of
\cite{adler-spice:explicit-chars}*{%
	Definition \xref{char-defn:trunc}%
}, the set
\[
\mc T\bigl((\bT, \bG), (r, r)\bigr) \cap \Int(G)\gamma
\ldef \set{\gamma' \in \Int(G)\gamma}{\gamma'_{< r} \in T}
\]
is \(T \cap \Int(G)\gamma\).
Since \(\gamma\) is regular, the intersection
is just \(\Int(N_G(T))\gamma\).
Further, the equivalence relation \(\overset0\sim\) on
that set, defined in
\cite{adler-spice:explicit-chars}*{%
	Definition \xref{char-defn:equiv}%
} by \(\gamma' \overset0\sim \gamma''\)
if and only if
\(\gamma''\) is conjugate in
\(T = \CC T r(\gamma')\) to \(\gamma'\),
is the identity relation.

Therefore, by
\cite{adler-spice:explicit-chars}*{%
	Proposition \xref{char-prop:induction1}
	and Corollary \xref{char-cor:char-tau|pi-1}%
},
\begin{equation}
\tag{$*_{> 0}$}
\label{eq:pos-far}
\begin{aligned}
\Theta_\pi(\gamma)
= \sum_{\gamma'} {}
	&\bigindx{
		T_{\MPlus0}G_{x, (r - d)/2}
	}{
		T_{\MPlus0}G_{x, s}
	}^{1/2}\bigindx{
		T_{\MPlus0}G_{x, \MPlus{((r - d)/2)}}
	}{
		T_{\MPlus0}G_{x, \MPlus s}
	}^{1/2}\times{} \\
&\qquad\mf G(\psi, \gamma')\varepsilon(\psi, \gamma')\dotm
	\psi(\gamma')\dotm
	\hat\mu^T_{X^*}(\cayley\inv(1)\bigr),
\end{aligned}
\end{equation}
where the sum runs over \(\Int(N_G(T))\gamma\),
and
\(\hat\mu^T_{X^*}\) is the Fourier transform of the
\(T\)-orbital integral of \(X^*\) (as in \S\ref{sec:mu-hat})
with respect to the measure on the singleton \(T/C_T(X^*)\)
that assigns it total measure \(1\).
We have made use of two facts that simplify the quoted
results.
\begin{itemize}
\item The representation \(\tau_0\) of
\cite{adler-spice:explicit-chars}*{\S\xref{char-sec:JK}}
is trivial.
\item The inducing subgroup \(K_\sigma = T G_{x, \MPlus0}\) of
\cite{adler-spice:explicit-chars}*{\S\xref{char-sec:JK}}
contains \(H' = \CC G r(\gamma') = T\) for all conjugates
\(\gamma'\) of \(\gamma\).
\end{itemize}
By Lemma \ref{lem:disc-as-index},
the leading product of square roots is
\(\abs{D_G(\gamma')}^{-1/2} = \abs{D_G(\gamma)}^{-1/2}\).
By Corollary \ref{cor:root}, the product
\(\mf G(\psi, \gamma')\varepsilon(\psi, \gamma')\)
of roots of unity is
\[
(-1)^{r + 1}
	\sgn_\epsilon\bigl(\Im_\epsilon(\gamma)\bigr)
= \tfrac1 2\sgn_\epsilon\bigl(\Im_\epsilon(\gamma)\bigr)
	\bigl[(-1)^{r + 1} + H(\SSAddChar, \field_\epsilon)\bigr]
\]
in the unramified case,
and
\begin{multline*}
\sgn_\varpi(-1)^{h - 1}S(\psi)
	\sgn_\varpi\bigl(\Im_\varpi(\gamma)\bigr) \\
= \tfrac1 2
	\sgn_\varpi\bigl(\Im_\varpi(\gamma)\bigr)
	\bigl[\sgn_\varpi(-1)^{h - 1}S(\psi) +
		H(\SSAddChar, \field_\varpi)\bigr]
\end{multline*}
in the ramified case.
By Lemma \ref{lem:torus-mu-hat},
\(\hat\mu^T_{X^*}\bigl(\cayley\inv(1)\bigr)
= \hat\mu^T_{X^*}(0) = 1\).

By \S\ref{sec:std-tori}:
\begin{itemize}
\item
in the unramified case,
\(\Int(N_G(T))\gamma = \sset{\gamma^{\pm1}}\);
\item
in the ramified case, when \(\sgn_\varpi(-1) = 1\),
again \(\Int(N_G(T))\gamma = \sset{\gamma^{\pm1}}\),
and
\(\mf G(\psi, \gamma)\varepsilon(\psi, \gamma)
= \sgn_\varpi(-1)^h S(\psi) + H(\SSAddChar, \field_\varpi)\);
and finally
\item
in the ramified case, when \(\sgn_\varpi(-1) = -1\),
now \(\Int(N_G(T))\gamma = \sset\gamma\),
and \(\sgn_\varpi(-1)^h S(\psi) + H(\SSAddChar, \field_\varpi) = 0\).
\end{itemize}
In each case, we see that \eqref{eq:pos-far} simplifies to
the desired formula.
\end{proof}

\begin{rem}
\label{rem:far}
The calculation in \cite{sally-shalika:characters}
has an error in the ramified case,
when \(\sgn_\varpi(-1) = 1\) and \(h\) is odd,
so that their formulas
for the character of
\(\Pi(\SSAddChar, \psi, \field_\varpi)\)
and ours for the character of
\(\bPi(T^\varpi, \psi)\),
which agree near the identity,
differ by a sign far from the identity.
Remarkably, correcting their formulas does not affect
their computations in \cite{sally-shalika:plancherel}, which
depend not on the individual characters
\(\Pi(\Phi, \psi, \field_\varpi)\)
and \(\Pi(\Phi', \psi, \field_\varpi)\)
but rather on the sum
\(\Pi(\Phi, \psi, \field_\varpi) +
	\Pi(\Phi', \psi, \field_\varpi)\),
which is unchanged;
and also does not affect their computations in
\cite{sally-shalika:orbital-integrals}.
To see this latter is more complicated; but, fortunately,
\cite{sally-shalika:orbital-integrals} writes out the
necessary calculations explicitly in the ramified case.
The only affected parts of the formula
for \(K_d(t_1, t_2)\) on p.~330--332 \loccit
are (d), (e), and (f),
and, even after correcting the error, the argument on
p.~334 \loccit shows that they are all \(0\).
In the formula for \(K_d(t_1, t_2)\) on p.~337,
the terms (d) and (e) no longer appear,
and the term (f) is replaced by (f').
Since (f') involves the product of two characters,
both affected by the sign error, nothing is changed.

We have used results of Shelstad
\cite{shelstad:formula}*{Theorem, p.~276},
together with the formulas of
\cite{debacker-sally:germs}*{Appendix A.3--A.4},
to confirm our calculations.
\end{rem}

\subsection{Character values near the identity}
\label{sec:near}

Recall that we have put \(X = X_\pi\).
By the formulas in \cite{spice:sl2-mu-hat},
the values \(\hat\mu^G_X(Y)\)
arising in Proposition \ref{prop:MK}
can often be most conveniently expressed in
terms of a number \(\gamma_\AddChar(X, Y)\).
We reproduce its definition, adapted to our current
situation using Definition \ref{defn:lots}
(and Lemma \ref{lem:SSPhi}).
Note that
\(\ttt^{\theta, \eta} = \ttt^{\eta^2\theta, 1}\),
and
\(\Im_{\eta^2\theta}(Y)
= \eta\inv\Im_\theta(Y)\).

\begin{defn}[\cite{spice:sl2-mu-hat}*{%
	Definition \xref{sl2-defn:Wald-i}%
}]
\label{defn:Wald-i}
Recall that \(X_\pi \in \ttt^\theta\).
\[
\gamma_\AddChar\bigl(X_\pi, Y\bigr)
\ldef \begin{cases}
H(\SSAddChar, \field_\theta)\sgn_\theta\bigl(
	\eta\inv\Im_\theta(Y)
\bigr), & Y \in \ttt^{\theta, \eta} \\
1,      & Y \in \mf a \\
0,      & \text{otherwise.}
\end{cases}
\]
\end{defn}

The values near the identity of any smooth, irreducible
character of \(G\) can be described in terms of a linear
combination of \(5\) functions (independent of the
representation), namely, the Fourier transforms of nilpotent
orbital integrals on \(G\) (see \cite{hc:queens}).
We will not write our character formulas in this form; but
we do find it convenient to isolate
a particular coefficient in this combination.

\begin{defn}
\label{defn:const}
The \term{constant term} \indexmem{c_0(\pi)}
of \(\pi\) is defined as follows.
In the unramified case,
\[
c_0(\pi) = -q^r.
\]
In the ramified case,
\[
c_0(\pi) = -\frac1 2 q^h\frac{q + 1}q.
\]
\end{defn}

Proposition \ref{prop:MK} and Lemma \ref{lem:deg-pi} below,
together with
\cite{spice:sl2-mu-hat}*{Definition \xref{sl2-defn:const}},
show that this is, indeed,
the coefficient for the Fourier transform of the  trivial
orbit.
By Theorem \ref{thm:exc-near}, we also have
\(c_0(\pi') = -1/2\) for any `exceptional' supercuspidal
representation \(\pi'\).

\subsubsection{The bad shell}

The most challenging range in which to understand the
character is that where the depth of the elements we
consider (modulo centre) is
the same as the depth of our representation.
This range is colloquially known as the `bad shell', for
precisely this reason.
Actually, it turns out that the \emph{unramified} character
formulas in this range are the same as those far from the
identity (see
Theorems \ref{thm:far} and \ref{thm:pos-bad-un});
the complication is only apparent in the ramified case.

\begin{thm}
\label{thm:pos-bad-un}
Suppose that
\(r > 0\), and
\(\gamma \in G_r \setminus Z(G)G_{\MPlus r}\).
In the unramified case, \(\Theta_\pi(\gamma) = 0\) unless
some \(G\)-conjugate of \(\gamma\) lies in \(T^\epsilon\).
If \(\gamma \in T^\epsilon\), then
\[
\Theta_\pi(\gamma)
= \tfrac1 2\sgn_\epsilon\bigl(\Im_\epsilon(\gamma)\bigr)
\frac{
	\psi(\gamma) + \psi(\gamma\inv)
}{
	\smabs{D_G(\gamma)}^{1/2}
}
\bigl[(-1)^{r + 1} + H(\SSAddChar, \field_\epsilon)\bigr].
\]
\end{thm}

\begin{proof}
By Definition \ref{defn:lots}, \(\theta = \epsilon\);
in particular, \(X \in \ttt^\epsilon\).
Put \(Y = \cayley\inv(\gamma)\).

The vanishing result follows from
Proposition \ref{prop:MK},
Lemmas \ref{lem:deg-pi}
and \ref{lem:cayley}\pref{item:cayley-torus},
and
\cite{spice:sl2-mu-hat}*{%
	Theorem \xref{sl2-thm:vanish-spun}%
}.

If \(\gamma \in T^\epsilon\)
(and \(\gamma \in G_r \setminus Z(G)G_r\)), then
Proposition \ref{prop:MK},
Lemmas \ref{lem:deg-pi}
and \ref{lem:cayley}\pref{item:cayley-depth-disc},
\cite{spice:sl2-mu-hat}*{%
	Theorem \xref{sl2-thm:shallow-spun}%
},
Lemma \ref{lem:Im-facts},
and Definition \ref{defn:Wald-i} give that
\[
\Theta_\pi(\gamma)
= \abs{D_G(\gamma)}^{-1/2}H(\SSAddChar, \field_\epsilon)
\sgn_\epsilon\bigl(\Im_\epsilon(\gamma)\bigr)
\bigl(
	\AddChar(\Tr(X\dotm Y)) +
	\AddChar(\Tr(-X\dotm Y))
\bigr).
\]
By Definition \ref{defn:pos-depth-param} and
Lemma \ref{lem:cayley}\pref{item:cayley-inv},
\[
\AddChar\bigl(\Tr(X\dotm Y)\bigr) = \psi(\gamma)
\qandq
\AddChar\bigl(\Tr(-X\dotm Y)\bigr)
	= \AddChar\bigl(\Tr(X\dotm-Y)\bigr)
	= \psi(\gamma\inv).\qedhere
\]
\end{proof}

\begin{thm}
\label{thm:pos-bad-ram}
Suppose that
\(r > 0\), and
\(\gamma \in G_r \setminus Z(G)G_{\MPlus r}\).
In the ramified case, \(\Theta_\pi(\gamma) = 0\) unless some
\(G\)-conjugate of \(\gamma\) lies in \(T^{\theta', \eta}\),
with \(\theta' \in \sset{\varpi, \epsilon\varpi}\)
and \(\eta \in \sset{1, \epsilon}\).
If \(\theta' = \varpi\), then
\begin{align*}
\Theta_\pi(\gamma)
={} & \frac{q^{-1/2}}{2\smabs{D_G(\gamma)}^{1/2}}
	\sum_{\substack{
		\gamma' \in (C_\varpi)_{r:\MPlus r} \\
		\gamma' \ne \gamma^{\pm1}
	}}
		\sgn_\varpi\bigl(
			\Tr_\varpi(\gamma - \gamma')
		\bigr)\psi(\gamma') +{} \\
& \qquad\frac1 2 H(\SSAddChar, \field_\varpi)
\sgn_\varpi\bigl(\eta\inv\Im_\varpi(\gamma)\bigr)
\frac{
	\psi(\gamma) + \psi(\gamma\inv)
}{
	\smabs{D_G(\gamma)}^{1/2}
}.
\end{align*}
If \(\theta' = \epsilon\varpi\), then
\[
\Theta_\pi(\gamma)
= \frac{q^{-1/2}}{2\smabs{D_G(\gamma)}^{1/2}}
\sum_{\gamma' \in (C_\varpi)_{r:\MPlus r}}
	\sgn_\varpi\bigl(
		\Tr_{\epsilon\varpi}(\gamma) - \Tr_\varpi(\gamma')
	\bigr)\psi(\gamma').
\]
\end{thm}

In the first formula,
we are regarding \(\gamma\) as an element of
\(C_\varpi\), not of \(T^{\varpi, \eta}\).
\textit{Via} the isomorphism \(T^\varpi \iso C_\varpi\),
we can then make sense of \(\psi(\gamma)\) and
\(\psi(\gamma')\);
and it makes sense to consider the inequality
\(\gamma' \ne \gamma^{\pm1}\), even though \(\gamma\) and
\(\gamma'\) may lie in different tori.

\begin{proof}
By Definition \ref{defn:lots}, \(\theta = \varpi\);
in particular, \(X \in \ttt^\varpi\).
Put \(Y = \cayley\inv(\gamma)\).

The vanishing result is trivial:
since \(r \not\in \Z\), no element of an unramified
or split torus can have depth \(r\);
i.e., all elements of depth \(r\) already lie in some
\(G\)-conjugate of \(T^{\theta', \eta}\), with
\(\theta'\) and \(\eta\) as above.

If \(\gamma \in T^{\varpi, \eta}\)
(and \(\gamma \in G_r \setminus Z(G)G_{\MPlus r}\)), with
\(\eta \in \sset{1, \epsilon}\), then write
\(Y = X^{\varpi, \eta}_c\), and note that
\(\widetilde Y \ldef X^{\varpi, 1}_c
= \Ad\begin{smallpmatrix}
\sqrt\eta & 0             \\
0         & \sqrt\eta\inv
\end{smallpmatrix}Y\) is a stable conjugate of \(Y\) that
lies in \(\ttt^\varpi\).
By Proposition \ref{prop:MK},
Lemmas \ref{lem:deg-pi} and
\ref{lem:cayley}(%
	\ref{item:cayley-torus},
	\ref{item:cayley-depth-disc}%
),
\cite{spice:sl2-mu-hat}*{%
	Theorem \xref{sl2-thm:that-bad-ram}%
}, Lemma \ref{lem:Im-facts},
and Definition \ref{defn:Wald-i},
\begin{align*}
\Theta_\pi(\gamma)
={} & \frac1 2 H(\SSAddChar, \field_\theta)
\sgn_\varpi\bigl(
	\eta\inv\Im_\varpi(\gamma)
\bigr)
\frac{
	\AddChar(\Tr(X\dotm\widetilde Y)) +
	\AddChar(\Tr(-X\dotm\widetilde Y))
}{
	\smabs{D_G(\gamma)}^{1/2}
} +{} \\
&\qquad\frac{q^{-1/2}}{2\smabs{D_G(\gamma)}^{1/2}}
\sum_{Z \in \ttt^\varpi_{r:\MPlus r}}
	\sgn_\varpi\bigl(
		Y^2 - Z^2
	\bigr)\AddChar(\Tr(X\dotm Z)).
\end{align*}
By Definition \ref{defn:pos-depth-param},
\[
\AddChar\bigl(\Tr(X\dotm Z)\bigr)
= \psi\bigl(\cayley(Z)\bigr)
\quad\text{for all \(Z \in \ttt^\varpi_r\).}
\]
Further, we have
\[
\AddChar\bigl(\Tr(X\dotm\widetilde Y)\bigr)
= \psi(\tilde\gamma),
\]
where \(\tilde\gamma = \Int\begin{smallpmatrix}
\sqrt\eta & 0             \\
0         & \sqrt\eta\inv
\end{smallpmatrix}\gamma = \cayley(\widetilde Y)\);
but note that
\(\gamma \in T^{\theta'}\) and \(\tilde\gamma \in T^\varpi\)
correspond to the same element of \(C_\varpi\),
so our notational conventions allow us to write
\(\psi(\gamma)\) instead of \(\psi(\tilde\gamma)\).
Similarly,
\(\AddChar\bigl(\Tr(-X\dotm\widetilde Y)\bigr)
= \psi(\gamma\inv)\).

Finally, note that, by Lemma \ref{lem:norm-and-tr},
since \(Y\) and \(Z\)
(regarded as elements of \(V_{\theta'} = V_\varpi\))
lie in \(\pp_\varpi^{2h - 1}\)
(where \(h\) is as in Definition \ref{defn:lots}), we have
the additive congruence
\[
Y^2 - Z^2
\equiv \Tr_\varpi\bigl(\gamma - \cayley(Z)\bigr)
\pmod{\pp^{2h}}.
\]
Since
\(Y^2, Z^2 \in \pp^{2h - 1}\)
and
\(Z \not\equiv Y \pmod{\pp_\varpi^{2h}}\),
we have that \(\ord(Y^2 - Z^2) = 2h - 1\).
Thus we can deduce the multiplicative congruence
\[
Y^2 - Z^2
\equiv \Tr_\varpi\bigl(\gamma - \cayley(Z)\bigr)
\pmod{1 + \pp},
\]
hence the equality
\[
\sgn_\varpi(Y^2 - Z^2)
= \sgn_\varpi\bigl(\Tr_\varpi(
	\gamma - \cayley(Z)
)\bigr).
\]
The formula now follows (in this case) from
Lemma \ref{lem:cayley}(%
	\ref{item:bijpos},
	\ref{item:cayley-torus}%
) upon putting \(\gamma' = \cayley(Z)\).

The argument in case \(\theta' = \epsilon\varpi\)
is similar but easier.
\end{proof}

\subsubsection{Character values very near the identity}

Finally, we consider character values \emph{very} near
the identity,
so that we are within the range of the local character
expansion.
The Hales--Moy--Prasad conjecture,
proven in \cite{debacker:homogeneity}*{Theorem 3.5.2} under
mild hypotheses on \(p\),
describes the precise range of validity for the local
character expansion
for any smooth, irreducible representation of
a reductive, \(p\)-adic group;
but we shall not need the general result here.
For our case (\(\bG = \SL_2\)), it can be verified by
direct computation from our formulas
that the local character expansion holds on \(G_{\MPlus r}\)
(see \cite{debacker-sally:germs}*{Appendix A}).

\begin{thm}
\label{thm:near}
Suppose \(\gamma \in G_{\MPlus r} \cap G\rss\)
and, if \(r = 0\), that
Hypothesis \ref{hyp:debacker-reeder:depth-zero-sc:res-12.4.1(2)}
holds.
Then \(\Theta_\pi(\gamma) = c_0(\pi)\) unless some \(G\)-conjugate
of \(\gamma\) lies in \(A\) or \(T^{\theta, \eta}\)
for some \(\eta\).
If \(\gamma \in A\), then
\begin{align*}
\Theta_\pi(\gamma)
= c_0(\pi) +{} &
\frac1{\abs{D_G(\gamma)}^{1/2}}. \\
\intertext{If \(\gamma \in T^{\theta, \eta}\), then}
\Theta_\pi(\gamma)
= c_0(\pi) +{} &
H(\SSAddChar, \field_\theta)\frac{
	\sgn_\theta\bigl(\eta\inv\Im_\theta(\gamma)\bigr)
}{
	\abs{D_G(\gamma)}^{1/2}
}.
\end{align*}
\end{thm}

Theorem \ref{thm:near} remains true without
the extra hypothesis.
A proof in that generality
will appear in~\cite{adler-debacker-roche-sally-spice:sl2}.

\begin{proof}
This is a combination of
Proposition \ref{prop:MK},
Lemma \ref{lem:deg-pi},
\cite{spice:sl2-mu-hat}*{%
	Theorems \xref{sl2-thm:close-spun}
	and \xref{sl2-thm:close-ram}%
},
Lemmas \ref{lem:cayley}(%
	\ref{item:cayley-torus},
	\ref{item:cayley-depth-disc}%
)
and
Definition \ref{defn:Wald-i} and Lemma \ref{lem:SSPhi}.
\end{proof}

\section{`Exceptional' supercuspidal characters}
\label{sec:exceptional}

By Remark \ref{rem:depth-0-leftist},
the only representations we still need to consider
after \S\ref{sec:ordinary} are
\(\pi^\pm \ldef \bPi^\pm(T^{\epsilon, 1}, \psi_0^1)\).
(Recall that the character $\psi_0^1$ is defined in
Notation~\ref{notn:torus-quad}, and the associated
representation in Definition~\ref{defn:depth-0-param}.)

\subsection{Character values far from the identity}

\begin{thm}
\label{thm:exc-far}
If \(\gamma \not\in Z(G)G_{\MPlus0}\), then
\(\Theta_\pi(\gamma) = 0\) unless some \(G\)-conjugate of
\(\gamma\) lies in \(T^\epsilon\).
If \(\gamma \in T^\epsilon\), then
\[
\Theta_{\pi^\pm}(\gamma)
= \frac{\sgn_\varpi(\gamma + \gamma\inv + 2)}2
\Bigl\{
	H(\SSAddChar, \field_\epsilon)\frac{
		\sgn_\epsilon\bigl(\Im_\epsilon(\gamma)\bigr)
	}{
		\smabs{D_G(\gamma)}^{1/2}
	} - 1
\Bigr\}.
\]
\end{thm}

Note that the character values of \(\pi^+\) and \(\pi^-\) in
this range are the same.

\begin{proof}
The proof is almost exactly as in
the depth-zero case of Theorem \ref{thm:near}
and
\cite{debacker-reeder:depth-zero-sc}*{\S\S9--10}.
In particular, we may assume that
\(\gamma \in T^\epsilon \subseteq G_{x, 0}\).

We need only make some minor adjustments to account
for the fact that \(\psi_0\) is not `regular', in the sense
of \S9.3 \loccit; i.e., that
\(\psi_0 \circ \Int(\sigma_\epsilon)
= \psi_0\inv = \psi_0\),
where \(\sigma_\epsilon\) is the non-trivial element of the Weyl group
of \(T^\epsilon\).

Under our hypotheses on \(\gamma\), its image \(\bar\gamma\)
in \(G_{x, 0:\MPlus0} = \SL_2(\resfld)\) is a regular,
semisimple element,
so Definitions \ref{defn:exc-rep-finite}
and \ref{defn:lusztig} give
\[
R^\pm_{\sfT^\epsilon, \psi_0}(\bar\gamma)
= \tfrac1 2 R^\sfG_{\sfT^\epsilon, \psi_0}(\bar\gamma).
\]
Therefore, as in the proof of
\cite{debacker-reeder:depth-zero-sc}*{Lemma 9.3.1},
using the Harish-Chandra integral formula
(\S9.1 \loccit, or \cite{hc:harmonic}*{Theorem 12})
gives
\begin{equation}
\tag{$*$}
\label{eq:debacker-reeder:depth-zero-sc:Lemma-9.3.1-irregular}
\Theta_\pi(\gamma)
= \tfrac1 2\varepsilon(\sG_x, \sfT^\epsilon)
	R(G, T^\epsilon, \psi_0)(\gamma)
= -\tfrac1 2 R(G, T^\epsilon, \psi_0)(\gamma).
\end{equation}

As in Theorem \ref{thm:far},
\begin{align*}
\widehat{\mc T}(\gamma\tsemi)
& {}= \sett{(T^\epsilon, \psi')}{%
	\(\psi' = \psi_0 \circ \Int(n)\inv\)
	for some \(n \in N_G(T^\epsilon)\)%
} \\
& {}= \sset{(T^\epsilon, \psi_0), (T^\epsilon, \psi_0\inv)} \\
& {}= \sset{(T^\epsilon, \psi_0)}.
\end{align*}
In our setting, however, the map
\(\abmapto{(d, \bar n)}{(n d)\inv\dotm(T^\epsilon, \psi_0)}\)
of \cite{debacker-reeder:depth-zero-sc}*{p.~857} is a double
cover, not a bijection; so the formula in
Lemma 10.0.4 \loccit becomes
\begin{equation}
\tag{$**$}
\label{eq:debacker-reeder:depth-zero-sc:Lemma-10.0.4-irregular}
\begin{aligned}
R(G, T^\epsilon, \psi_0)(\gamma)
& {}= 2\sum_{(T', \psi')}
	\psi'(\gamma\tsemi)
	R(G_{\gamma\tsemi}, T', 1)(\gamma\tunip) \\
& {}= 2\psi_0(\gamma\tsemi)R(T^\epsilon, T^\epsilon, 1)(\gamma\tunip) \\
& {}= 2\psi_0(\gamma),
\end{aligned}
\end{equation}
where the sum again runs over the set of orbits in
\(\widehat{\mc T}(\gamma\tsemi)\) under the natural
(trivial) action of \(T^\epsilon\).

Again as in Theorem \ref{thm:far},
\[
\frac{
	\sgn_\epsilon\bigl(\Im_\epsilon(\gamma)\bigr)
}{
	\smabs{D_G(\gamma)}^{1/2}
} = 1,
\]
so the result now follows by combining
\eqref{eq:debacker-reeder:depth-zero-sc:Lemma-9.3.1-irregular}
with
\eqref{eq:debacker-reeder:depth-zero-sc:Lemma-10.0.4-irregular},
and using
Lemmas \ref{lem:quad-char}
and \ref{lem:SSPhi}.
\end{proof}

\subsection{Character values near the identity}

\begin{thm}
\label{thm:exc-near}
Suppose that \(\gamma \in G\rss \cap G_{\MPlus0}\).
If \(\gamma \in A\), then
\[
\Theta_{\pi^\pm}(\gamma)
= \frac1 2\Bigl\{
	\frac1{\smabs{D_G(\gamma)}^{1/2}} - 1
\Bigr\}.
\]
If \(\gamma \in T^{\theta', \eta}\),
where
\begin{itemize}
\item \(\theta' = \epsilon\) and \(\eta \in \sset{1, \varpi}\)
or
\item \(\theta' \in \sset{\varpi, \epsilon\varpi}\)
and \(\eta \in \sset{1, \epsilon}\),
\end{itemize}
then
\[
\Theta_{\pi^\pm}(\gamma)
= \frac1 2\Bigl\{
	\pm H(\SSAddChar, \field_{\theta'})\frac{
		\sgn_{\theta'}\bigl(
			\eta\inv\Im_{\theta'}(\gamma)
		\bigr)
	}{
		\smabs{D_G(\gamma)}^{1/2}
	} -
	1
\Bigr\}.
\]
\end{thm}

\begin{proof}
We use Proposition \ref{prop:MK-exc},
which writes \(\Theta_\pi \circ \cayley\)
in the indicated range as a linear combination of Fourier
transforms of orbital integrals.

In order to compute this combination of Fourier transforms of
orbital integrals, we adopt the notation of
\cite{spice:sl2-mu-hat}*{%
	Notation \xref{sl2-notn:Bessel-abbrev}%
}, so that
\begin{equation}
\tag{$\ddag$}
\label{eq:spice:sl2-mu-hat:notn:Bessel-abbrev}
\bigl[
	A; B_1, B_\epsilon, B_\varpi, B_{\epsilon\varpi}
\bigr]_{\theta, r'}
\end{equation}
stands for the function whose value at an element
\(Y\) of an elliptic Cartan subalgebra
\(\ttt^{\theta', \eta} \iso V_{\theta'}\) is
\[
\abs\theta^{1/2}A +
q^{-(r' + 1)}\abs{D_\gg(Y)}^{-1/2}B_{\theta'}\bigl(
	\eta\inv\Im_{\theta'}(Y)
\bigr),
\]
and whose value at an element
\(Y\) of the split Cartan subalgebra
\(\mf a\) is
\[
\abs\theta^{1/2}A +
q^{-(r' + 1)}\abs{D_\gg(Y)}^{-1/2}B_1.
\]

Then Theorems \xref{sl2-thm:close-spun} and
\xref{sl2-thm:close-ram} \loccit,
combined with Definition \ref{defn:Wald-i},
give
\[
\begin{array}{r@{}r@{}l@{}l@{}l@{}l@{}r@{}l}
\hat\mu^G_{X^{\epsilon, 1}_1}
={} & \bigl[&
	-q\inv; {}
	& 1, {} & H(\SSAddChar, \field_\epsilon)\sgn_\epsilon, {}
	& 0, {} & 0
&\bigr]_{\epsilon, 0} \\
\hat\mu^G_{X^{\varpi, 1}_1}
={} & \bigl[&
	-\tfrac1 2 q^{-3/2}(q + 1); {}
	& 1, {} & 0, {}
	& H(\SSAddChar, \field_\varpi)\sgn_\varpi, {} & 0
&\bigr]_{\varpi, 0};
\end{array}
\]
but it is important to realize that there are two obstacles
to combining the formulas.
First, the subscripts are different
(\((\epsilon, 0)\) versus \((\varpi, 0)\));
and, second, the measures with respect to which the orbital
integrals are computed are not those in
Lemmas \ref{lem:MK-green} and \ref{lem:MK-lusztig}.
(It may seem that a third obstacle is the fact that
Definition \ref{defn:Wald-i} is stated only in the
positive-depth setting; but, since we are working in the Lie
algebra, there is no harm now in multiplying by a scalar to
see that, in fact, it remains valid in the depth-zero case.)
Our approach to the first problem will be to replace all
subscripts with the arbitrarily chosen \((\epsilon, -1)\)
(which we then drop from the notation).
For the second, we recall that the quoted orbital integrals
use the various measures \(\textup d_{\theta'}\dot g\)
on \(G/T^{\theta'}\),
and so use \S\ref{sec:measure} to replace them by the
measures \(\textup dg/\textup dt^{\theta'}\).
Making both of these adjustments gives
\begin{equation}
\tag{$*$}
\label{eq:epsilon-and-varpi}
\begin{array}{r@{}r@{}l@{}l@{}l@{}l@{}r@{}l}
\hat\mu^G_{X^{\epsilon, 1}_1}
={} & \bigl[&
	-1; {}
	& 1, {} & H(\SSAddChar, \field_\epsilon)\sgn_\epsilon, {}
	& 0, {} & 0
&\bigr] \\
\hat\mu^G_{X^{\varpi, 1}_1}
={} & q^{1/2}\bigl[&
	-\tfrac1 2 q\inv(q + 1); {}
	& 1, {} & 0, {}
	& H(\SSAddChar, \field_\varpi)\sgn_\varpi, {} & 0
&\bigr].
\end{array}
\end{equation}
It is important to note that the difference between this
equation and the previous one is \emph{not} just notational;
since we have changed normalizations of measures, we are
actually describing different \emph{functions}.

We have that
\begin{align*}
\Ad\begin{pmatrix}
\epsilon & 0 \\
0        & 1
\end{pmatrix}X^{\varpi, \epsilon}_1 = X^{\varpi, 1}_1, \\
\intertext{so}
\hat\mu^G_{X^{\varpi, \epsilon}_1}
= \hat\mu^G_{X^{\varpi, 1}_1} \circ \Ad\begin{pmatrix}
\epsilon & 0 \\
0        & 1
\end{pmatrix}.
\end{align*}
A direct computation shows that this reduces to
\begin{equation}
\tag{$**$}
\label{eq:other-varpi}
\hat\mu^G_{X^{\varpi, \epsilon}_1}
= q^{1/2}\bigl[
	-\tfrac1 2 q\inv(q + 1);
	1, 0,
	-H(\SSAddChar, \field_\varpi)\sgn_\varpi, 0
\bigr];
\end{equation}
i.e., all that has changed is that
\(H(\SSAddChar, \field_\varpi)\) has become
\(-H(\SSAddChar, \field_\varpi)\).

Further, as observed in
\cite{spice:sl2-mu-hat}*{Remark \xref{sl2-rem:what-about}},
we may adapt formulas involving one choice of uniformizer
(such as \(\varpi\)) to another choice
(such as \(\epsilon\varpi\)) by simple substitution; so
(remembering that the order of the arguments is significant)
we find
\begin{equation}
\tag{$*{*}*$}
\label{eq:epsilon-varpi}
\begin{array}{r@{}r@{}l@{}l@{}l@{}l@{}r@{}l}
\hat\mu^G_{X^{\epsilon\varpi, 1}_1}
={} & q^{1/2}\bigl[&
	-\tfrac1 2 q\inv(q + 1); {}
	& 1, {} & 0, {}
	& 0, {} &
		H(\SSAddChar, \field_{\epsilon\varpi})\sgn_{\epsilon\varpi}
&\bigr]; \\
\hat\mu^G_{X^{\epsilon\varpi, \epsilon}_1}
={} & q^{1/2}\bigl[&
	-\tfrac1 2 q\inv(q + 1); {}
	& 1, {} & 0, {}
	& 0, {} &
		-H(\SSAddChar, \field_{\epsilon\varpi})\sgn_{\epsilon\varpi}
&\bigr].
\end{array}
\end{equation}
By \eqref{eq:epsilon-and-varpi}, \eqref{eq:other-varpi},
and \eqref{eq:epsilon-varpi},
\begin{multline*}
\hat\mu^G_{X^{\varpi, 1}_1} -
	\hat\mu^G_{X^{\varpi, \epsilon, 1}_1} +
\hat\mu^G_{X^{\epsilon\varpi, 1}_1} -
	\hat\mu^G_{X^{\epsilon\varpi, \epsilon}_1} \\
= 2q^{1/2}\bigl[
	0;
	0, 0,
	H(\SSAddChar, \field_\varpi)\sgn_\varpi,
		H(\SSAddChar, \field_{\epsilon\varpi})\sgn_{\epsilon\varpi}
\bigr].
\end{multline*}
By
\eqref{eq:spice:sl2-mu-hat:notn:Bessel-abbrev},
and
Lemmas \ref{lem:Im-facts}
and
\ref{lem:cayley}\pref{item:cayley-depth-disc},
Proposition \ref{prop:MK-exc}
now simplifies to the desired formula.
\end{proof}

\end{document}